%% file: arxiv_v2.tex
\documentclass[10pt,reqno]{amsart}  
\input{packages2.tex}
\input{notations2.tex}
\input{bibliography_setup.tex}
\newlist{genprop}{enumerate}{1}
\bibliography{bibclt}

\calclayout

\numberwithin{equation}{section} 
\newtheorem{theorem}{Theorem}[section]
\newtheorem{assumption}{Assumption}
\newtheorem{lemma}[theorem]{Lemma}
\newtheorem{proposition}[theorem]{Proposition}
\newtheorem{remark}[theorem]{Remark}

\newcommand{\nc}{\normalcolor}

\date{\today}
\author{Giorgio Cipolloni}
\address{Princeton Center for Theoretical Science, Princeton University, Princeton, NJ 08544, USA}
\author{L\'aszl\'o Erd\H{o}s\(^\#\)}
\address{IST Austria, Am Campus 1, 3400 Klosterneuburg, Austria}
\author{Dominik Schr\"oder\(^{\ast}\)}
\address{ETH Zurich, R\"amistrasse 101, 8092 Zurich, Switzerland}
\email{gc4233@princeton.edu} 
\email{lerdos@ist.ac.at}
\email{dschroeder@ethz.ch}
\thanks{\(^\#\)Supported by ERC Advanced Grant ``RMTBeyond'' No.~101020331}
\thanks{\(^\ast\)Supported by the SNSF Ambizione Grant \texttt{PZ00P2 209089}}
\subjclass[2010]{60B20, 15B52} 
\keywords{Dyson Brownian Motion, Local Law, Girko’s Formula, Linear Statistics, Central Limit Theorem}
\title{MESOSCOPIC CENTRAL LIMIT THEOREM FOR NON-HERMITIAN RANDOM MATRICES}
\date{\today}

\begin{document}

\begin{abstract}
    We prove that the mesoscopic linear statistics \(\sum_i f(n^a(\sigma_i-z_0))\) of the eigenvalues \(\{\sigma_i\}_i\) of large $n\times n$ non-Hermitian random matrices with complex centred i.i.d.\ entries are asymptotically Gaussian for any $H^{2}_0$-functions \(f\) around any point $z_0$ in the bulk of the spectrum on any mesoscopic scale $0<a<1/2$. This extends our previous result \cite{1912.04100}, that was
    valid on the macroscopic scale, $a=0$, to cover the entire mesoscopic regime.
    The main novelty is a  \emph{local law} for the product of resolvents for
    the Hermitization of $X$ at spectral parameters $z_1, z_2$ with an improved error
    term in the entire mesoscopic regime $|z_1-z_2|\gg n^{-1/2}$.
    The proof  is dynamical; it relies on a recursive tandem of  the characteristic flow method
    and the Green function comparison idea combined with a separation
    of the unstable mode of the underlying stability operator.
\end{abstract}

\maketitle

\section{Introduction}

We consider the eigenvalues $\{ \sigma_i\}_{i=1}^n$ of an $n\times n$ random matrix $X$
with i.i.d.\ entries under the standard normalisation condition $\E x_{ij} =0$, $\E |x_{ij}|^2 =\frac{1}{n}$.
The classical \emph{circular law} \cite{MR773436, MR1428519, MR2409368}
asserts that the empirical eigenvalue density converges to the uniform distribution  on the unit disk $\DD$:
\begin{equation}\label{circ}
    \frac{1}{n}\sum_{i=1}^n f(\sigma_i) \to \frac{1}{\pi}\int_{\DD}  f(z)\dif^2 z, \quad n\to \infty,
\end{equation}
for any continuous bounded test function $f$.
In fact this limit also holds on any mesoscopic scale by the \emph{local circular law}
\cite{MR3230002}, i.e.
\begin{equation}\label{localcirc}
    \frac{n^{2a}}{n}\sum_{i=1}^n f\big(n^a (\sigma_i-z_0)\big) \to \frac{1}{\pi}\int_{\C}  f(z)\dif^2 z, \qquad  0<a<\frac{1}{2},
\end{equation}
where the compactly supported $C^2$ test function is scaled to concentrate on an $n^{-a}$-neighborhood
of any fixed point in the bulk spectrum $|z_0|<1$. The threshold $a<1/2$ is sharp since on scales of order $n^{-1/2}$
there are only  finitely many fluctuating eigenvalues hence a law of large number type concentration
cannot hold.

In this paper we prove a central limit theorem (CLT) for the fluctuation around the
local circular law~\eqref{localcirc} for  complex i.i.d. matrices,   i.e.
we show that
\begin{equation}\label{mesoclt}
    \sum_{i=1}^n f\big(n^a (\sigma_i-z_0)\big) -  \frac{n}{n^{2a}}\frac{1}{\pi}  \int_{\C}  f(z)\dif^2 z -\frac{1}{8\pi}\int_\mathbf{C}\Delta f(z)\,\dif^2z\Longrightarrow
    L(f)\sim \mathcal N_\C, \qquad  0<a<\frac{1}{2},
\end{equation}
and compute the variance $\E\abs{L(f)}^2=\norm{\nabla f}^2/4\pi$ of the limiting normal distribution.
Note the unusual normalisation:
the sum in~\eqref{mesoclt} contains roughly  $n^{1-2a}$ terms, but it is not divided  by $\sqrt{n^{1-2a}}$
unlike for the standard CLT for sums of  independent random variables. The eigenvalues $\sigma_i$ are
strongly correlated, and their fluctuations are much smaller than that of an independent
point process (e.g. Poisson). It is very remarkable that nevertheless the normal distribution emerges;
in fact these eigenvalues asymptotically follow a \emph{Gaussian Free Field (GFF)}, a logarithmically
correlated Gaussian process.

The CLT on the macroscopic scale, $a=0$, around the circular law~\eqref{circ},
\begin{equation}\label{macroclt}
    \sum_{i=1}^n f(\sigma_i) - \frac{n}{\pi}\int_{\DD}  f(z)\dif^2 z +
    \frac{\kappa_4}{\pi}\int_\DD f(z)(2\abs{z}^2-1)\dif^2 z-\frac{1}{8\pi}\int_\mathbf{C}\Delta f(z)\,\dif^2z\Longrightarrow  L(f)\sim\mathcal N_\C,
\end{equation}
has been proven earlier with a long history. Here \(\kappa_4:=n^2 \bigl[\E \abs{x_{ij}}^4-2(\E \abs{x_{ij}}^2)^2-\abs{\E x_{ij}^2}^2\bigr]\) denotes the normalised joint cumulant of \(x_{ij},x_{ij},\ov{x_{ij}},\ov{x_{ij}}\). Historically
the first results were for the \emph{complex Ginibre ensemble}, i.e. when $x_{ij}$ are complex Gaussians
(and therefore \(\kappa_4=0\));
in this case an explicit formula for the joint density function of all eigenvalues is available.
Forrester in~\cite{MR1687948}
proved~\eqref{macroclt} for  radially symmetric \(f\), he found
\(\E\abs{L(f)}^2 = (4\pi)^{-1}\int_{\DD} \abs{\nabla f}^2 \dif^2z\)
and gave a heuristic prediction that the variance $\E\abs{L(f)}^2$ contains
an additional boundary term %
for general  $f$.
Rider and Vir\'ag in~\cite{MR2361453} have rigorously verified Forrester's prediction for any \(f\in C^1(\DD)\)
and they also presented a GFF interpretation of the result. %
Rider in~\cite{MR2095933} also considered special indicator test functions depending
only on the angle or on the modulus that are not in $H^1(\DD)$ even in the mesoscopic regime
with $\E\abs{L(f)}^2$ growing as $\log n$.

Beyond the explicitly computable  complex Ginibre case the first result was obtained by
Rider and Silverstein~\cite[Theorem 1.1]{MR2294978}. They proved~\eqref{macroclt}
for \(X\) with general i.i.d.\ complex matrix elements
but only  for test functions \(f\) that are analytic on an unnaturally  large disk of radius 2.  In the real symmetry class
the domain of analyticity was optimized in~\cite{MR3540493} and the result was also extended to
\emph{elliptic ensembles} allowing correlation between $x_{ij}$ and $x_{ji}$.
Later even \emph{products} of i.i.d.\ matrices were considered in~\cite{MR4125967}.
Alternatively, the moment method was used  in~\cite{MR2738319}
to prove CLT with polynomial test functions $f$.
Beyond analytic test functions, Nguyen and Vu in \cite{MR3161483} proved a CLT for $f=\log$ and $X$ with i.i.d. entries;
Tao and Vu in~\cite[Corollary 10]{MR3306005} proved CLT for the counting function on balls
even on mesoscopic scales assuming
the first four moments of \(X\) match those of the complex Ginibre ensemble. The
comparison method from Tao and Vu
was extended by Kopel~\cite[Corollary 1]{1510.02987} to general smooth test functions and also to real $X$
with an additional study on the real eigenvalues  (see also~\cite{MR3612267}).
Finally, the macroscopic CLT~\eqref{macroclt} in full generality,
i.e. for matrices with general i.i.d. entry distribution and
general smooth (in fact $H^{2+\delta}$) test functions, has been proven by us in~\cite{1912.04100}
and in~\cite{MR4235475} for the complex and real cases, respectively.
For more details on the history of the circular laws~\eqref{circ}-\eqref{localcirc} and the macroscopic CLT~\eqref{macroclt},
as well as several  further references, see~\cite{1912.04100, MR4235475}.

Apart from~\cite{MR2095933} that holds only  for complex Ginibre with special test functions and
apart from~\cite[Corollary 10]{MR3306005} that assumes four matching moments
and test  functions being the indicator functions of  mesoscopic balls, all previous CLT results
were on macroscopic scales. The extension to mesoscopic scales~\eqref{mesoclt}
is clearly not feasible with moment method or with methods relying on the analyticity of the test function.
Direct calculations based  upon the explicit Ginibre formulas
beyond the special test functions in~\cite{MR2095933, MR3306005}
may be possible also for mesoscopic scales
but their extension to the general i.i.d. case would again require four moment matching.

In this paper we demonstrate that our approach in~\cite{1912.04100, MR4235475}
can be extended to cover the entire mesoscopic regime. For simplicity,
we work with the complex  symmetry class, the real case requires additional
technical steps that we will explain in Remark~\ref{rmk:real} but do not carry out in details.
To highlight the novelty of the current proof, we briefly recall the main ideas in~\cite{1912.04100}.
The starting point is Girko's formula,
\begin{equation}\label{girko}
    \sum_{\sigma\in \Spec(X)} f(\sigma) = -\frac{1}{4\pi} \int_{\C } \Delta f(z)\int_0^\infty \Im \Tr  G^z(\ii\eta)\dif\eta \dif^2 z,
\end{equation}
expressing the linear statistics of the non-Hermitian eigenvalues of $X$
in terms of the resolvent  \(G^z(w):=  (H^z-w)^{-1}\) of the \emph{Hermitisation} $H^z$ of $X$,
\begin{equation}\label{eq:linz1}
    H^z:= \begin{pmatrix}
        0                & X-z \\
        X^*-\overline{z} & 0
    \end{pmatrix}
\end{equation}
parametrized by \(z\in \C  \).  Note that Girko's formula requires to understand the resolvent $G^z(\ii \eta)$ very well
for arbitrary small $\eta$, even for the macroscopic CLT.

The different $\eta$-regimes in~\eqref{girko} require very different methods.
\begin{description}[style=standard]
    \item[Sub-critical \(\eta\ll1/n\)] The absence of eigenvalues of $H^z$ in $[-\eta, \eta]$ is proved using smoothing inequalities for the lower tail of the lowest eigenvalue  \cite{MR2255338, MR2684367, MR4129729} (see also \cite{MR4474380, MR4408004, 2204.06026}). This regime is handled exactly as in the macroscopic case~\cite{1912.04100,MR4235475}.
    \item[Critical \(\eta\sim1/n\)] The asymptotic independence of resolvents for \(\abs{z_1-z_2}\gg 1/\sqrt{n}\) is proved dynamically using the Dyson Brownian motion (DBM) technique.
        Our  new result is the
        asymptotic orthogonality of the low-lying eigenvectors of $H^{z_1}$ and  $H^{z_2}$. Once this key input
        is established, the proof proceeds exactly as in~\cite{MR4235475} following
        the coupling  technique  and the homogenization idea that were first introduced in~\cite{MR3541852}  for  Wigner matrices, substantially generalized later in~\cite{MR3914908}
        for general DBM (see also \cite{MR4416591}),
        and adapted to $H^z$ in~\cite{MR3916329}. The almost orthogonality of eigenvectors implies
        the almost independence of the driving Brownian motions in the DBM for $\lambda^{z_1}$ and $\lambda^{z_2}$.
    \item[Super-critical \(\eta\gg1/n\)] A central limit theorem for resolvents is established using iterated cumulant expansions. The covariance of $\Tr G^{z_1}(\ii\eta_1)$ and $\Tr G^{z_2}(\ii\eta_2)$ for \emph{different} parameters $z_1\ne z_2$ depends critically on the product of resolvents $\Tr G^{z_1}(\ii\eta_1)G^{z_2}(\ii\eta_2)$ which we evaluate to high precision using our new \emph{multi-resolvent local laws}.
\end{description}

We now explain the novelty of the present work for the critical and super-critical regime in form of an improved multi-resolvent local law. The conventional single resolvent local law asserts that $G^z(\ii\eta)$ can be
approximated deterministically by an explicitly computable matrix $M^z(\ii\eta)$ (see~\eqref{Mz})
up to a negligible error as long as $|\eta| \gg 1/n$:
\begin{equation}\label{1G}
    \frac{1}{2n} \Big|\Tr \big[ G^z(\ii\eta)-M^z(\ii\eta)\big] \Big| \lesssim \frac{n^\xi}{n |\eta|}
\end{equation}
holds with very high probability for any fixed $\xi>0$.

While single resolvent local laws are well understood, their multi-resolvent
versions are much more subtle. The naive intuition from~\eqref{1G} would suggest
that $ G^{z_1}(\ii\eta_1)G^{z_2}(\ii\eta_2)\approx M^{z_1}(\ii\eta_1)M^{z_2}(\ii\eta_2) $,
but this is wrong. The correct deterministic approximation
of $ G^{z_1}(\ii\eta_1)G^{z_2}(\ii\eta_2)$  is $M_{12}:={\mathcal{B}}_{12}^{-1}[M^{z_1}(\ii\eta_1)M^{z_2}(\ii\eta_2) ]$,
where ${\mathcal{B}}_{12}$ is the \emph{stability operator}, given explicitly in~\eqref{stabop}.
This operator has a small eigenvalue $\beta$ of order $|z_1-z_2|^2+\eta_1+\eta_2$.
The key question is the error term in this approximation.

Setting $\eta_1=\eta_2=:\eta>0$ for simplicity,
one may guess (and we prove below) the bound
\begin{equation}\label{GGnaive}
    \frac{1}{2n}  \Big|\Tr\big[ G^{z_1}(\ii\eta)G^{z_2}(\ii\eta) - M_{12}\big] \Big| \lesssim \frac{n^\xi}{n\eta^2}.
\end{equation}
If $z_1=z_2$, then  this bound
is essentially optimal and in this case $M_{12}\sim 1/\eta$ since $\beta\sim \eta$.
However, when $f$ in~\eqref{girko} is mesoscopically supported, then typically
we have $|z_1-z_2|\sim n^{-a}\gg n^{-1/2}$ in the  calculation of the variance of~\eqref{girko}.
In this case $M_{12}\sim [ |z_1-z_2|^2+\eta]^{-1}$, i.e. for $\eta$'s such that $1/n\ll \eta\ll |z_1-z_2|^2\sim n^{-2a}$,
the bound on $M_{12}$ is already smaller than $1/\eta$. One therefore expects that the error term~\eqref{GGnaive}
also improves in this regime and indeed we need some improvement to handle the $\eta\gg 1/n$ regime when computing
higher moments of~\eqref{girko}. The main technical result in our proofs of the macroscopic CLT
was~\cite[Theorem 5.2]{1912.04100}, asserting that the error in~\eqref{GGnaive}
improves by a factor $n^{-\epsilon_1}$ if $|z_1-z_2|\ge n^{-\epsilon_2}$
for some small $\epsilon_1, \epsilon_2$. This improvement, however,  does not apply
to genuine mesoscopic scales when $|z_1-z_2|\sim n^{-a} \ll n^{-\epsilon_2}$.
In  Theorem~\ref{theo:impll} of this paper we present a substantial improvement
of~\cite[Theorem 5.2]{1912.04100}, essentially asserting that
the error in~\eqref{GGnaive}
can be improved by a factor $n^{-\epsilon_1}$ as long as $|z_1-z_2|\ge n^{-\frac{1}{2}+\epsilon_2}$,
i.e. the improvement is present in the entire mesoscopic regime.

We stress that~\eqref{GGnaive} even without any additional improvement
is new  in the mesoscopic regime and its proof is highly nontrivial.
Multi-resolvent local laws of the
form~\eqref{GGnaive} for products of resolvents $G(z)=(W-z)^{-1}$
of Wigner matrices $W$ at different spectral parameters
have been proven earlier
(see e.g. \cite[Theorem 3.4]{MR4372147}
and its refinements in~\cite[Proposition 3.4]{MR4334253}, \cite[Proposition 5.1]{2012.13218},
\cite[Theorem 2.5]{MR4479913}, and~\cite[Theorem 2.2]{2203.01861}).
However, these  proofs
rely on the fact that both resolvents stem from the \emph{same} Hermitian matrix
hence they have the same spectral resolution. This allowed us to use resolvent identities
to express products of resolvents in terms of their first powers and thus prevent the
instability of the operator ${\mathcal{B}}_{12}$ from influencing the error terms.
Practically, multi-resolvent local laws were reduced to well established single resolvent local laws
by algebraic identities\footnote{The analysis was more complicated
    when deterministic matrices $A_1, A_2, \ldots$
    were also present, i.e. we considered local laws for $G_1A_1G_2A_2\ldots$. In this case we split $A$'s into a diagonal
    and a traceless part; the former was handled by resolvent identities, while on traceless matrices
    the stability operator has bounded inverse. In nutshell, we still circumvented the instability
    of ${\mathcal{B}}_{12}^{-1}$ by resolvent identities in the Wigner case.}.
The spectral resolution of
$H^{z_1}$ and $H^{z_2}$, however, are different even though they are defined via
the same matrix $X$. Lacking  the convenient resolvent identity, the previous methods would involve  inverting
${\mathcal{B}}_{12}$ even along its unstable direction which would lead to
an error term that is  bigger than~\eqref{GGnaive} by a large
multiplicative factor  $|\beta|^{-1}$. This was still affordable in the proof of the macroscopic
CLT since in the typical regime we had $|z_1-z_2|\sim 1$, hence $|\beta|^{-1}\sim 1$ was harmless. The
mesoscopic regime requires a completely new approach, which allows us to deal with the unstable
direction of the stability operator in a novel way completely circumventing the resolvent identity.

The proof of our improved multi-resolvent local law relies on the \emph{characteristic flow method}
that has previously been used for single-resolvent local laws and closely related quantities in
various models in
\cite{MR4009708, MR4168391, 2105.01178, 2202.06714, 2204.03419, 2206.03029, MR4416591}.
To our best knowledge, this method has not been applied in a multi-resolvent setup before
with the exception of~\cite[Proposition 4.5]{2206.03029} where the product of
special time-evolved resolvents at two different times
was considered.

The key idea in the characteristic flow method
is to consider an Ornstein-Uhlenbeck flow $X_t$ with initial condition $X_0=X$ and follow
the time evolution of its Hermitised resolvent $G_t$, but at time evolved spectral parameters $z_t$ and $\eta_t$. These parameters satisfy a natural first order differential equation
(the characteristic flow equation) that is chosen so that the leading terms
in the flow $\Tr G_t^{z_{1,t}}(\ii \eta_{1, t})G_t^{z_{2,t}}(\ii \eta_{2, t})$ cancel out.
In particular, unlike previous results using the characteristic flow method, we consider a
    {\it matrix
        version of the characteristics} (see \eqref{eq:newmatchar})
for the first time in the random matrix setting, %
which we believe to be useful also for much more general random matrix ensembles.
Along this flow $\eta_t>0$ decreases.
In this way one can transfer a local law~\eqref{GGnaive} from large $\eta>0$ to
a similar local law for much smaller $\eta>0$ at the expense of adding a Gaussian component
to the entry distribution of $X$ but without an additional $|\beta|^{-1}$ factor,
see Proposition~\ref{pro:charaveinfin}.
Furthermore, when $\eta$  becomes smaller than the threshold $|z_1-z_2|^2$, one $\eta^{-1/2}$ factor
in the right hand side of~\eqref{GGnaive} switches to the better $|\beta|^{-1/2}\sim |z_1-z_2|^{-1}$
factor, explicitly bringing  in the improvement we were looking for.
Next, we need to remove the added Gaussian component by fairly standard Green function comparison
arguments, however maintaining the improved precision in the error term requires nontrivial extra work.
Technically we do this via a {\it recursive tandem}: we successively reduce $\eta$  by small steps
and we immediately  remove the Gaussian component before further reduction in $\eta$.
The regimes $\eta\gtrsim |z_1-z_2|^2$ and $\eta\lesssim |z_1-z_2|^2$ need
separate estimates.

\begin{remark} We mention that with a similar (in fact simpler) scheme
    it is possible to give an alternative short proof of the single resolvent local for $G^z(\ii\eta)$
    as well, see~\eqref{eq:singlegllaw} later.  The recursive  tandem of the characteristic flow
    and the GFT offers an alternative proof to many existing optimal local laws and it seems to be
    more powerful than previous methods in many situations. For  the sake of brevity of the current paper we
    refrain from demonstrating this idea for $G^z(\ii\eta)$ since the optimal local law in this case has already been
    established.
\end{remark}

The multi-resolvent local law~\cref{theo:impll}, the improved
version of the local law~\eqref{GGnaive}, is also the key novel input for the critical \(\eta\)-regime and implies
the almost orthogonality of eigenvectors ${\bm w}_i^z$ of $H^z$ in the entire mesoscopic regime:
\begin{equation}\label{orth}
    |\langle {\bm w_i}^{z_1}, {\bm w_j}^{z_2}\rangle|\le n^{-\epsilon_1}, \qquad |z_1-z_2|\gg n^{-1/2+\epsilon_2},
\end{equation}
for any $i, j$,
and a similar bound for the singular vectors of $X-z$,
see Lemma~\ref{lem:overb}.
The analogous
result was proven only for  $|z_1-z_2|\gg n^{-\epsilon}$ in~\cite{1912.04100, MR4235475}.

We remark that  the relation~\eqref{orth} together with the
closely related almost independence of the low lying eigenvalues $\lambda^{z_1}$ and
$\lambda^{z_2}$ are of substantial interest in themselves, independently of the proof
of the mesoscopic CLT. Since low lying eigenvalues of $H^z$ are intuitively
related to the eigenvalues of $X$ near $z$, the above relations indicate that
the local spectral behavior of $X$ near $z_1$ and $z_2$ are largely
independent as long as $\abs{z_1-z_2}\gg n^{-1/2}$, in other words the ``correlation length'' in the spectrum of $X$ is $n^{-1/2}$.
We stress that currently we do not know how to turn this intuition
into a rigorous proof for general i.i.d.\ ensemble since our results concern only the eigenvalues and eigenvectors
of the Hermitisation of  $H^z$, equivalently, the singular values and singular vectors of $X-z$
for any fixed (deterministic) $z$, and not directly the eigenvalues and eigenvectors of $X$.

We note that for Ginibre ensemble the  explicit two point eigenvalue correlation function
exhibits an exponential decay for  $|z_1-z_2|\gg n^{-1/2}$ and
the \emph{eigenvector overlap} of eigenvectors belonging
to two different eigenvalues of $X$ has been explicitly computed
in~\cite{MR3851824, MR4095019}, effectively
proving a polynomial decay of correlation of eigenvectors beyond the scale $n^{-1/2}$
in the spectrum of $X$. The extension of these
results to i.i.d. matrices remains an outstanding open problem, closely related to the unsolved
\emph{bulk universality conjecture}\footnote{The same universality at the edge of the spectrum along the unit circle
    has been
    proven in~\cite{MR4221653}.} for the local eigenvalue statistics of an i.i.d. matrix $X$,
which represents the non-Hermitian analogue of the celebrated Wigner-Dyson-Mehta universality
for Wigner matrices~\cite{MR2917064}.

\subsection*{Notations and conventions}
We introduce some notations we use throughout the paper. For integers \(k\in\N \) we use the
notation \([k]:= \{1,\dots, k\}\). We write \(\DD\subset \C\) for the open unit disk, and for any
\(z\in\C \) we use the notation \(\dif^2 z:= 2^{-1} \ii(\dif z\wedge \dif \overline{z})\) for the two
dimensional volume form on \(\C \). For positive quantities \(f,g\) we write \(f\lesssim g\) and \(f\sim g\)
if \(f \le C g\) or \(c g\le f\le Cg\), respectively, for some constants \(c,C>0\) which depend only on the
constants appearing in~\eqref{eq:hmb}  and on $a, \tau$ in~\eqref{eq:resctestf}.
For any two positive real numbers \(\omega_*,\omega^*\in\R _+\) by \(\omega_*\ll \omega^*\)
we denote that \(\omega_*\le c \omega^*\) for some small constant \(0<c\le 1/100\). We denote
vectors by bold-faced lower case Roman letters \({\bm x}, {\bm y}\in\C ^k\), for some \(k\in\N\).
Vector and matrix norms, \(\norm{\vx}\) and \(\norm{A}\), indicate the usual Euclidean norm and
the corresponding induced matrix norm.
For any \(d \times d\) matrix \(A\) we use the notation \(\braket{ A}:= \frac{1}{d}\Tr  A\) to denote the
normalized trace of \(A\). Moreover, for vectors \({\bm x}, {\bm y}\in\C ^n\) and matrices \(A,B\in \C ^{2n\times 2n}\) we define the scalar product
\[
    \braket{ {\bm x},{\bm y}}:= \sum_i \overline{x}_i y_i, \qquad \braket{ A,B}:= \braket{ A^*B}.
\]
For an open set $\Omega\subset \C$, by $H_0^2(\Omega)$ we denote the Sobolev space defined as the completion of the smooth compactly supported functions $C_c^\infty(\Omega)$ under the norm
\[
    \lVert f\rVert_{H_0^2(\Omega)}=\left(\int_\Omega \big|\nabla f(z)\big|^2\,\dif^2z\right)^{1/2}.
\] We will use the concept of ``with very high probability'' meaning that for any fixed \(D>0\) the probability of the event is bigger than \(1-n^{-D}\) if \(n\ge n_0(D)\). Moreover, we use the convention that \(\xi>0\) denotes an arbitrary
small exponent which is independent of \(n\). We introduce the notion of \emph{stochastic domination} (see e.g.~\cite{MR3068390}): given two families of non-negative random variables
\[
    X=\tuple*{ X^{(n)}(u) \given n\in\N, u\in U^{(n)} }\quad\text{and}\quad Y=\tuple*{ Y^{(n)}(u) \given n\in\N, u\in U^{(n)} }
\]
indexed by \(n\) (and possibly some parameter \(u\)  in some parameter space $U^{(n)}$),
we say that \(X\) is stochastically dominated by \(Y\), if for all \(\xi, D>0\) we have \begin{equation}\label{stochdom}
    \sup_{u\in U^{(n)}} \Prob\left[X^{(n)}(u)>n^\xi  Y^{(n)}(u)\right]\leq n^{-D}
\end{equation}
for large enough \(n\geq n_0(\xi,D)\). In this case we use the notation \(X\prec Y\) or \(X= \landauOprec*{Y}\).

\subsection*{Acknowledgement} The authors are grateful to Joscha Henheik for his
help with the formulas in Appendix~\ref{app:evect}.

\section{Main result}

Let $X$ be an $n\times n$ matrix with independent identically distributed  (i.i.d.) complex entries
such that \smash{\(x_{ab}\stackrel{{\rm d}}{=} n^{-1/2}\chi\)}
for some complex random variable \(\chi\), satisfying the following:
\begin{assumption}\label{ass:1}
    The random variable $\chi$ satisfies
    \(\E \chi=\E \chi^2=0\)
    and \(\E \abs{\chi}^2=1\).
    In addition we assume the existence of its high moments,
    i.e.\ that there exist constants \(C_p>0\), for any \(p\in\N \), such that
    \begin{equation}\label{eq:hmb}
        \E \abs{\chi}^p\le C_p.
    \end{equation}
\end{assumption}
We focus on the complex symmetry class, see Remark~\ref{rmk:real} for the necessary modifications
for the real case.

Denote by $\{\sigma_i\}_{i\in [n]}$ the eigenvalues of $X$, and consider the centered linear statistics
\begin{equation}
    \label{eq:linstat}
    L_n(f_{z_0,a}):=\sum_i f_{z_0,a}(\sigma_i)-\mathbf{E}\sum_i f_{z_0,a}(\sigma_i),
\end{equation}
with
\begin{equation}
    \label{eq:resctestf}
    f_{z_0,a}(z):=f\big(n^a(z-z_0)\big), \qquad a\in \left(0,\frac{1}{2}\right), \qquad |z_0|\le 1-\tau,
\end{equation}
for some small fixed $\tau>0$. Here $f\in H_0^{2}(\Omega)$ with
a compact set $\Omega\subset \mathbf{C}$. We already mentioned in the introduction
that the sum in $L_n(f_{z_0,a})$ contains roughly  $n^{1-2a}$
summands  but it is not normalized by $(n^{1-2a})^{-1/2}$ unlike in the standard CLT for independent summands.

The main result of this paper is the following Central Limit Theorem for all mesoscopic scales.

\begin{theorem}
    \label{theo:mesoclt}
    Let $X$ be an $n\times n$ matrix satisfying Assumption~\ref{ass:1}, fix a small $\tau>0$, and let $|z_0|\le 1-\tau$ and $a\in (0, \frac{1}{2})$. Let $f_{z_0,a}$ be defined as in \eqref{eq:resctestf}, with $f\in H_0^{2}(\Omega)$
    for some a compact set  $\Omega\subset \mathbf{C}$.
    Then $L_n(f_{z_0,a})$ converges (in the sense of moments\footnote{We say that a sequence of random variables $Y_n$ converges to $Y_\infty$ \emph{in the sense of moments}, $Y_n  \stackrel{\rm m}{\Longrightarrow} Y_\infty$,  if $\E |Y_n|^k=\E|Y_\infty|^k+\mathcal{O}(n^{-c(k)})$ for any $k\in\mathbf{N}$ for some small $c(k)>0$. Even if not stated explicitely, the implicit constant in $\mathcal{O}(\cdot)$ may depend on $k$.} and therefore in distribution) to a complex Gaussian random variable $L(f)$,
    $$
        L_n(f_{z_0,a}) \stackrel{\rm m}{\Longrightarrow} L(f),
    $$
    with expectation $\E L(f)=0$, and second moments
    $\E|L(f)|^2=C(f,f)$, $\E L(f)^2=C(\overline{f},f)$, where
    \begin{equation}\label{variance}
        C(g,f):=\frac{1}{4\pi}\braket{\nabla g,\nabla f}_{L^2(\Omega)}.
    \end{equation}
    Moreover, the expectation is given by
    \begin{equation}
        \label{eq:compexp}
        \begin{split}
            \mathbf{E}\sum_i f_{z_0,a}(\sigma_i)&=\frac{n^{1-2a}}{\pi}\int_\mathbf{C} f(z)\, \dif^2 z+\frac{1}{8\pi}\int_\mathbf{C}\Delta f(z)\,\dif^2z+\mathcal{O}\left(n^{-2a}+n^{-c}\right),
        \end{split}
    \end{equation}
    for some small fixed $c>0$. The implicit constant in $\mathcal{O}(\cdot)$ may depend on $\norm{\Delta f}_{L^2(\Omega)}$, and $|\Omega|$.

\end{theorem}

\begin{remark}~
    \begin{enumerate}[label=(\roman*)]
        \item In Theorem~\ref{theo:mesoclt} we stated the CLT for the mesoscopic regime $a\in (0,1/2)$.
              The complementary regime $a\in [0,\epsilon]$, for some small fixed $\epsilon>0$, was already
              covered in \cite[Remark 2.6]{1912.04100} (in fact that proof also holds on any
              almost macroscopic scale where the power function
              $n^a$ in~\eqref{eq:resctestf} is replaced with any sequence  $C_n\to \infty$, $C_n\le n^\epsilon$).
        \item Note that in the macroscopic regime, $a=0$,
              both asymptotic variance and expectation of \(L_n(f)\) depend on the fourth
              cumulant \(\kappa_4=\kappa_4(\chi)=\E\abs{\chi}^4-2(\E\abs{\chi}^2)^2\) of the entry
              distribution, c.f.~\cite[Theorem 2.2]{1912.04100}. The fact that in the mesoscopic regime the
              CLT only depends on the first two moments of $\chi$ is yet another manifestation of the effect that
              local properties of the spectrum are more universal than global properties.
    \end{enumerate}
\end{remark}

\begin{remark}[Gaussian Free Field]
    We recall from~\cite[Section 2.1]{1912.04100} that in the macroscopic case $a=0$ for \(\kappa_4\ge0\) the limiting Gaussian process \(L\) could be interpreted as
    \begin{equation}\label{L}
        L = \frac{1}{\sqrt{4\pi}} Ph + \sqrt{\kappa_4} (\braket{\cdot }_\DD-\braket{\cdot}_{\partial\DD})\Xi,
    \end{equation}
    where \(Ph\) is the projection of the Gaussian free field (GFF) \(h\) on some bounded domain \(\Omega\supset\DD\) conditioned to be harmonic in the complement \(\DD^c\) of the unit disk \(\DD\), \(\Xi\) is a standard real Gaussian
    variable, independent of $h$,  and \(\braket{\cdot}_{\DD},\braket{\cdot}_{\partial\DD}\)
    denote the averaging functionals on \(\DD,\partial\DD\).

    By Theorem~\ref{theo:mesoclt} we can now conclude that in the bulk on
    all mesoscopic scales \(0<a<1/2\),
    \begin{equation}
        L = \frac{1}{\sqrt{4\pi}} h
    \end{equation}
    for some \(h\) in the equivalence class of the whole-plane GFF
    (which is only defined modulo additive constants), see e.g.~\cite[Section 2.2.1]{1302.4738}.
    Compared to the macroscopic case, however, this interpretation is also valid for \(\kappa_4<0\).
\end{remark}

\begin{remark}
    By polarization we  also conclude a multivariate CLT.
    In fact, our proof actually gives the joint Gaussianity of
    $$
        L_n(f_{z_{1},a_1}^{(1)}), L_n(f_{z_{2},a_2}^{(2)}), \ldots,   L_n(f_{z_{p},a_p}^{(p)}), \qquad p\in \N,
    $$
    for the rescaled versions $f_{z_{i},a_i}^{(i)}(z): = f^{(i)}(n^{a_i}(z-z_{i}))$ of
    different test functions $f^{(i)}\in H_0^{2}(\Omega)$ around fixed reference points $z_{i}$
    with $|z_{i}|\le 1-\tau$ and scales $a_i\in (0,1/2)$ that are not necessarily equal.
    In particular, from our proof below (especially around~\eqref{ff}) one can easily see that $\E L_n(f_{z_{i},a_i}^{(i)})\overline{L_n(f_{z_{j},a_j}^{(j)}})$
    can be non-zero (in the limit $n\to\infty$) only if $a_i=a_j$, $z_{i}= z_{j}$ and
    $\mathrm{supp}(f^{(i)})\cap \mathrm{supp}(f^{(j)})\ne \emptyset$.
\end{remark}

\begin{remark}
    \label{rem:opt}
    In Theorem~\ref{theo:mesoclt} we assumed $f\in H^2_0(\Omega)$ for simplicity, but in fact the somewhat weaker
    regularity  condition of the form $\mbox{supp}(f)\subset \Omega$ with
    $\Delta f\in L^{1+\delta}(\Omega)$ for any fixed $\delta>0$ and compact $\Omega$  suffices.
    By the Calderon-Zygmund inequality, this condition implies $f\in W^{2, 1+\delta}_0(\Omega)$, hence $f\in H^1(\Omega)$
    as well, and these conditions are sufficient for our proof.
    Note that the assumption $f\in H^1(\Omega)$ is necessary to make sure that the variance in \eqref{variance} is finite; in particular, this implies that our regularity condition on $f$ is close to being  optimal.
\end{remark}

\section{Proof strategy}

To analyze the linear statistics in \eqref{eq:linstat}
for the test function $f_{z_0,a}$ given in \eqref{eq:resctestf},
we rely on Girko's formula \cite{MR773436} in the form used in \cite{MR3306005}:
\begin{equation}
    \label{eq:splitingirko}
    \begin{split}
        L_n(f_{z_0,a})&=\frac{1}{4\pi}\int \Delta f_{z_0,a}(z) \left[\log\big| \mathrm{det}(H^z-\ii T)\big|+
            \ii\left(\int_0^{\eta_0}+\int_{\eta_0}^{\eta_c}+\int_{\eta_c}^T\right) \bigg[\mathrm{Tr}[G^z(\ii\eta)]-\E\mathrm{Tr}[G^z(\ii\eta)]\bigg]\,\dif \eta\right]\,\dif^2 z \\
        &=: J_T(f_{z_0,a})+I_{0}^{\eta_0}(f_{z_0,a})+I_{\eta_0}^{\eta_c}(f_{z_0,a})+I_{\eta_c}^T(f_{z_0,a}),
    \end{split}
\end{equation}
where we choose the integration  thresholds
\[
    \eta_0:=n^{-1-\delta_0},\qquad \eta_c:=n^{-1+\delta_1},
\]
for some small $\delta_0,\delta_1>0$ and we set $T=n^{100}$.
We recall the Hermitisation $H^z$ of $X-z$ given by
\begin{equation}
    \label{eq:herm}
    H^z:=\left(\begin{matrix}
            0       & X-z \\
            (X-z)^* & 0
        \end{matrix}\right),
\end{equation}
and its resolvent $G^z(w):=(H^z-w)^{-1}$, with $w\in\mathbf{C}\setminus\mathbf{R}$.
By spectral symmetry, the eigenvalues of $H^z$ come in opposite pairs, so we label them
as $\{ \lambda^z_{\pm i}\}_{i\in [n]}$ where $\lambda_{-i}^z=-\lambda_{i}^z$. The corresponding
orthonormal eigenvectors \(\{{\bm w}^z_{\pm i}\}_{i\in [n]}\) consequently decompose into two $n$-vectors
with symmetry \(\{{\bm w}^z_{\pm i}\}_{i\in [n]}=\{({\bm u}_i^z,\pm {\bm v}_i^z)\}_{i\in [n]}\) and
with $\| {\bm u}_i^z\|^2=\| {\bm v}_i^z\|^2=1/2$.

We have split \eqref{eq:splitingirko} into the sum of four terms and each of them will be analyzed
using different techniques. The term $J_T$ will easily be estimated as in \cite[Proof of Theorem 2.3]{MR4408013},
whilst the  fact that $I_0^{\eta_0}$ is negligible will follow by smoothing inequalities for the smallest singular
value of $X-z$ (see \cite[Theorem 3.2]{MR2684367}). To estimate $I_{\eta_0}^{\eta_c}$, which is the regime
when $\eta$ is proportional to the level spacing of the eigenvalues of $H^z$ around zero,
we will need the asymptotic independence of $\mathrm{Tr} G^{z_1}$ and
$\mathrm{Tr} G^{z_2}$ for $|z_1-z_2|\gg n^{-1/2}$ (see Proposition~\ref{prop:indmr} below).
The proof of this proposition is analogous to \cite[Section 7]{1912.04100} relying on the Dyson Brownian motion,
once the following asymptotic orthogonality of the singular vectors of $X-z_1$ and $X-z_2$ is proven.

\begin{theorem}
    \label{lem:overb}
    Let \(\{{\bm w}^{z_l}_{\pm i}\}_{i=1}^n=\{({\bm u}_i^{z_l},\pm {\bm v}_i^{z_l})\}_{i=1}^n\)
    be the eigenvectors of  \(H^{z_l}\) for  \(l=1,2\). Then for any sufficiently small \(\omega_d, \omega_p>0\) there exist \(\omega_B, \omega_E>0\) such that if \(n^{-1/2+\omega_p}\le\abs{z_1-z_2}\le n^{-\omega_d}\), then
    \begin{equation}
        \label{eq:bbev}
        \abs*{\braket{ {\bm u}_i^{z_1}, {\bm u}_j^{z_2}}}+\abs*{\braket{ {\bm v}_i^{z_1}, {\bm v}_j^{z_2}}}\le n^{-\omega_E}, \quad 1\le i,j\le n^{\omega_B},
    \end{equation}
    with very high probability.
\end{theorem}

The proof of this theorem is presented in Section~\ref{app:ind};
it will be a simple consequence of  the new two--resolvent local law for
$\Tr G^{z_1}(\ii \eta_1) G^{z_2}(\ii \eta_2)$ given in  Theorem~\ref{theo:impll} below.
Finally, the leading contribution to $L_n(f)$ comes  from the regime $I_{\eta_c}^T$,
which thus needs to be computed more precisely by identifying all its moments. This is
done first for the resolvent  in Proposition~\ref{prop:CLTresm} and later we extend it
to the $z-$ and $\eta$-integrals of the resolvent according to Girko's formula.
The fundamental input to prove Proposition~\ref{prop:CLTresm} will be again our new local law from Theorem~\ref{theo:impll}.

Before stating the new two--resolvent local law for
$\mathrm{Tr}G^{z_1}G^{z_2}$ we recall the local law for a single resolvent.
In the regime $|\Im w|\gg 1/n$ the resolvent $G^z(w)$ has a deterministic leading term.
This deterministic approximation is given by
\begin{equation}
    \label{Mz}
    M^z(w):=\left(\begin{matrix}
            m^z(w)              & -zu^z(w) \\
            -\overline{z}u^z(w) & m^z(w)
        \end{matrix}\right), \qquad u^z(w):=\frac{w}{w+m^z(w)},
\end{equation}
with $m^z$ being the unique solution of the scalar equation, with a side condition,
\begin{equation}
    \label{eq:MDEscal}
    -\frac{1}{m^z(w)}=w+m^z(w)-\frac{|z|^2}{w+m^z(w)}, \qquad \Im m^z(w)\Im w>0.
\end{equation}
Here $M^z=M^z(w)$ is a $2n\times 2n$ \emph{block constant} matrix, i.e. it has a $2\times 2$ block structure with each block being a constant multiple of the $n\times n$ identity. We remark that throughout this paper by $2\times 2$ block matrices we always refer to $2n\times 2n$ matrices which  consists of four $n\times n$ blocks. Furthermore, we will say that a matrix $A\in\C^{2n \times 2n}$ is block traceless if it is a $2\times 2$ block matrix such that the trace of each of its blocks is equal to zero. We will work on the imaginary axis, $\Re w=0$, as required in Girko's formula. More precisely, we have the following \emph{average} and \emph{isotropic} local laws from \cite{MR3770875}
\begin{equation}
    \label{eq:singlegllaw}
    \abs{\braket{A(G^z(\ii\eta)-M^z(\ii\eta))}} \prec \frac{\norm{A}}{n\abs{\eta}}, \qquad  \abs{\braket{\vx,(G^z(\ii\eta)-M^z(\ii\eta))\vy}} \prec \frac{\norm{\vx}\norm{\vy}}{\sqrt{n\abs{\eta}}}
\end{equation}
for any deterministic matrix and vectors \(A,\vx,\vy\), uniformly in $|\eta|\ge n^{-1+\epsilon}$ and $|z|\le 1-\tau$.  Furthermore, using trivial computations, along the imaginary axis we have the expansion for $m^z, u^z$:
\begin{lemma}
    Fix small $\eta,\tau>0$, $z$ such that $|z|\le 1-\tau$, let $m^z(\ii\eta)$ be the unique solution of \eqref{eq:MDEscal}, and let $u^z(\ii\eta)$ be defined as in \eqref{Mz}. Then, we have
    \begin{equation}
        \label{eq:expsmalleta}
        \begin{split}
            m^z(\ii\eta)&=\ii \sqrt{1-|z|^2}+\ii\eta\frac{2|z|^2-1}{2(1-|z|^2)}+\mathcal{O}(\eta^2), \\
            u^z(\ii\eta)&=1-\frac{\eta}{\sqrt{1-|z|^2}}+\mathcal{O}(\eta^2).
        \end{split}
    \end{equation}
    The implicit constant in $\mathcal{O}(\cdot)$ depends on $\tau$.
\end{lemma}

The deterministic approximation $M^z$ comes from the unique solution of the \emph{matrix Dyson equation}
$$
    - [M^z(w)]^{-1} = \begin{pmatrix} w & z \cr \bar z & w \end{pmatrix} + \mathcal{S}[M^z(w)], \qquad w\in \C\setminus \R
$$
with the side condition $(\Im w)\Im M^z(w)>0$. Here $\mathcal{S}$ is the \emph{covariance operator}
which acts on any matrix $R\in\mathbf{C}^{2n\times 2n}$  as\footnote{Here we recall the convention that for any $A\in\C^{d\times d}$ we use the notation $\braket{A}=d^{-1}\mathrm{Tr}[A]$. For example, $\braket{R}=(2n)^{-1}\mathrm{Tr}[R]$ and $\braket{R_{11}}=n^{-1}\mathrm{Tr}[R_{11}]$.}
\begin{equation}\label{defS}
    \mathcal{S}[R]:=\left( \begin{matrix}
            \braket{R_{22}} & 0 \\ 0 & \braket{R_{11}}
        \end{matrix} \right) \quad \text{i.e.} \quad \mathcal{S}[R]=2\braket{RE_2}E_1+2\braket{RE_1}E_2=\braket{R}-\braket{RE_-}E_-,
\end{equation}
where $R_{ij}$, with $i,j\in [2]$ are the four blocks of $R$.
In the second formula we expressed $\mathcal{S}$ in a basis representation, where
we set $E_-:=E_1-E_2$ and we
defined the \emph{ $2\times 2$ \emph{block constant matrices}}
\begin{equation}
    \label{eq:defe1e2f}
    E_1:=\left(\begin{matrix}
            1 & 0 \\
            0 & 0
        \end{matrix}\right), \qquad E_2:=\left(\begin{matrix}
            0 & 0 \\
            0 & 1
        \end{matrix}\right), \qquad F:=\left(\begin{matrix}
            0 & 0 \\
            1 & 0
        \end{matrix}\right).
\end{equation}
Note that $\{\sqrt{2} E_1,\sqrt{2}E_2,\sqrt{2}F,\sqrt{2}F^*\}$
is an orthonormal basis of the $2\times 2$ block constant matrices.

The \emph{averaged (tracial) version} of  our key  two--resolvent local contained in
the following theorem.

\begin{theorem}
    \label{theo:impll}
    Fix small $\epsilon,\tau,\omega_d>0$, independent of $n$  and  let
    $z_1,z_2\in \mathbf{C}$ such that $|z_i|\le 1-\tau$ and $|z_1-z_2|\le n^{-\omega_d}$, then for any deterministic matrices $A, B$ it holds
    \begin{equation}
        \label{eq:goodll}
        \big|\braket{(G^{z_1}(\ii\eta_1)AG^{z_2}(\ii\eta_2)-M_{12}^A)B}\big|\prec \frac{1}{n\eta_*^{3/2}(\eta^*)^{1/2}} \left(\eta_*^{1/6}+n^{-1/10}+\frac{1}{\sqrt{n\eta_*}}+\left(\frac{\eta^*}{\eta^*+|z_1-z_2|^2}\right)^{1/4}\right),
    \end{equation}
    with
    \begin{equation}
        \label{eq:defM12A}
        M_{12}^A=M_{12}^A(z_1,\ii\eta_1,z_2,\ii\eta_2):=(1-M_1\mathcal{S}[\cdot]M_2)^{-1}[M_1AM_2], \qquad M_j := M^{z_j}(\ii \eta_j), \quad j=1,2,
    \end{equation}
    and $\eta_*:=|\eta_1|\wedge|\eta_2|$, $\eta^*:=|\eta_1|\vee|\eta_2|$.
    The bound in \eqref{eq:goodll} holds uniformly for any
    matrices $A,B$ with $\norm{A}+\norm{B}\lesssim 1$
    and for $\eta_*\ge n^{-1+\epsilon}$. Additionally, for the deterministic term in \eqref{eq:goodll}
    we have the bound
    \begin{equation}
        \lVert M_{12}^A\rVert\lesssim \frac{1}{|z_1-z_2|^2+\eta^*}.
    \end{equation}
\end{theorem}
The proof of Theorem~\ref{theo:impll} will be divided into two steps: (i) in Section~\ref{sec:normllaw}
we use the characteristics flow (see the introduction for relevant references using this method)
to show that if we know \eqref{eq:goodll} for large $\eta$'s then we can propagate the same
bound to smaller $\eta$'s at the expense of
adding a Gaussian component to $X$ (this will be used with an initial $\eta\sim 1$),
(ii) in Section~\ref{sec:GFT} we use a \emph{Green function comparison argument (GFT)} to remove the Gaussian component added in (i).

As we already pointed out in the introduction,  the local law in \eqref{eq:goodll}
is a significant improvement compared to \cite[Theorem 5.2]{1912.04100}.
The main difference is that the error term in \cite[Theorem 5.2]{1912.04100} contained
an additional large factor $\| {\mathcal{B}}_{12}^{-1}\|$,
the norm of the inverse of the {\it stability operator}
\begin{equation}\label{stabop}
    {\mathcal{B}}_{12}:= 1-M_1\mathcal{S}[\cdot]M_2
\end{equation}
acting on $\C^{2n\times 2n}$ matrices.
As we have $\| {\mathcal{B}}_{12}^{-1}\|\sim \big[ |z_1-z_2|^2 +|\eta_1|+|\eta_2|\big]^{-1} $,
this factor was affordable in \cite{1912.04100}
since there we considered the regime $|z_1-z_2|\ge n^{-\epsilon}$
but it would be badly not affordable in the current situation $|z_1-z_2|\ge n^{-1/2+\epsilon}$.
The removal of the factor $\| {\mathcal{B}}_{12}^{-1}\|$, hence getting an improved bound
for any $|z_1-z_2|\ge n^{-1/2+\epsilon}$,  is
the main achievement of the new characteristic flow technique.
In connection with this improvement, the range of
$\eta$ has also been improved: the estimate in \eqref{eq:goodll} holds uniformly
in $\eta_*\ge n^{-1+\epsilon}$ while the local law \cite[Theorem 5.2]{1912.04100}
holds
only for $\eta_*\ge n^{-1+\epsilon}|z_1-z_2|^{-2}$, which in the current mesoscopic case can be basically order one,
rendering the estimate useless. Furthermore,
the error term in \eqref{eq:goodll} is better than the one in~\cite[Theorem 5.2]{1912.04100} even in the
almost macroscopic
regime $|z_1-z_2|\ge n^{-\epsilon}$
for certain values of $\eta_*$. Finally, for a complete comparison
we also mention  that~\cite[Theorem 5.2]{1912.04100} was valid for all $z$ values, including
the edge regime $|z|\approx 1$ of the circular law, while for technical convenience
we restricted Theorem~\ref{theo:impll} to the bulk regime $|z|\le 1-\tau$.

We stated only the averaged (tracial) version of the two--resolvent local law since
this is needed  in the proof of our main result Theorem~\ref{theo:mesoclt}.  However
our method would also give  an isotropic local law for the matrix elements  $\langle {\bf x}, G_1AG_2 {\bf y}\rangle $
for deterministic vectors ${\bf x}, {\bf y}$
with a similar gain from the regime $|z_1-z_2|\gg \eta^*$ as in \eqref{eq:goodll}.

By Theorem~\ref{theo:impll} we readily conclude (see Appendices~\ref{app:CLT}--\ref{app:ind}) the following  two
propositions, whose combination will prove Theorem~\ref{theo:mesoclt} in Section~\ref{sec:mesoCLT}.

\begin{proposition}[CLT for resolvents]\label{prop:CLTresm}
    Let \(\epsilon,\xi,\tau,\omega_p,\omega_d>0\) be small constants  and $p\in \N$. Denote by $\Pi_p$
    the set of pairings\footnote{Note that $\Pi_p=\emptyset$ if $p$ is odd.}  on \([p]\).
    Then for \(z_1,\dots,z_p\in\C\), with $|z_i|\le 1-\tau$, $n^{-1/2+\omega_p}\le|z_l-z_m|\le n^{-\omega_d}$, and \(\eta_1,\dots,\eta_p\ge n^{-1+\epsilon}\), we have
    \begin{gather}
        \begin{aligned}
            \E\prod_{i\in[p]} \braket{G_i-\E G_i} & = \sum_{P\in \Pi_p}\prod_{\{i,j\}\in P} \E  \braket{G_i-\E G_i}\braket{G_j-\E G_j}  + \mathcal{O}\left(\Psi\right)     \\
                                                  & = \frac{1}{n^p}\sum_{P\in \Pi_p}\prod_{\{i,j\}\in P} \frac{V_{i,j}+\kappa_4 U_i U_j}{2}+ \mathcal{O}\left(\Psi\right),
        \end{aligned}\label{eq CLT resovlent}\raisetag{-5em}
    \end{gather}
    where \(G_i=G^{z_i}(\ii \eta_i)\),
    \begin{equation}\label{eq psi error}
        \Psi:= \frac{n^\xi}{(n\eta_*)^{1/2}}\prod_{i\in[p]}\frac{1}{n\eta_i},
    \end{equation}
    \(\eta_*:= \min_i\eta_i\), and \(V_{i,j}=V_{i,j}(z_i,z_j,\eta_i,\eta_j)\) and \(U_i=U_i(z_i,\eta_i)\) are defined as
    \begin{equation}
        \label{eq:exder}
        \begin{split}
            V_{i,j}&:= \frac{1}{2}\partial_{\eta_i}\partial_{\eta_j} \log \bigl[ 1+(u_i u_j\abs{z_i}\abs{z_j})^2-m_i^2 m_j^2-2u_i u_j\Re z_i\overline{z_j}\bigr], \\
            U_i&:= \frac{\ii}{\sqrt{2}}\partial_{\eta_i} m_i^2,
        \end{split}
    \end{equation}
    with \(m_i=m^{z_i}(\ii\eta_i)\) and \(u_i=u^{z_i}(\ii\eta_i)\) from~\eqref{Mz}.
    Finally, $\kappa_4:= \E |\chi|^4-2 $ is the fourth cumulant
    of the random variable $\chi$ in Assumption~\ref{ass:1}.
\end{proposition}

The expectation of $G_i$ appearing in \eqref{eq CLT resovlent} has already been identified with
sufficiently high precision in~\cite[Lemma 6.2]{1912.04100}:
\begin{equation}
    \label{prop clt exp}
    \braket{\E G}= \braket{M} - \frac{\ii\kappa_4}{4n}\partial_\eta(m^4) + \cO\Bigl(\frac{1}{\abs{1-\abs{z}}n^{3/2} (1+\eta)}+\frac{1}{\abs{1-\abs{z}}(n\eta)^2}\Bigr),
\end{equation}
which holds for $\eta\ge n^{-1+\epsilon}$ and $|z|\le C$, for some constant $C>0$.

The following proposition states that
$\braket{G^{z_1}}$ and $\braket{G^{z_2}}$ are asymptotically independent  as long as $n^{-1/2}\ll |z_1-z_2|\ll 1$.
This complements \cite[Proposition 3.5]{1912.04100}
which proves a similar result in the regime $|z_1-z_2|\sim 1$.

\begin{proposition}[Independence of resolvents with small imaginary part]
    \label{prop:indmr}
    Fix \(p\in \mathbf{N}\). For any sufficiently small constants \(\tau,\omega_p,\omega_d,\delta_0,\delta_1>0\), there exists \(\omega\gg \delta_0,\delta_1\) such that for any \(\abs{z_l}\le 1-\tau\), \(n^{-1/2+\omega_p}\le\abs{z_l-z_m}\le n^{-\omega_d}\), with \(l,m \in [p]\), \(l\ne m\), it holds
    \begin{equation}
        \label{eq:indtrlm}
        \E \prod_{l=1}^p \braket{ G^{z_l}(\ii\eta_l)}=\prod_{l=1}^p\E  \braket{ G^{z_l}(\ii\eta_l)}+\mathcal{O}\left(\frac{n^{p\delta_0+\delta_1}}{n^{\omega}}\right),
    \end{equation}
    for any \(\eta_1,\dots,\eta_p\in [n^{-1-\delta_0},n^{-1+\delta_1}]\).
\end{proposition}

\begin{remark}\label{rmk:real} We formulated our results for the complex case.
    We now briefly explain how the above strategy needs to be modified for the real symmetry class, i.e.
    what should be changed to extend our real macroscopic CLT proof~\cite{MR4235475} to
    the entire mesoscopic regime. Apart from some irrelevant technicalities, the main difference between the proofs for the
    complex and the real cases is that the singular vector overlap bound in Theorem~\ref{lem:overb}
    is needed not only for the low lying singular values, but practically for all of them.
    This requires to prove the improved version of the local law~\eqref{GGnaive}
    for $\Tr G^{z_1}(w_1) G^{z_2}(w_2)$
    not only on the imaginary axis, $w_j=\ii \eta_j$, but for any spectral parameters $w_1, w_2$ (see~\cite[Theorem 3.5]{MR4235475}). The GFT argument can easily be adjusted to this
    more general case, but the characteristics flow requires very precise explicit calculations
    that were simpler on the imaginary axis. These calculations are elementary but fairly tedious,
    so we omit them here.
\end{remark}

\section{Mesoscopic CLT for linear statistics: Proof of Theorem~\ref{theo:mesoclt}}
\label{sec:mesoCLT}

The proof of Theorem~\ref{theo:mesoclt}  follows similarly to \cite[Section 4]{1912.04100} once
Propositions~\ref{prop:CLTresm}--\ref{prop:indmr} are given as inputs. We thus now only present the necessary
major  steps in the proof of Theorem~\ref{theo:mesoclt} without presenting the detailed
(fairly elementary) computations.
The computations to get the expectation in \eqref{eq:compexp}
are completely analogous (indeed easier since here $|z|\le 1-\tau$) to \cite[Section 4.2]{1912.04100},
and so omitted. Here we only focus on the computation of the variance and higher moments.
We remark that in \cite{1912.04100} we assumed that $f\in H^{2+\delta}_0(\Omega)$ instead
of $f\in H^2_0(\Omega)$ only to compute the boundary term from the last two lines of \cite[Eq. (4.29)]{1912.04100}.
Since we now have $|z_0|\le 1-\tau$ this additional ($+\delta$) regularity assumption
is not needed in this case (see also the related Remark~\ref{rem:opt}).

First of all we notice that the main contribution to \eqref{eq:splitingirko} comes when $\eta\gg n^{-1}$,
i.e. the regime $I_{\eta_c}^T$ is the only term in \eqref{eq:splitingirko} giving an order one contribution
to the linear statistics $L_n(f_{z_0,a})$. This fact is stated in the following lemma whose proof
is postponed to the end of this section.

\begin{lemma}
    \label{lem:firstappr}
    Fix $\tau>0$,  $|z_0|\le 1-\tau$ and $a\in (0, \frac{1}{2})$,
    then for any $p\in\N$, and for $f_{z_0,a}^{(i)}(z):=f^{(i)}(n^a(z-z_0))$,
    with $f^{(i)}\in H_0^2(\Omega)$, it holds
    \begin{equation}
        \label{eq:firstappr}
        \E\prod_{i=1}^p L_n(f_{z_0,a}^{(i)})=\E \prod_{i=1}^p I_{\eta_c}^T(f_{z_0,a}^{(i)})+\mathcal{O}\left(n^{-c(p)}\right),
    \end{equation}
    for some small $c(p)>0$. The implicit constant in $\mathcal{O}(\cdot)$ may depend on $p$, $\norm{\Delta f^{(i)}}_{L^2(\Omega)}$, and $|\Omega|$.
\end{lemma}

Using Lemma~\ref{lem:firstappr} we then compute the deterministic approximation of the moments
of $L_n(f_{z_0,a})$ via the moments of the leading term  $I_{\eta_c}^T$.
The proof of this lemma is presented after we conclude the proof of our main result Theorem~\ref{theo:mesoclt}.

\begin{lemma}
    \label{lem:detterm}
    Consider $f_{z_0,a}^{(i)}$ as above, and recall that $\Pi_p$   denotes the set of pairings on $[p]$. Then it holds
    \begin{equation}
        \label{eq:detterm}
        \E \prod_{i=1}^p L_n(f_{z_0,a}^{(i)})=\sum_{P\in \Pi_p} \prod_{\{i,j\}\in P} \biggl[-\int_\C \dif^2 z_i
            \Delta f_{z_0,a}^{(i)} (z_i) \int_\C\dif^2 z_j\Delta f_{z_0,a}^{(j)} (z_j)\int_{0}^\infty \dif\eta_i\int_{0}^\infty\dif\eta_j \frac{V_{i,j}+\kappa_4 U_i U_j}{8\pi^2}\biggr] + \mathcal{O}(n^{-c(p)}),
    \end{equation}
    for some small \(c(p)>0\), with the $(z_i,\eta_i)$-dependent quantities \(V_{i,j}\) and \(U_i\) being defined in~\eqref{eq:exder}.  The implicit constant in $\mathcal{O}(\cdot)$ may depend on $p$, $\norm{\Delta f^{(i)}}_{L^2(\Omega)}$, and $|\Omega|$.

\end{lemma}

We are now ready to prove Theorem~\ref{theo:mesoclt}.

\begin{proof}[Proof of Theorem~\ref{theo:mesoclt}]
    To keep the presentation concise we only present the computations for higher moments, the computations for the expectation are completely analogous (actually easier) and so omitted (see e.g. \cite[Section 4.2]{1912.04100}).

    By Lemma~\ref{lem:detterm} we are only left with the computation of the deterministic term in the right  hand side of \eqref{eq:detterm} for \(f^{(1)},\ldots,f^{(p)}\in \{f,\overline{f}\}\). In particular, using the explicit formulas for $V_{i,j}$ and $U_iU_j$ from \eqref{eq:exder}, and using the explicit formulas for $m_i$, $u_i$ at $\eta_i=0$, we readily conclude
    \begin{equation}
        -\int_{0}^\infty \dif\eta_i\int_{0}^\infty\dif\eta_j \big[V_{i,j}+\kappa_4 U_i U_j\big]= -\frac{1}{2}\log|z_i-z_j|^2+\frac{\kappa_4}{2}(1-|z_i|^2)(1-|z_j|^2),
    \end{equation}
    for $|z_i|,|z_j|\le 1-\tau$ (see~\cite[Sections 4.3.1--4.3.3]{1912.04100}).

    Then, performing integration by parts in the $z_i,z_j$-variables and using that
    \[
        -\partial_{z_1}\partial_{\overline{z_2}}\log|z_1-z_2|^2\dif^2z_1\dif^2 z_2= \frac{\pi}{2} \delta (z_1-z_2)
    \]
    in the sense of distributions, we obtain\footnote{ Here by $\approx$ we mean that the equality holds up to error of size $n^{-c}$ for some small $c>0$.}
    \begin{equation}
        \label{ff}
        \begin{split}
            \E L_n(f_{z_0,a}^{(i)})L_n(f_{z_0,a}^{(j)}) &\approx\frac{1}{4\pi} \int_\C \nabla f_{z_0,a}^{(i)}(z) \cdot \nabla f_{z_0,a}^{(j)}(z)\, \dif^2 z+\frac{\kappa_4}{\pi^2}\left(\int_\C f_{z_0,a}^{(i)}(z)\,\dif^2z\right)\left(\int_\C f_{z_0,a}^{(j)}(z)\,\dif^2z\right)  \\
            &=\frac{1}{4\pi}  \int_\C \nabla f^{(i)}(z) \cdot \nabla f^{(j)}(z)\, \dif^2 z+\mathcal{O}\left(n^{-4a}\right),
        \end{split}
    \end{equation}
    which concludes the proof of Theorem~\ref{theo:mesoclt}.
\end{proof}

We now conclude this section with the proof of Lemma~\ref{lem:firstappr} and Lemma~\ref{lem:detterm}.

\begin{proof}[Proof of Lemma~\ref{lem:firstappr}]

    This proof  is analogous to  \cite[Lemma 4.3]{1912.04100}. We only present a few minor differences for completeness. Proceeding as in the proof of \cite[Lemma 4.3]{1912.04100}, using that $\|\Delta f^{(i)}_{z_i,a_i}\|_{L^q(\Omega)}= \|\Delta f^{(i)}\|_{L^q(\Omega)}$, $q\ge 1$, we conclude the following a priori bounds
    \begin{equation}
        \label{eq:b1}
        \abs{J_T}\le \frac{n^{1+\xi}\norm{ \Delta f}_{L^1(\Omega)}}{T^2}, \qquad \abs*{I_0^{\eta_0}}+\abs*{I_{\eta_0}^{\eta_c}}+\abs{I_{\eta_c}^T}\le n^\xi \norm{ \Delta f}_{L^2(\Omega)} \abs{\Omega}^{1/2}
    \end{equation}
    with very high probability for each  $f=f^{(i)}$, with $J_T = J_T(f^{(i)}_{z_i,a_i})$, and similarly for the other terms.
    Additionally, by \cite[Theorem 3.2]{MR2684367} we conclude
    \begin{equation}
        \label{eq:b2}
        \E\abs*{I_0^{\eta_0}}\le n^{-\delta'} \norm{ \Delta f}_{L^2(\Omega)},
    \end{equation}
    for some small fixed $\delta'>0$. Next we will prove that
    \begin{equation}
        \label{eq:b3}
        \E\abs*{I_{\eta_0}^{\eta_c}}^2\le n^{-\delta'} \norm{ \Delta f}_{L^2(\Omega)}^2.
    \end{equation}
    Note that combining \eqref{eq:b1}--\eqref{eq:b3} we immediately conclude \eqref{eq:firstappr}.

    For the proof of~\eqref{eq:b3} we now compute (here we neglect $\log n$-factors)
    \begin{equation}
        \label{eq:b4}
        \begin{split}
            \E \abs{I_{\eta_0}^{\eta_c}}^2&=\frac{n^2}{4\pi^2}\int\int_{\abs{z_2-z_1}\ge n^{-a-\delta_*}} \dif^2 z_1\dif^2 z_2\, \Delta f_{z_0,a}(z_1)  \Delta \overline{f_{z_0,a}(z_2)}  \\
            &\qquad\quad\times \int\int_{\eta_0}^{\eta_c} \dif \eta_1 \dif \eta_2  \E \Big[\braket{  G^{z_1}(\ii\eta_1) -\E  G^{z_1}(\ii\eta_1)}\braket{  G^{z_2}(\ii\eta_2) -\E  G^{z_2}(\ii\eta_2)}\Big]+\mathcal{O}\left(\frac{n^{2\xi}}{n^{\delta_*}} \right),
        \end{split}
    \end{equation}
    where we used that the regime $|z_1-z_2|< n^{-a-\delta_*}$ is bounded by $n^{2\xi-\delta_*}$:
    \[
        \begin{split}
            &\frac{n^2}{4\pi^2}\int\int_{\abs{z_2-z_1}< n^{-a-\delta_*}} \dif^2 z_1\dif^2 z_2\, \Delta f_{z_0,a}(z_1)  \Delta \overline{f_{z_0,a}(z_2)}  \\
            &\qquad\quad\times \int\int_{\eta_0}^{\eta_c} \dif \eta_1 \dif \eta_2  \E \Big[\braket{  G^{z_1}(\ii\eta_1) -\E  G^{z_1}(\ii\eta_1)}\braket{  G^{z_2}(\ii\eta_2) -\E  G^{z_2}(\ii\eta_2)}\Big] \\
            & \qquad\lesssim n^2\int\int_{\abs{z_2-z_1}< n^{-a-\delta_*}} \dif^2 z_1\dif^2 z_2\, \big|\Delta f_{z_0,a}(z_1) \big|\big| \Delta \overline{f_{z_0,a}(z_2)}\big|  \\
            &\qquad \qquad\quad\times \int\int_{\eta_0}^{\eta_c} \dif \eta_1 \dif \eta_2  \E \Big[\big|\braket{G^{z_1}(\ii\eta_1) -\E  G^{z_1}(\ii\eta_1)}\big|^2+\big|\braket{G^{z_2}(\ii\eta_2) -\E  G^{z_2}(\ii\eta_2)}\big|^2\Big] \\
            &\qquad\lesssim n^{2\xi}\int \dif^2 z_1\, \big|\Delta f_{z_0,a}(z_1) \big| \left(\int_{\abs{z_2-z_1}< n^{-a-\delta_*}}\dif^2 z_2\right)^{1/2} \left(\int_{\abs{z_2-z_1}< n^{-a-\delta_*}}\dif^2 z_2\, \big| \Delta \overline{f_{z_0,a}(z_2)}\big|
            ^2\right)^{1/2} \\
            &\qquad\lesssim n^{2\xi} \lVert \Delta f\rVert_{L^2(\Omega)}^2 n^{-a-\delta_*}n^a\lesssim n^{2\xi-\delta_*},
        \end{split}
    \]
    where we recall that $f_{z_0,a}(z)=f(n^a(z-z_0))$ by \eqref{eq:resctestf}; more precisely, we used the scaling
    \[
        \lVert\Delta f_{z_0,a}\rVert_{L^1(\Omega)}=  \lVert \Delta f\rVert_{L^1(\Omega)}, \qquad\quad    \lVert\Delta f_{z_0,a}\rVert_{L^2(\Omega)}= n^a   \lVert\Delta f\rVert_{L^2(\Omega)}.
    \]
    We point out that in the second inequality we also used the averaged local law in \eqref{eq:singlegllaw}. In particular, we remark that the implicit constant in $\mathcal{O}(\cdot)$ in \eqref{eq:b4} depends on $\norm{\Delta f}_{L^2(\Omega)}$ even if not written explicitly.

    Then, by Proposition~\ref{prop:indmr}, we conclude that
    \begin{equation}
        \label{eq:b5}
        \E \Big[\braket{  G^{z_1}(\ii\eta_1) -\E  G^{z_1}(\ii\eta_1)}\braket{  G^{z_2}(\ii\eta_2) -\E  G^{z_2}(\ii\eta_2)}\Big]=\mathcal{O}\left(\frac{n^{c(\delta_0+\delta_1)}}{n^\omega}\right),
    \end{equation}
    for some $c>0$ and $\omega\gg \delta_0+\delta_1$. Plugging \eqref{eq:b5} in \eqref{eq:b4} we conclude the proof of \eqref{eq:b3}.

\end{proof}

\begin{proof}[Proof of Lemma~\ref{lem:detterm}]

    Again, this proof is basically the same as its macroscopic counterpart in \cite[Lemma 4.7]{1912.04100}.
    Using the same notation as in \cite[Proof of Lemma 4.7]{1912.04100}, we define
    \[
        \widehat{Z}_i:= \bigcup_{j<i}\set{  z_j : \,\abs{z_i-z_j}\le n^{-a-\nu}}
    \]
    for some small fixed $\nu>0$. Then, similarly to \eqref{eq:b4}, we start
    removing the regime $|z_i-z_j|<n^{-a-\nu}$ (recall that $\eta_c=n^{-1+\delta_1}$):
    \begin{equation}
        \E \prod_{i=1}^p L_n\left(f_{z_0,a}^{(i)}\right)=\frac{(-n)^p}{(2\pi \ii)^p}\prod_{i\in [p]}\int_{\widehat{Z}_i^c}\dif^2 z_i
        \Delta f^{(i)}(z_i) \E \prod_{i\in[p]} \int_{\eta_c}^T \braket{G^{z_i}(\ii\eta_i)-\E G^{z_i}(\ii\eta_i)} \dif \eta_i
        + \mathcal{O}(n^{-c}),
    \end{equation}
    for some small $c>0$ which depends on $p,\xi,\delta_1,\nu$ and that may change from line to line.
    We point out that, before removing $|z_i-z_j|<n^{-a-\nu}$, here we also used Lemma~\ref{lem:firstappr}  to remove the integration regimes $\eta\in [0,\eta_c)$ and $\eta\in (T,\infty)$.

    Then, by Proposition~\ref{prop:CLTresm}, it readily follows that
    \begin{equation}
        \label{eq:prodresappl}
        \E \prod_{i=1}^p L_n\left(f_{z_0,a}^{(i)}\right)=-\prod_{i=1}^p \int_{\widehat{Z}_i^c} \dif^2 z_i \Delta f_{z_0,a}^{(i)} \sum_{P\in \Pi_p} \prod_{\{i,j\}\in P}  \int\int_{\eta_c}^T \dif\eta_i\dif\eta_j \frac{V_{i,j}+\kappa_4 U_i U_j}{8\pi^2}+ \mathcal{O}(n^{-c}).
    \end{equation}
    Using the following bounds
    \[
        \big|V_{i,j}\big|\lesssim \frac{(1+\eta_i)^{-2}(1+\eta_j)^{-2}}{|z_1-z_2|^2+\eta_i+\eta_j}, \qquad \big|U_i\big|\lesssim \frac{1}{(1+\eta_i)^3},
    \]
    from \cite[Eq. (4.21)]{1912.04100}, to add back first the removed
    $\eta$-regimes $[0,\eta_c]\cup [T,\infty)$ and then the regime $\widehat{Z}_i$
    in \eqref{eq:prodresappl} we conclude \eqref{eq:detterm}.
\end{proof}

\nc

\section{Proof of Theorem~\ref{theo:impll} for matrices with a Gaussian component}
\label{sec:normllaw}

Consider the Ornstein-Uhlenbeck (OU) flow
\begin{equation}
    \label{eq:OUaa}
    \dif X_t=-\frac{1}{2}X_t \dif t+\frac{\dif \mathfrak{B}_t}{\sqrt{n}},
\end{equation}
with $\mathfrak{B}_t$ a matrix whose entries are complex i.i.d. Brownian motions. Let
$$
    W_t: = \begin{pmatrix} 0 & X_t \cr X_t^* & 0 \end{pmatrix},
$$
be the Hermitisation of $X_t$ as in \eqref{eq:herm} (for $z=0$), and similarly define $B_t$ being the Hermitisation of $\mathfrak{B}_t$. Furthermore, for $i=1,2$, let
$G_i=G_t^{z_i}(\ii\eta_i):=(H_t^{z_i}-\ii\eta_i)^{-1}$ be the resolvent of $H_t^{z_i}:=W_t-Z_i$, with $|\eta_i|>0$ and (recall the definition of $F$ from \eqref{eq:defe1e2f})
\[
    Z_i:=\left(\begin{matrix}
            0              & z_i \\
            \overline{z_i} & 0
        \end{matrix}\right)=\overline{z_i}F+z_i F^*.
\]
In particular we let $\eta_i$ assume both positive and negative values. Note that along the flow \eqref{eq:OUaa}
the first two moments of $X_t$ are preserved, hence the deterministic approximation of $G_t^{z_i}(\ii\eta_i)$
is given by $M^{z_i}(\ii\eta_i)$ for any $t\ge 0$ (see \eqref{eq:singlegllaw}), i.e.
the deterministic approximation of $G_t^{z_i}(\ii\eta_i)$ is independent of time.

By \eqref{eq:OUaa} and It\^{o}'s formula,  the evolution of $\braket{G_t^{z_1}AG_t^{z_2}B}$, for any deterministic matrices $A$, $B$\footnote{ Here $B$ denotes a deterministic matrix while $B_t$ denotes a matrix whose entries are i.i.d. Brownian motions; we apologize for this slight abuse of notation.}, and for time-independent spectral parameters $\eta_i$ and $z_i$, is described by the following flow (recall the definition of $E_1,E_2$ from \eqref{eq:defe1e2f}):
\begin{equation}
    \label{eq:eqderres}
    \begin{split}
        \dif \braket{G_1AG_2B}&=\sum_\alpha\partial_\alpha\braket{G_1AG_2B}\frac{\dif (B_t)_\alpha}{\sqrt{n}}+\braket{G_1AG_2B}\dif t \\
        &\quad+2\braket{G_1AG_2E_1}\braket{G_2BG_1E_2}\dif t+2\braket{G_1AG_2E_2}\braket{G_2BG_1E_1}\dif t \\
        &\quad+ 2\braket{G_1E_1}\braket{G_1AG_2BG_1E_2}\dif t + 2\braket{G_1E_2}\braket{G_1AG_2BG_1E_1}\dif t \\
        &\quad+ 2\braket{G_2E_1}\braket{G_2BG_1AG_2E_2}\dif t+ 2\braket{G_2E_2}\braket{G_2BG_1AG_2E_1}\dif t \\
        &\quad+\frac{1}{2}\braket{(Z_1+\ii\eta_1)G_1AG_2BG_1}\dif t+\frac{1}{2}\braket{(Z_2+\ii\eta_2)G_2BG_1AG_2}\dif t.
    \end{split}
\end{equation}
Here $\alpha=(a,b)\in [2n]^2$ denotes a double index, and $\partial_\alpha$ denotes the directional derivative $\partial_{w_\alpha}$, with $w_\alpha=w_\alpha(t):=(W_t)_\alpha$. We point out that the summation $\sum_\alpha$ in \eqref{eq:eqderres} is restricted to either $a\le n$, $n<b\le 2n$ or $n<a\le 2n$, $b\le n$ even if not stated explicitly; we will use this notation throughout the paper.

We now allow both pairs of spectral parameters, $\eta_i$ and $z_i$, $i=1,2$, to be time dependent
in a specific way. Define
\[
    \Lambda_i:=\left(\begin{matrix}
            \ii\eta_i      & z_i       \\
            \overline{z_i} & \ii\eta_i
        \end{matrix}\right),
\]
and consider its time evolution along the  following differential equation (called the {\it characteristic equation})
\begin{equation}
    \label{eq:matchar}
    \partial_t\Lambda_{i,t}=-\frac{\Lambda_{i,t}}{2}-\mathcal{S}[M(\Lambda_{i,t})]
\end{equation}
with some initial condition $\Lambda_{i,0}$, with $\mathcal{S}$ being defined in \eqref{defS}.
Here we used the notation $M(\Lambda_{i,t}):=M^{z_{i,t}}(\ii\eta_{i,t})$. Written component-wise, we thus have that
\begin{equation}\label{comp1}
    \partial_t \eta_{i,t}=-\Im m^{z_i,t}(\ii \eta_{i,t})-\frac{\eta_{i,t}}{2}, \qquad\qquad\quad \partial_t z_{i,t}=-\frac{z_{i,t}}{2}.
\end{equation}
Since $\Im m^z(\ii\eta)$ is undefined for $\eta=0$, we will always run this flow up to a maximal time
$$T^*=
    T^*(\Lambda_{i,0})= T^*(z_{i,0}, \eta_{i,0}):=\sup \{ t\; : \; \sgn \eta_{i,t}=\sgn \eta_{i,0} \}
$$
to guarantee that $\eta_{i, t}$  never crosses the real axis.
Define $m_{i,t}:=m^{z_{i,t}}(\ii\eta_{i,t})$, $u_{i,t}:=u_{z_{i,t}}(\ii\eta_{i,t})$, and note that
their time dependence is particularly simple:
\begin{equation}
    \label{eq:timeev}
    m_{i,t}=e^{t/2}m_{i,0}, \qquad u_{i,t}= e^t u_{i,0}, \qquad z_{i,t}=e^{-t/2}z_{i,0}, \qquad \eta_{i,t}=e^{-t/2}\eta_{i,0}-(e^{t/2}-e^{-t/2})\Im m_{i,0}.
\end{equation}
Additionally, by \eqref{eq:expsmalleta} we have
$ \Im m_{i,0}^{z_{i,0}} = \mathrm{sgn}(\eta_{i,0})\sqrt{1-|z_{i,0}|^2}
    +\mathcal{O}(|\eta_{i,0}|)$, which, together with \eqref{eq:timeev}, for $t\ll 1$ gives $\eta_{i,t}=\eta_{i,0}- \mathrm{sgn}(\eta_{i,0})c_it$,
with some time dependent positive coefficient
\begin{equation}
    \label{eq:defci}
    c_i=c_i(t)= \sqrt{1-|z_{i,t}|^2}+\mathcal{O}(\eta_{i,0}).
\end{equation}
Note that $c_i$ is well separated from zero along the whole flow, as a consequence of the fact that if initially $|z_{i,0}|\le 1-\tau$  for some $\tau>0$ then  we also have $|z_{i,t}|\le 1-\tau$ for any $t\ge 0$.  In particular, this shows that $|\eta_{i,t}|=|\eta_{i,0}|-c_it$, so the flow approaches the real axis with
a speed of order one in the regime away from the non-Hermitian spectral edge $|z|=1$. This shows that the characteristics are monotone in time, i.e. $|\eta_{i,t}|\le |\eta_{i,s}|$ for $s\le t$.

Define $G_{i,t}:=(W_t-\Lambda_{i,t})^{-1}$, then combining \eqref{eq:eqderres} with \eqref{eq:matchar} and using that $Z_i+\ii\eta_i=\Lambda_i$, we get (in the r.h.s. we use the notation $G_i=G_{i,t}$ for simplicity):
\begin{equation}
    \label{eq:fulleqaa}
    \begin{split}
        \dif \braket{G_{1,t}AG_{2,t}B}&=\sum_\alpha\partial_\alpha\braket{G_1AG_2B}\frac{\dif B_\alpha}{\sqrt{n}}+\braket{G_1AG_2B}\dif t \\
        &\quad+2\braket{G_1AG_2E_1}\braket{G_2BG_1E_2}\dif t+2\braket{G_1AG_2E_2}\braket{G_2BG_1E_1}\dif t \\
        &\quad+ 2\braket{(G_1-M_1)E_1}\braket{G_1AG_2BG_1E_2}\dif t \\
        &\quad + 2\braket{(G_1-M_1)E_2}\braket{G_1AG_2BG_1E_1}\dif t \\
        &\quad+ 2\braket{(G_2-M_2)E_1}\braket{G_2BG_1AG_2E_2}\dif t \\
        &\quad+ 2\braket{(G_2-M_2)E_2}\braket{G_2BG_1AG_2E_1}\dif t.
    \end{split}
\end{equation}
Note that the careful choice of the characteristic ODE~\eqref{eq:matchar} guarantees
that the last line of~\eqref{eq:eqderres} and the leading terms of the third and fourth lines  of~\eqref{eq:eqderres} cancel.

\begin{remark}
    Beside the spectral parameter $\eta_t$ our characteristic flow
    also moves the additional  parameter $z_t$;   previous
    applications of the flow method operated only with moving  the spectral parameter.
    The main cancellation concerns
    the   $\langle \ii\eta GBGAG\rangle$ term in~\eqref{eq:eqderres}
    and this could be achieved by the correct  choice of the $\eta_t$-flow alone.
    However, our choice of time dependent $z_t$ also cancels the $\langle ZGBGAG\rangle$
    term automatically, saving us from additional work to estimate it using the
    off-diagonality of $Z$. The canonical  form~\eqref{eq:matchar} of the flow seems the most efficient,
    and it also  gives a hint how to find the best characteristic flow for much more general ensembles.
\end{remark}

The main result of this section is the following Proposition~\ref{pro:charaveinfin} below
that shows how a two-resolvent
local law at large $\eta$ can be propagated to smaller $\eta$.   In Part 1 we formulate
a general estimate, which will then be improved in Part 2 for the special case
when $|z_1-z_2|$ is relatively big compared with $\eta$. Note that both results
are conditional: given a (small) $\eta$ where we want to prove the local law,
we construct a (larger) $\eta_0$ such that after time $T$ the characteristic  flow \eqref{comp1} with initial
condition $\eta_0$ ends up precisely at our target  $\eta=\eta_T$.
Assuming the local law at $\eta_0$, this proposition proves the local law at  $\eta=\eta_T$.

Before stating Proposition~\ref{pro:charaveinfin}, in the following lemma we prove
a simple property of the characteristics in \eqref{comp1}.
Recall that $T^*=T^*(z, \eta)$ is the maximal time so that the $\eta_t$-flow with initial condition $z, \eta$ does
not cross the real axis.

\begin{lemma}
    \label{lem:charpropert}
    Fix $n$-independent \nc
    $\tau,\omega_1>0$, and pick any $|\eta|>0$, $0< T\le n^{-\omega_1}$, $|z|\le 1-\tau$. Then, there exist
    an initial data $\eta_0,z_0$, with $T^*(z_0, \eta_0)\ge T$,
    such that the solutions to~\eqref{comp1}  with these initial data satisfy $\eta_T=\eta$, $z_T=z$ after time $T$. We have $|\eta_0|\gtrsim T$ and $|z_0|\le 1-\tau/2$.
\end{lemma}
\begin{proof}
    This lemma is a simple consequence of the fact that the flow $t\to \eta_{i,t}$, given in the last equation of \eqref{eq:timeev}, moves toward the real axis with a linear speed that is well separated from zero. To establish this fact, we use that, since the time $T\ll 1$ is short, the right hand side of the second  equation of \eqref{comp1} is bounded and initially $|z_0|$ is well separated from $1$, we see that $|z_t|\le 1-\tau/2$ for all $t\le T\le n^{-\omega_1}$, in
    particular we stay in the bulk regime for all  $t\in[0,T]$. This guarantees that $|\Im m^z| \ge c$, with some small $n$-independent constant $c>0$ along the whole solution up to time $T$. Thus the right hand side of the first equation in~\eqref{comp1} is negative, well separated away from zero for all $t\in [0,T]$. This establishes the linear speed of $\eta_{i,t}$. Therefore, if we are given some $\eta$, $z$, and $T$, with $T\ll 1$ and $z$ in the bulk of the spectrum of $X$, then by running this approximately linear flow backward in time we can find initial values $z_0,\eta_0$ as required in the lemma.
\end{proof}

\begin{proposition}
    \label{pro:charaveinfin}
    Fix small $n$-independent constants $\epsilon,\tau, \omega_d>0$, and, for $i=1,2$, let
    \[
        \Lambda_{i,0}=\left(
        \begin{matrix}
                \ii\eta_{i,0}      & z_{i,0}       \\
                \overline{z_{i,0}} & \ii\eta_{i,0}
            \end{matrix}
        \right)
    \]
    with $|z_{i,0}| \le 1- \tau$, $|z_{1,0}-z_{2,0}| \le n^{-\omega_d}$ and $\eta_{i,0}\ne 0$. %
    Let $\Lambda_{i,t}$ be the solution of \eqref{eq:matchar}  with initial condition $\Lambda_{i,0}$
    for any $t\le T^*(z_{i,0}, \eta_{i,0})$.
    Set $G_{i,t}:=(W_t-\Lambda_{i,t})^{-1}$, and denote the deterministic approximation of
    $G_{1,t} A G_{2,t}$ by  $M_{12,t}^A := M_{12}^A(z_{1,t},\ii \eta_{1,t}, z_{2,t}, \ii \eta_{2,t})$
    as  given in \eqref{eq:defM12A}. Then we have the following statements:

    {\bf Part 1.}
    Choose $\eta_{i,0}$ such that $|\eta_{i,0}|\le n^{-\omega_1}$ for some $\omega_1>0$.
    Assume that
    for any arbitrary small $\xi>0$ it holds
    \begin{equation}
        \label{eq:incond}
        \big|\braket{G_{1,0}AG_{2,0}B-M_{12,0}^AB}\big|\le \frac{n^\xi}{n\eta_{*,0}\sqrt{|\eta_{1,0}\eta_{2,0}|}}
    \end{equation}
    with very high probability, uniformly for
    matrices $A, B$ with $\norm{A}+\norm{B}\le 1$. Then
    \begin{equation}
        \begin{split}
            \label{eq:largeeta}
            \big|\braket{G_{1,T}AG_{2,T}B-M_{12,T}^AB}\big|\le\frac{n^{2\xi}}{n\eta_{*,T}\sqrt{|\eta_{1,T}\eta_{2,T}|}},
        \end{split}
    \end{equation}
    with very high probability, uniformly in  $T \le \min_i T^*(z_{i,0}, \eta_{i,0})$
    such that $\eta_{*,T}:=\min\{ |\eta_{1, T}|, |\eta_{2, T}|\}
        \ge n^{-1+\epsilon}$ and uniformly in matrices $A,B$ with $\norm{A}+\norm{B}\le 1$.

        {\bf Part 2.} Choose $\eta_{i,0}$ such that $|\eta_{i,0}|\lesssim |z_{1,0}-z_{2,0}|^2$.   %
    Assume that for any arbitrary small $\xi>0$ it holds
    \begin{equation}
        \label{eq:incondsmalleta}
        \big|\braket{G_{1,0}AG_{2,0}B-M_{12,0}^AB}\big|\le n^\xi \mathcal{E}(n,\eta_{1,0},\eta_{2,0}),
    \end{equation}
    with very high probability uniformly in matrices $A, B$ with $\norm{A}+\norm{B}\le 1$, for some given error function
    $ \mathcal{E}(n,\eta_{1,0},\eta_{2,0})\lesssim (n\eta_{*,0}^2)^{-1}$. Then we have
    \begin{equation}
        \label{eq:smalleta}
        \big|\braket{G_{1,T}AG_{2,T}B-M_{12,T}^AB}\big|\le \frac{n^{2\xi}}{n\sqrt{|\eta_{1,T}\eta_{2,T}|\eta_{*,T}(\eta_T^*+|z_{1,T}-z_{2,T}|^2)}}+\frac{n^{3\xi}}{(n\eta_{*,T})^{3/2}\sqrt{|\eta_{1,T}\eta_{2,T}|}}+n^{2\xi} \mathcal{E}(n,\eta_{1,0},\eta_{2,0}),
    \end{equation}
    with very high probability  uniformly in  $T \le \min_i T^*(z_{i,0}, \eta_{i,0})$
    such that $\eta_{*,T}\ge n^{-1+\epsilon}$  and  uniformly in matrices $A,B$ with $\norm{A}+\norm{B}\le 1$.
    Here $\eta_T^*:=\max\{ |\eta_{1,T}|, |\eta_{2,T}|\}$.
\end{proposition}
Note that the difference between \eqref{eq:largeeta} and \eqref{eq:smalleta} lies in the fact that in  \eqref{eq:smalleta}
the leading error term is smaller. However,  this bound is a genuine  improvement compared to
\eqref{eq:largeeta} only for $\eta_T^*\lesssim |z_{1,T}-z_{2,T}|^2$.
The bounds \eqref{eq:largeeta}, \eqref{eq:smalleta} agree for
$\eta_T^*\gtrsim |z_{1,T}-z_{2,T}|^2$ and $|\eta_{1,T}|\sim |\eta_{2,T}|$.

\begin{proof}[Proof of Proposition~\ref{pro:charaveinfin}]

    At the
    beginning the proofs of both parts
    will be presented together
    and then we will specialize to the two cases later. In the sequel we often omit the $t$-dependence
    and use $G_i := G_{i,t}= (W_t-\Lambda_{i,t})^{-1}$, $M_i:= M^{z_{i,t}}(\ii \eta_{i,t})$, and a similar definition for $m_i$, $u_i$.
    Using the Schwarz inequality and the Ward identity $GG^*=\Im G/\eta$ we have
    \begin{equation}
        \label{eq:reflat}
        \begin{split}
            |\braket{G_1AG_2BG_1E_2}|&\le  \braket{G_1AG_2G_2^*A^*G_1^*}^{1/2}\braket{BG_1E_2E_2^*G_1^*B^*}^{1/2} \\
            & \le\frac{\braket{\Im G_1A\Im G_2 A^*}^{1/2}\braket{\Im G_1B^*B}^{1/2}}{|\eta_{1,t}|\sqrt{|\eta_{2,t}|}}\prec \frac{|\braket{\Im G_1 A\Im G_2 A^*}|^{1/2}\norm{B} }{\sqrt{\eta_{*,t}|\eta_{1,t}\eta_{2,t}|}}.
        \end{split}
    \end{equation}
    We point out that in the second inequality we used that
    \[
        \braket{BG_1E_2E_2^*G_1^*B^*}\le \lVert E_2\rVert^2 \braket{BG_1G_1^*B^*}\le \braket{BG_1G_1^*B^*}= \frac{\braket{\Im G_1B^*B}}{\eta_{1,t}},
    \]
    where the last equality follows by Ward identity, and in the last inequality of \eqref{eq:reflat} we used
    \[
        \braket{\Im G_1B^*B}\le \lVert BB^*\rVert \braket{\Im G_1} \lesssim \lVert B\rVert^2 \big( \braket{\Im M_1}+| \braket{\Im G_1-\Im M_1}|\big)\prec  \lVert B\rVert^2+\frac{ \lVert B\rVert^2}{n|\eta_{1,t}|}\lesssim \lVert B\rVert^2,
    \]
    which follows by the imaginary part of the local law $|\braket{G_1-M_1}|\prec 1/(n|\eta_{1,t}|)$ from \eqref{eq:singlegllaw}. We remark that the key point in the estimate \eqref{eq:reflat} is that the $G_1AG_2$ block is separated from the rest
    and the estimate reduces a trace with three $G$'s to one with two $G$'s
    (up to imaginary part), i.e. it is of the similar form as the left hand side
    in Proposition~\ref{pro:charaveinfin}. This allows us to have a closed Gronwall-like inequality for products with two resolvents. We point out that this philosophy will be often used within the proof of Proposition~\ref{pro:charaveinfin} (see e.g. \eqref{eq:estquadvar} below).

    Using the bound \eqref{eq:reflat} in \eqref{eq:fulleqaa}, together with the single resolvent local law \eqref{eq:singlegllaw}, we get
    \begin{equation}
        \begin{split}
            \label{eq:fulleqaaaftsimpl}
            \dif \braket{G_{1,t}AG_{2,t}B}&=\sum_\alpha\partial_\alpha\braket{G_1AG_2B}\frac{\dif B_\alpha}{\sqrt{n}}+\braket{G_1AG_2B}\dif t \\
            &\quad+\braket{\mathcal{S}[G_1AG_2]G_2BG_1}\dif t+\mathcal{O}\left(\frac{n^\xi |\braket{\Im G_1 A\Im G_2A^*}|^{1/2}}{n\eta_{*,t}^{3/2}\sqrt{|\eta_{1,t}\eta_{2,t}|}}\right)\dif t.
        \end{split}
    \end{equation}
    Here we estimated the error in terms of $|\braket{\Im G_1 A\Im G_2A^*}|^{1/2}$ and ignored $\norm{B}$
    for brevity (recall that $\norm{A}+\norm{B}\le 1$). Note that to go from \eqref{eq:fulleqaa}
    to \eqref{eq:fulleqaaaftsimpl} we also used that from \eqref{defS} we have
    \[
        2\braket{G_1AG_2E_1}\braket{G_2BG_1E_2}+2\braket{G_1AG_2E_2}\braket{G_2BG_1E_1}=\braket{\mathcal{S}[G_1AG_2]G_2BG_1}.
    \]

    We now consider the stochastic term in \eqref{eq:fulleqaaaftsimpl}.
    We first estimate its quadratic variation and then use the Burkholder--Davis--Gundy (BDG) inequality to conclude
    a bound with very high probability.  Let $\mathcal{F}_t$ denote the filtration generated
    by $(B_s)_{0\le s\le t}$, with $B_t$ from \eqref{eq:OUaa}.
    The quadratic variation
    process of $n^{-1/2}\sum_\alpha\partial_\alpha\braket{G_1AG_2B}\dif B_\alpha$ is given by
    \begin{equation}
        \label{eq:quadvar}
        \begin{split}
            &\E\left[ \frac{1}{n}\sum_{\alpha,\beta}\big[\braket{G_1\Delta^\alpha G_1AG_2B}+\braket{G_1AG_2\Delta^\alpha G_2 B}\big]\cdot \big[\overline{\braket{G_1\Delta^\beta G_1AG_2B}}+\overline{\braket{G_1AG_2\Delta^\beta G_2 B}}\big] \dif B_\alpha\overline{\dif B_\beta}\Bigg| \mathcal{F}_t\right] \\
            &=\frac{1}{n^2}\sum_{ab}\big[(G_1AG_2BG_1)_{ab}+(G_2BG_1AG_2)_{ab}\big]\cdot
            \big[\overline{(G_1AG_2BG_1)_{ab}}+\overline{(G_2BG_1AG_2)_{ab}}\big]\, \dif t \\
            &=\frac{1}{n^2}\Bigg[\braket{G_1AG_2BG_1 E_i G_1^*B^* G_2^*A^*G_1^*E_j}+
                \braket{G_2BG_1AG_2 E_iG_2^* A^* G_1^*B^*G_2^*E_j}
                +2\Re \braket{G_1AG_2BG_1E_iG_2^*A^*G_1^*B^*G_2^*E_j}\Bigg]\dif t,
        \end{split}
    \end{equation}
    where $\alpha, \beta$ denote index pairs, and $(\Delta^{ab})_{cd}=\delta_{ac}\delta_{bd}$.
    In the last line of \eqref{eq:quadvar} the indices $i,j$ are summed over two pairs $(i,j)\in \{(1,2),(2,1)\}$.
    Similarly to \eqref{eq:reflat}, by Schwarz inequality (performed by separating the block $G_1AG_2$ from the rest),
    it is easy to see that the quadratic variation is bounded by a multiple of (recall that $\norm{A}+\norm{B}\le 1$)
    \begin{equation}
        \label{eq:estquadvar}
        \frac{|\braket{\Im G_1A\Im G_2A^*}|}{n^2\eta_{*,t}^2|\eta_{1,t}\eta_{2,t}|}.
    \end{equation}
    Here we also used that $\norm{G_i}\le |\eta_{i,t}|^{-1}$ deterministically, and that $|\braket{G_i}|\le 1$ with very high probability by the single resolvent local law \eqref{eq:singlegllaw}. Then by the martingale inequality \cite[Appendix B, Eq. (18)]{10.1137/1.9780898719017}, with $c=0$ for continuous martingales, we conclude
    \begin{equation}
        \label{eq:BDG}
        \sup_{0\le t \le T}\left|\int_0^t \sum_\alpha\partial_\alpha\braket{G_1AG_2B}\frac{\dif B_\alpha}{\sqrt{n}}\right| \lesssim  n^\xi \left(\int_0^T \frac{|\braket{\Im G_1A\Im G_2A^*}|}{n^2\eta_{*,t}^2|\eta_{1,t}\eta_{2,t}|}\, \dif t\right)^{1/2},
    \end{equation}
    with very high probability. Recall that $0<T\ll 1$ is fixed.

    Combining \eqref{eq:fulleqaaaftsimpl} and \eqref{eq:BDG}, we get the integral equation
    \begin{equation}
        \label{eq:fulleqanewaa}
        \begin{split}
            \braket{G_{1,T}AG_{2,T}B}&=\braket{G_{1,0}AG_{2,0}B}+\int_0^T\braket{\mathcal{S}[G_{1,t}AG_{2,t}]G_{2,t}BG_{1,t}}\,\dif t+\int_0^T\braket{G_{1,t}AG_{2,t}B}\,\dif t \\
            &\quad+\mathcal{O}\left(n^\xi \left(\int_0^T \frac{|\braket{\Im G_{1,t}A\Im G_{2,t}A^*}|}{n^2\eta_{*,t}^2|\eta_{1,t}\eta_{2,t}|}\, \dif t\right)^{1/2}+n^\xi \int_0^T \frac{|\braket{\Im G_{1,t}A\Im G_{2,t}A^*}|^{1/2}}{n\eta_{*,t}^{3/2}\sqrt{|\eta_{1,t}\eta_{2,t}|}}\, \dif t\right),
        \end{split}
    \end{equation}
    with very high probability, ignoring the $B$-error terms.
    We now start distinguishing the proof of  \eqref{eq:largeeta} and \eqref{eq:smalleta}.

    \subsection*{Proof of Part 1.}   Recall that the \emph{stability operator},
    defined as $\mathcal{B}_{12}:= 1- M_1 \mathcal{S} [\cdot ]
        M_2$, with $M_i= M^{z_i}(\ii\eta_i)$, acts on the Hilbert space of $(2n)\times (2n)$ matrices
    equipped with the usual Hilbert-Schmidt scalar product.
    It will play a key role in the analysis, in fact we will need to compute the inverse of its adjoint
    $[\mathcal{B}_{12}^*]^{-1}$.
    In Appendix~\ref{app:evect} we collected all precise
    information on the eigenvalues and left/right eigenvectors\footnote{These are actually matrices, but
        we will call them eigenvectors since we view them as elements of the vector space of $(2n)\times (2n)$  matrices.}
    of $\mathcal{B}_{12}$, giving immediately
    the spectral data of $\mathcal{B}_{12}^*$ as well.
    First, the $2\times 2$ block structure of $M_i$ and $\mathcal{S}$ shows that $\mathcal{B}_{12}$
    and $\mathcal{B}_{12}^*$ are just the identity on the $4n^2-4$ dimensional subspace
    of block traceless matrices. So effectively we need to understand $\mathcal{B}_{12}$
    on the four dimensional  subspace of block constant matrices that is invariant
    for both $\mathcal{B}_{12}$  and $\mathcal{B}_{12}^*$.
    The main point
    is that in our relevant regime,  $|\eta_i|\le n^{-\omega_2}$,  $|z_i|\le 1-\tau$ and  $|z_1-z_2|\le n^{-\omega_d}$,
    only one simple eigenvalue, denoted by $\beta_-$, is very small,
    a second eigenvalue, denoted by $\beta_+$ is well separated away from zero,
    \begin{equation}\label{betam}
        |\beta_-| \sim |z_1-z_2|^2 + \eta_1 + \eta_2, \qquad |\beta_+|\sim 1,
    \end{equation}
    see  \eqref{eq:lowbeigva},
    and all the other eigenvalues are 1
    with multiplicity $4n^2-2$.
    The left and right eigenvectors corresponding to $\beta_\pm$,
    denoted by $L_\pm=L_{12,\pm}$ and $R_\pm= R_{12,\pm}$,
    are block constant matrices, defined to
    satisfy\footnote{Here we deliberately use the convention that a left eigenvector $L$ is \emph{defined}
        such that $L^*$ is the right eigenvector of the adjoint operator.
        This convention will simplify many formulas below. }
    \begin{equation}\label{LR}
        \mathcal{B}_{12}[R_{\pm}]=\beta_\pm R_{\pm}, \qquad
        \mathcal{B}_{12}^*[L_{\pm}^*]=\overline{\beta_\pm}L_{\pm}^*,
    \end{equation}
    noting that
    the two nontrivial eigenvalues  of $\mathcal{B}_{12}^*$ are $\overline{\beta_\pm}$
    (here we dropped the 12 indices).
    Explicit formulas are given in Appendix~\ref{app:evect} but they  are largely  irrelevant for us,
    here we only remark that it is possible
    to choose  the  normalization such that $\| L_\pm\|\sim 1$, $\| R_\pm\|\sim 1$ and $|\langle R_-^* L_-^*\rangle|\sim 1$.
    In fact, the spectral data
    corresponding to $\beta_+$, as well as eigenvectors
    corresponding to the trivial eigenvalues 1
    will not be necessary for our main argument.
    Also notice that the eigenvalues of $\mathcal{B}_{12}$
    and $\mathcal{B}_{21}$ are the same (see~\eqref{eq:evalues} in Appendix~\ref{app:evect}), but their
    eigenvectors are not identical, e.g. $R_{12,\pm}\ne R_{21,\pm}$.

    The inverse of  $\mathcal{B}_{12}^*$ can be computed by its spectral decomposition,
    separating the one dimensional (non-orthogonal) spectral projection
    $\Pi_-$ %
    corresponding
    to the small eigenvalue $\bar\beta_-$ from the spectral projection $\Pi$ corresponding to all the other
    eigenvalues.  Explicitly,
    any matrix $Q$ can be decomposed as
    \begin{equation}
        \label{eq:goodspecdec}
        Q=\Pi_-[Q]+\Pi[Q], \qquad \Pi_- [Q]:= \frac{\braket{R_-^*Q}}{\braket{R_-^*L_-^*}}L_-^*, \qquad
        \Pi[Q]:=\frac{1}{2\pi\ii}\oint \frac{1}{z-\mathcal{B}_{12}^*}[Q]  \,\dif z,
    \end{equation}
    with the integral over a contour which encircles $\overline{\beta_+}$ and $1$, excludes $\overline{\beta_-}$
    and is well separated (order one away) from all eigenvalues.
    The resolvent $(z-\mathcal{B}_{12}^*)^{-1}$ can be viewed only on the four dimensional invariant
    subspace of block constant matrices, hence its norm is bounded as all four eigenvalues of $z-\mathcal{B}_{12}^*$ are
    well separated away from zero when $z$ is on the contour. Thus both $\Pi$ and $\Pi_-$ are bounded;
    \begin{equation}\label{Pi}
        \| \Pi\| + \| \Pi_-\|\lesssim 1.
    \end{equation}
    We clearly have that $(\mathcal{B}_{12}^*)^{-1}$ is bounded on the range of $\Pi$, i.e. for any matrix $Q$
    \begin{equation}
        \label{eq:fundboundA}
        \norm{(\mathcal{B}_{12}^*)^{-1})[\Pi[Q]]}\lesssim \norm{Q}.
    \end{equation}
    Since $(\mathcal{B}_{12}^*)^{-1}$ on the complementary  one dimensional spectral subspace $\mbox{Span}(L_-^*)$
    can be very large, of order $|\beta_-|^{-1}$, this subspace requires a separate treatment.
    Note that all these bounds trivially extend to $(\mathcal{B}_{12}^*)^{-1}$ viewed on the space of all $(2n)\times (2n)$
    matrices due to the invariance of the space of block traceless matrices.

    Owing to the time evolution, we actually need the spectral data for
    $\mathcal{B}_{12,t}[\cdot]:=(1-M_{1,t}\mathcal{S}[\cdot]M_{2,t})$ for small times $t\ge 0$.
    Let $L_{\pm,t}$ and $R_{\pm,t}$, be the left and right eigenvectors of $\mathcal{B}_{12,t}$ with
    corresponding eigenvalues $\beta_{\pm,t}$, then by \eqref{eq:timeev} it readily follows that $L_{\pm,t}= L_{\pm,0}$, $R_{\pm,t}=e^tR_{\pm,0}$ for any $t\ge 0$. Since the difference between the eigenvectors at
    time $0$ and time $t$  amounts to a simple rescaling by an irrelevant factor $e^t=1+\mathcal{O}(t)$,
    we can use the zero-time eigenvectors $L_\pm:=L_{\pm,0}$, $R_\pm:=R_{\pm,0}$
    for all later times.
    The eigenvalues $\beta_{\pm,t}$ depend on $t$ nontrivially   but
    smoothly and the time dependent version of~\eqref{betam} holds, in particular
    $\beta_{-,t}$  still remains well separated from the rest of the spectrum if $t\ll 1$.

    After all these preparations, we first handle  \eqref{eq:largeeta} in the case when either $A^*$ or $B^*$
    lies in the range of $\Pi$. The proof  relies on the following technical lemma (whose proof is postponed to the Appendix~\ref{app:techres}):
    \begin{lemma}
        \label{pro:imppro}
        Fix any small $\tau,\epsilon>0$, and fix $z_i,\eta_i$, with $i=1,2$, such that $|z_i|\le 1-\tau$,
        $|\eta_i|\ge n^{-1+\epsilon}$. Let $A,B$ be any deterministic matrices with $\norm{A}+\norm{B}\le1$
        and such that at least one among $\norm{[\mathcal{B}_{21}^{-1}]^*[A^*]}$, $\norm{[\mathcal{B}_{12}^{-1}]^*[B^*]}$
        is bounded by an ($n$,$\eta$)--independent constant, where
        $\mathcal{B}_{12}[\cdot]:=1-M_1\mathcal{S}[\cdot]M_2$, $M_i= M^{z_i}(\ii\eta_i)$. Then it holds
        \begin{equation}
            |\braket{(G_1AG_2-M_{12}^A)B}|\prec \frac{1}{n\eta_*\sqrt{|\eta_1\eta_2|}},
        \end{equation}
        uniformly in $\eta_*\ge n^{-1+\epsilon}$.
    \end{lemma}

    Given a deterministic matrix $A$, we split $A^*$ as in \eqref{eq:goodspecdec}, then,
    by \eqref{eq:fundboundA}, it clearly follows that $\norm{(\mathcal{B}_{12}^{-1})^*[\Pi[A^*]]}$ is bounded.
    In particular, by \eqref{eq:goodspecdec} and  bilinearity of~\eqref{eq:largeeta}, we obtain
    that Lemma~\ref{pro:imppro} proves \eqref{eq:largeeta} in all cases except for when $A^*$ and $B^*$ are in the range of
    the corresponding $\Pi_-$.
    Thus the remainder of the proof focuses on the case
    $A=L_{21,-}$ and $B=L_{12,-}$, where we recall from~\eqref{LR} that $L_{12,-}$ is defined by $\mathcal{B}_{12}^*[L_{12,-}^*]=\overline{\beta_-}L_{12,-}^*$.
    From now on we introduce the shorthand notations
    $$ L_\pm:=L_{12,\pm},\qquad L_\pm':=L_{21,\pm}.
    $$
    For definiteness we only consider the case %
    when $\eta_{1,t}\eta_{2,t}<0$. In this case we have (see  \eqref{eq:usetaylexp}):
    \[
        L_-=(1+\mathcal{O}(|z_{1,t}-z_{2,t}|))I+\mathcal{O}(|z_{1,t}-z_{2,t}|)E_-,
    \]
    and an analogous relation for $L_-'$. The case $\eta_{1,t}\eta_{2,t}>0$ is completely analogous and so omitted.

    Define the stopping time
    \begin{equation}
        \label{eq:tau1}
        \tau_1:=\inf\left\{t\ge 0 : \left|\braket{(G_{1,t}(\ii\eta_{1,t})L_-'G_{2,t}(\ii\eta_{2,t})-M_{12,t}^{L_-'})L_-}\right|= \frac{n^{2\xi}}{n\eta_{*,t}\sqrt{|\eta_{1,t}\eta_{2,t}|}}\right\}\wedge T.
    \end{equation}
    Here $\xi\le (\epsilon\wedge \omega_d)/10$ with $\epsilon>0$ such that $\eta_{*,t}\ge n^{-1+\epsilon}$ and $\omega_d$ such that $|z_{1,t}-z_{2,t}|\lesssim n^{-\omega_d}$ for any $t\ge 0$. We remark that $L_-, L_-'$ in \eqref{eq:tau1} are independent of time.

    Define
    \[
        Y_t:=\braket{(G_{1,t}(\ii\eta_{1,t})L_-'G_{2,t}(\ii\eta_{2,t})-M_{12,t}^{L_-'})L_-}.
    \]
    In order to study the time evolution of $Y_t$, we need to understand how $M_{12,t}^{L_-'}$ evolves in time. This is explained in the following lemma, whose proof is postponed to Appendix~\ref{app:techres} (see \eqref{eq:12derm}).
    \begin{lemma}
        \label{lem:mder}
        For any $A,B\in \mathbf{C}^{2n\times 2n}$  It holds that
        \begin{equation}
            \label{eq:domdera}
            \partial_t \braket{M_{12,t}^AB} = \braket{M_{12,t}^AB}+\braket{\SS[M_{12,t}^A]M_{21,t}^B}.
        \end{equation}
    \end{lemma}
    Then, choosing $A=L_-'$, $B=L_-$ in \eqref{eq:fulleqanewaa} and \eqref{eq:domdera},  we obtain
    \begin{equation}
        \label{eq:intermstep}
        \begin{split}
            Y_T&=Y_0+\int_0^T Y_t\,\dif t+\int_0^T\big[\braket{\mathcal{S}[G_{1,t}L_-'G_{2,t}]G_{2,t}L_-G_{1,t}}-\braket{\mathcal{S}[M_{12}^{L_-'}]M_{21}^{L_-}}\big]\,\dif t\\
            &\quad+\mathcal{O}\left(n^\xi \left(\int_0^T \frac{|\braket{\Im G_{1,t}L_-'\Im G_{2,t}(L_-')^*}|}{n^2\eta_{*,t}^2|\eta_{1,t}\eta_{2,t}|}\, \dif t\right)^{1/2}+n^\xi \int_0^T \, \frac{|\braket{\Im G_{1,t}L_-'\Im G_{2,t}(L_-')^*}|^{1/2}}{n\eta_{*,t}^{3/2} \sqrt{|\eta_{1,t}\eta_{2,t}|}}\dif t\right).
        \end{split}
    \end{equation}

    Next, to estimate the last term in the first line of \eqref{eq:intermstep} we rely on the following lemma, whose proof is postponed to Appendix~\ref{app:techres}.
    \begin{lemma}
        \label{lem:newlem1}
        Denote $M_{12}^{L_-'}=M_{12,t}^{L_-'}$, $G_i=G_{i,t}(\ii\eta_{i,t})$, then it holds
        \begin{equation}
            \label{eq:addinfneed}
            \braket{\mathcal{S}[G_1L_-'G_2]G_2L_-G_1}=\braket{\mathcal{S}[M_{12}^{L_-'}]M_{21}^{L_-}}+2\braket{M_{12}^I}Y_t+\mathcal{O}\left((1+|Y_t|)|Y_t|+\frac{n^\xi}{n\eta_{*,t}^2\sqrt{|\eta_{1,t}\eta_{2,t}|}}\right).
        \end{equation}
    \end{lemma}

    Combining \eqref{eq:intermstep} with \eqref{eq:addinfneed}, we conclude
    \begin{equation}
        \label{eq:fulleqanewaaab}
        \begin{split}
            Y_t &=Y_0+2\int_0^t\braket{M_{12,s}^I}Y_s\,\dif s+\mathcal{O}\left(\int_0^t\big(1+|Y_s|\big)|Y_s|\,\dif s+\frac{n^\xi}{n\eta_{*,t}\sqrt{|\eta_{1,t}\eta_{2,t}|}}\right),
        \end{split}
    \end{equation}
    where we used that  $|\braket{\Im G_{1,t}L_-'\Im G_{2,t}(L_-')^*}|\le \eta_{*,t}^{-1}$ with very high probability, by Schwarz inequality and the single resolvent local law \eqref{eq:singlegllaw}, to estimate the error terms in the second line of \eqref{eq:intermstep}. Finally, using that, by the definition of the stopping time $\tau_1$, we have $|Y_s|\le n^{2\xi}(n\eta_{*,s}\sqrt{|\eta_{1,s}\eta_{2,s}|})^{-1}$, for $0\le s\le \tau_1$, to estimate $(1+|Y_s|)|Y_s|$ (here we also used that $\xi\le \epsilon/10$), we conclude
    \begin{equation}
        \label{eq:fulleqanewaaa}
        \begin{split}
            Y_t &=Y_0+2\int_0^t\braket{M_{12,s}^I}Y_s\,\dif s+\mathcal{O}\left(\frac{n^\xi}{n\eta_{*,t}\sqrt{|\eta_{1,t}\eta_{2,t}|}}\right).
        \end{split}
    \end{equation}

    We now precisely estimate the deterministic term $\braket{M_{12,s}^I}$ in \eqref{eq:fulleqanewaaa} (the proof is postponed to Appendix~\ref{app:techres}):
    \begin{lemma}
        \label{lem:m12bound}
        Let $m_{i,t}:=m^{z_{i,t}}(\ii\eta_{i,t})$, $u_{i,t}:=u^{z_{i,t}}(\ii\eta_{i,t})$, then we have\begin{equation}
            \label{eq:explM12}
            \begin{split}
                \big|\braket{M_{12,t}^I}\big|\le \frac{a_0}{b_0-a_0t}\big(1+\mathcal{O}\left(|\eta_{1,0}|+|\eta_{2,0}|+|z_{1,0}-z_{2,0}|^2+t\right)\big),
            \end{split}
        \end{equation}
        where
        \[
            a_0:=2-|z_{1,0}|^2-|z_{2,0}|^2, \qquad\quad b_0:=|z_{1,0}-z_{2,0}|^2+|\eta_{1,0}|\sqrt{1-|z_{1,0}|^2}+|\eta_{2,0}|\sqrt{1-|z_{2,0}|^2}.
        \]
    \end{lemma}

    In order to obtain an estimate for $Y_t$ from \eqref{eq:fulleqanewaaa} we will use a Gronwall inequality. For this purpose we compute
    \begin{equation}
        \label{eq:intgronwall}
        \begin{split}
            \exp\left(2\int_s^t \big|\braket{M_{12,r}^I}\big|\,\dif r\right)\le \exp\left(2\int_s^t \left|\frac{a_0}{b_0-a_0r}\right|\,\dif r\right)=\left(\left|\frac{b_0-a_0s}{b_0-a_0t}\right|\right)^2&\sim\left(\frac{|z_{1,s}-z_{2,s}|^2+|\eta_{1,s}|+|\eta_{2,s}|}{|z_{1,t}-z_{2,t}|^2+|\eta_{1,t}|+|\eta_{2,t}|}\right)^2.
        \end{split}
    \end{equation}
    Here in the last step we used various elementary properties of $z_{i,t}$ and $\eta_{i,t}$ that follow from \eqref{eq:timeev}. In particular, we used that $\eta_{i,t}=\eta_{i,0}- \mathrm{sgn}(\eta_{i,0})c_it$, with $c_i$ given in \eqref{eq:defci}, and that $1- |z_{i,t}| \sim 1- |z_{i,t}| \sim 1- |z_{i,0}|$, as well as $|z_{1,t}-z_{2,t}|\sim |z_{1,s}-z_{2,s}|\sim |z_{1,0}-z_{2,0}|$, for the times we consider $0\le s\le t\le T\ll 1$. Combining these information we obtain
    \[
        b_0-a_0t=|z_{1,0}-z_{2,0}|^2+|\eta_{1,0}|\sqrt{1-|z_{1,0}|^2}+|\eta_{2,0}|\sqrt{1-|z_{2,0}|^2}-t[2-|z_{1,0}|^2-|z_{2,0}|^2]\sim |z_{1,t}-z_{2,t}|^2+|\eta_{1,t}|+|\eta_{2,t}|,
    \]
    which was used in the last step of \eqref{eq:intgronwall}. Note that in our regime $|z_{1,t}-z_{2,t}|\lesssim |\eta_{1,t}|+|\eta_{2,t}|$ for any $0\le t\le T$, hence the $|z_{1,t}-z_{2,t}|$ terms are negligible in \eqref{eq:intgronwall}. Using again the elementary properties of $z_{i,t}$ and $\eta_{i,t}$ described above, we note that
    \begin{equation}
        \label{eq:step121}
        \frac{|\eta_{1,s}|+|\eta_{2,s}|}{|\eta_{1,t}|+|\eta_{2,t}|}\lesssim 1+(t-s)\frac{(\sqrt{1-|z_{1,0}|^2}+\sqrt{1-|z_{2,0}|^2})}{|\eta_{1,t}|+|\eta_{2,t}|}\lesssim 1+(t-s)\frac{(\sqrt{1-|z_{1,0}|^2}+\sqrt{1-|z_{2,0}|^2})}{\eta_{*,t}}\lesssim \frac{\eta_{*,s}}{\eta_{*,t}},
    \end{equation}
    by $1- |z_{i,t}| \sim 1- |z_{i,t}| \sim 1- |z_{i,0}|$ and $|\eta_{i,s}|=|\eta_{i,t}|+(c_i+o(1))(t-s)$, with $c_i$ defined below \eqref{eq:timeev}. In particular, this also implies that
    \begin{equation}
        \label{eq:step122}
        \frac{|\eta_{1,s}|+|\eta_{2,s}|}{|\eta_{1,t}|+|\eta_{2,t}|}\lesssim \frac{\sqrt{|\eta_{1,s}\eta_{2,s}|}}{\sqrt{|\eta_{1,t}\eta_{2,t}|}}.
    \end{equation}
    Combining \eqref{eq:intgronwall}--\eqref{eq:step122}, we conclude
    \begin{equation}
        \label{eq:fin12}
        \exp\left(2\int_s^t \big|\braket{M_{12,r}^I}\big|\,\dif r\right)\lesssim \frac{\eta_{*,s}\sqrt{|\eta_{1,s}\eta_{2,s}|}}{\eta_{*,t}\sqrt{|\eta_{1,t}\eta_{2,t}|}}.
    \end{equation}

    Finally, using \eqref{eq:fin12}, by a simple Gronwall inequality, together with \eqref{eq:incond} to bound $Y_0$, we conclude
    \begin{equation}
        |Y_t|\lesssim \frac{n ^\xi}{n\eta_{*,t}\sqrt{|\eta_{1,t}\eta_{2,t}|}},
    \end{equation}
    with very high probability for any $t< \tau_1$. This proves that $\tau_1=T$ and concludes the proof of \eqref{eq:largeeta}.

    \subsection*{Proof of Part 2.}

    Since $|\eta_{i,0}|\lesssim |z_{1,0}-z_{2,0}|^2$ by assumption, and
    \[
        \begin{split}
            \eta_{i,t}&=\eta_{i,0}-(1-|z_{i,0}|^2)^{1/2}t+\mathcal{O}(t|\eta_{i,0}|), \\
            |z_{1,0}-z_{2,0}|^2&=e^t|z_{1,t}-z_{2,t}|^2=|z_{1,t}-z_{2,t}|^2(1+\mathcal{O}(t)),
        \end{split}
    \]
    we also have that $\eta_{i,t}^*\lesssim |z_{1,t}-z_{2,t}|^2$ for any $0\le t\le T$ in our perturbative regime $T\lesssim |z_{1,0}-z_{2,0}|^2 \le n^{-2\omega_d}$.

    Unlike in Part 1 (see \eqref{eq:goodspecdec}), we do not need to separate the spectral projection
    corresponding to the smallest eigenvalue $\beta_-$, this is because  the norm estimate~\eqref{Pi} combined with
    ~\eqref{betam} giving
    \begin{equation}\label{Bb}
        \norm{ (\mathcal{B}_{12}^*)^{-1}}\lesssim |\beta_-|^{-1}\lesssim |z_1-z_2|^{-2}
    \end{equation}
    is affordable in the regime
    of Part 2 when $|z_1-z_2|$ is relatively large.  Therefore we  can use a standard
    orthogonal decomposition in the space of block constant matrices
    and we use the orthonormal basis $\{I,E_-,\sqrt{2}F,\sqrt{2}F^*\}$
    in which the covariance operator $\mathcal{S}$ is particularly simple.
    We thus decompose any deterministic matrices $A$ as
    \begin{equation}
        \label{eq:decA}
        A=:\braket{A}I+\braket{AE_-}E_-+2\braket{AF^*}F+2\braket{AF}F^*+\mathring{A},
    \end{equation}
    with $\mathring{A}$ being defined by this formula. In particular,  $\mathring{A}$
    is just the orthogonal projection (with respect to the usual Hilbert-Schmidt scalar product)
    of $A$ onto the space of block traceless matrices, i.e. matrices whose all four blocks are traceless.
    Let $A,B\in \{E_-,I,F,F^*\}$, then we define
    \begin{equation}
        \begin{split}
            Y_t:&=\sup_{A,B\in \{E_-,I,F,F^*\}}|\braket{\big(G_{1,t}(\ii\eta_{1,t})AG_{2,t}(\ii\eta_{2,t})B-M_{12,t}^A(\ii\eta_{1,t},\ii\eta_{2,t})\big)B}|\\
            &\quad+\sup_{A,B\in \{E_-,I,F,F^*\}}[|\braket{\big(G_{1,t}(\ii\eta_{1,t})AG_{2,t}(-\ii\eta_{2,t})B-M_{12,t}^A(\ii\eta_{1,t},-\ii\eta_{2,t})\big)B}|.
        \end{split}
    \end{equation}

    Choose $\xi\le \epsilon/10$ and define the stopping time
    \[
        \tau_2:=\inf\left\{t\ge 0\, :\, Y_t= \frac{n^{2\xi}}{n\sqrt{\eta_{*,t}|\eta_{1,t}\eta_{2,t}|(|z_{1,t}-z_{2,t}|^2+\eta_t^*)}}+\frac{1}{\sqrt{n\eta_{*,t}}}\cdot\frac{n^{3\xi}}{n\eta_{*,t}\sqrt{|\eta_{1,t}\eta_{2,t}|}}+n^{2\xi}\mathcal{E}(n,\eta_{1,0},\eta_{2,0})\right\}\wedge T.
    \]
    Note that by the definition of $\tau_2$ and the assumption $|\mathcal{E}(n,\eta_{1,0},\eta_{2,0})|\lesssim (n\eta_{*,0}^2)^{-1}\lesssim (n\eta_{*,t}^2)^{-1}$  it follows that $Y_t\le n^{2\xi} (n\eta_{*,t}^2)^{-1}$ for any $t<\tau_2$. Here we also used that $\xi\le \epsilon/10$ (recall that $\eta_{*,t}\ge n^{-1+\epsilon}$ for any $t\ge 0$) so that $n^\xi/\sqrt{n\eta_{*,t}}\le 1$

    We now proceed similarly to \eqref{eq:addinfneed}--\eqref{eq:fulleqanewaaa}, we thus do not write all the details but only explain the main differences.
    By adding and subtracting the deterministic approximation of all the terms in  \eqref{eq:fulleqanewaa}, by using Lemma~\ref{lem:mder} to show that all the deterministic terms exactly cancel, and that $\norm{M_{12,t}^A}\lesssim 1/|z_{1,t}-z_{2,t}|^2$, for any $t\le \tau_2$, we obtain
    \begin{equation}
        \begin{split}
            Y_t&\le Y_0+C\int_0^t \left(\frac{1}{|z_{1,s}-z_{2,s}|^2}+Y_s\right)Y_s\,\dif s+n^\xi \left(\int_0^t \frac{Y_s}{n^2\eta_{*,s}^2|\eta_{1,s}\eta_{2,s}|}\,\dif s\right)^{1/2}+n^\xi\int_0^t\frac{Y_s^{1/2}}{n\eta_{*,s}^{3/2}\sqrt{|\eta_{1,s}\eta_{2,s}|}}\, \dif s \\
            &+\frac{n^\xi}{n\sqrt{\eta_{*,t}|\eta_{1,t}\eta_{2,t}|(|z_{1,t}-z_{2,t}|^2+\eta_t^*)}} \\
            &\le  Y_0+C\int_0^t \left(\frac{1}{|z_{1,s}-z_{2,s}|^2}+Y_s+ \frac{n^\xi}{\sqrt{n}\eta_{*,s}^{3/2}}\right)Y_s\,\dif s+\frac{n^\xi}{N\sqrt{\eta_{*,t}|\eta_{1,t}\eta_{2,t}|(|z_{1,t}-z_{2,t}|^2+\eta_t^*)}}\\
            &\quad+\frac{1}{\sqrt{n\eta_{*,t}}}\cdot\frac{n^{2\xi}}{n\eta_{*,t}\sqrt{|\eta_{1,t}\eta_{2,t}|}},
        \end{split}
    \end{equation}
    for some constant $C>0$. Note that in the last inequality we first used Schwarz inequality (here we also use the monotonicity of the characteristics $|\eta_{i,t}|\le |\eta_{i,s}|$ for $s\le t$)
    \[
        \begin{split}
            \frac{Y_s^{1/2}}{n\eta_{*,s}^{3/2}\sqrt{|\eta_{1,s}\eta_{2,s}|}}&\le  \frac{Y_s}{\sqrt{n}\eta_{*,s}^{3/2}}+\frac{1}{\sqrt{n\eta_{*,s}}}\cdot\frac{1}{N\eta_{*,s}^2\sqrt{|\eta_{1,s}\eta_{2,s}|}} \\
            \left(\int_0^t \frac{Y_s}{n^2\eta_{*,s}^2|\eta_{1,s}\eta_{2,s}|}\,\dif s\right)^{1/2}&\le \frac{1}{n^{3/4}\eta_{*,t}^{3/4}(\eta_{1,s}\eta_{2,s})^{1/4}} \left(\int_0^t \frac{Y_s}{\sqrt{n}\eta_{*,s}^{3/2}}\,\dif s\right)^{1/2} \\
            &\le \frac{1}{\sqrt{n\eta_{*,t}}}\cdot\frac{1}{n\eta_{*,t}\sqrt{|\eta_{1,t}\eta_{2,t}|}}+\int_0^t \frac{Y_s}{\sqrt{n}\eta_{*,s}^{3/2}}\,\dif s,
        \end{split}
    \]
    and then we used that $Y_s\le n^{2\xi}(n\eta_{*,s}^2)^{-1}$ for $s\le \tau_2$. Finally, by a simple Gronwall inequality, using that $T\lesssim |z_{1,0}-z_{2,0}|^2$ and the bound $Y_0\le n^\xi \mathcal{E}(n,\eta_{1,0},\eta_{2,0})$ from \eqref{eq:incondsmalleta} ,we conclude that
    \[
        Y_t\lesssim \frac{n^\xi}{n\sqrt{\eta_{*,t}|\eta_{1,t}\eta_{2,t}|(|z_{1,t}-z_{2,t}|^2+\eta_t^*)}}+\frac{1}{\sqrt{n\eta_{*,t}}}\cdot\frac{n^{2\xi}}{n\eta_{*,t}\sqrt{|\eta_{1,t}\eta_{2,t}|}}+n^\xi \mathcal{E}(n,\eta_{1,0},\eta_{2,0}),
    \]
    for any $t\le\tau_2$. This shows that $\tau_2=T$ and thus that \eqref{eq:smalleta} holds, completing the proof of Proposition~\ref{pro:charaveinfin}.

\end{proof}

\section{GFT: Proof of Theorem~\ref{theo:impll}}
\label{sec:GFT}

The goal of this section is to show that we can remove the Gaussian components added in Section~\ref{sec:normllaw}. Consider the Ornstein-Uhlenbeck flow
\begin{equation}
    \label{eq:OU2}
    \dif X_t=-\frac{1}{2}X_t \dif t+\frac{\dif B_t}{\sqrt{N}}.
\end{equation}
and define
\begin{equation}
    R_t:=\braket{G_{1}^tAG_2^tB-M_{12}^{A}B},
\end{equation}
with
\begin{equation}
    G_i^t := ( W_t - Z_i - \ii\eta_i )^{-1}, \qquad Z_i:=\begin{pmatrix}
        0 & z_i \\\ov{z_i}&0
    \end{pmatrix},\qquad W_t := \begin{pmatrix}
        0     & X_t \\
        X_t^* & 0
    \end{pmatrix},
\end{equation}
where  $t$ in \(G_i^t\) denotes time dependence and it
should not be confused with the transpose \(G_i^\intercal\). Here we recall that $A,B$ are generic deterministic square matrices of bounded norm. Also we mention that this
\(G_i^t\) is not the same as $G_{i,t} = (W_t - \Lambda_{i,t})^{-1}$
used in Section~\ref{sec:normllaw} since now both
spectral parameters $z$ and $\eta$ are time independent and only $W_t$ changes
with time.

Note that along the OU flow the first two moments of $X_t$ are preserved hence $M_{12}^{A_1}$ is independent of time. Our main technical result of this section is the following~\cref{gft prop}.
\begin{proposition}\label{gft prop}
    Let $A,B$ be arbitrary deterministic matrices with $\norm{A}+\norm{B}\le 1$, and let $z,\eta_1,\eta_2$ be spectral parameters with $\abs{z}\le 1-\tau$ and $\eta_\ast:=\abs{\eta_1}\wedge \abs{\eta_2}\ge n^{-1+\epsilon}$ for some fixed $\epsilon,\tau>0$. Then for any $\xi>0$ and any even \(p\ge 4\) it holds that
    \begin{equation}
        \label{eq:GFTboundgood}
        \abs{\dif \E \abs{R_t}^p} \lesssim  \sum_{k=4}^p \frac{n^{2-k/2+\xi}}{(n\eta_1\eta_2)^k}\E \abs{R_t}^{p-k}  + n^\xi\Bigl(1+\frac{1}{\sqrt{n}\eta_\ast}\Bigr) \E\biggl[\frac{\abs{R_t}^{p-1}}{n\eta_1\eta_2} + \abs{R_t}^{p-3}\Bigl(\frac{1}{n\eta_1\eta_2}\Bigr)^3\biggr].
    \end{equation}
\end{proposition}
From~\cref{gft prop} we obtain the following~\cref{pro:GFTf1} by integration:
\begin{proposition}
    \label{pro:GFTf1} Let $X$ be an i.i.d.\ matrix, and let $X_t$ the solution of the OU flow \eqref{eq:OUaa}, with initial data $X_0=X$. Then for any small $\tau,\epsilon>0$, for any $1\ge |\eta_i|\ge n^{-1+\epsilon}$, $\abs{z_i}\le 1-\tau$, and for any $p\in \mathbf{N}$, denoting $G_i^t:=(H_t-Z_i-\ii\eta_i)^{-1}$, it holds:
    \begin{equation}
        \label{eq:GFTfav}
        \begin{split}
            \Big(\mathbf{E}|\braket{(G_1^0AG_2^0-M_{12}^A)B}|^p\Big)^{1/p} &\lesssim e^{t \left(\frac{\eta^*+|z_1-z_2|^2}{\eta^*}+\frac{1}{\sqrt{n}\eta_*^{3/2}}\right)} \Bigg[\Big(\mathbf{E}|\braket{(G_1^tAG_2^t-M_{12}^A)B}|^p\Big)^{1/p}\\
            &\quad+\frac{n^\xi}{n\eta_*\eta^*}\left(\eta_*^{1/6}+n^{-1/10}+\left(\frac{\eta^*}{\eta^*+|z_1-z_2|^2}\right)^{1/4}\right)\Bigg],
        \end{split}
    \end{equation}
    for any $0\le t\lesssim 1$ and any small $\xi>0$, where $A,B$ are deterministic matrices with $\|A\|+\|B\|\le 1$, and $\eta_*:=|\eta_1|\wedge\abs{\eta_2}$, $\eta^*:=|\eta_1| \vee |\eta_2|$.
\end{proposition}

We are now ready to prove Theorem~\ref{theo:impll}.

\begin{proof}[Proof of Theorem~\ref{theo:impll}]

    The proof of this theorem follows by an induction argument, and it is divided into two cases: (i) $\eta^*\ge|z_1-z_2|^2$, (ii) $\eta^*<|z_1-z_2|^2$. Since the proofs of (i) and (ii) are analogous we mostly focus on the proof in case (i) and then explain the minor changes for the proof of case (ii).

    We start with the first step of the induction. By \cite[Theorem 5.2]{1912.04100} for any $z_1,z_2$ it holds
    \begin{equation}
        \label{eq:inll}
        \big|\braket{(G^{z_1}(\ii\eta_1)AG^{z_2}(\ii\eta_2)-M_{12}^A)B}\big|\lesssim \frac{n^{2\xi}}{n\eta_*\sqrt{|\eta_1\eta_2|}},
    \end{equation}
    with very high probability for any small $\xi>0$ uniformly in $|\eta_*|\gtrsim n^{-\xi}$. Fix $|z_1|,|z_2|\le 1-\tau$
    such that $|z_1-z_2|\le n^{-\omega_d}$, and $\eta_1,\eta_2$ such that $\eta_*\ge n^{-1/3}$
    and $\eta^*\ge |z_1-z_2|^2$, and fix $T=n^{-\xi}$, then by
    Lemma~\ref{lem:charpropert} there exist $\eta_{1,0},\eta_{2,0}, z_{1,0},z_{2,0}$, with $|\eta_{i,0}|\gtrsim n^{-\xi}$, $|z_{i,0}|\le 1-\tau/2$, such that for the solution of \eqref{eq:matchar} with initial condition $\eta_{i,0}$, $z_{i,0}$ it holds $\eta_{i,T}=\eta_i$, $z_{i,T}=z_i$. Additionally, since \eqref{eq:incond} for those $\eta_{i,0},z_{i,0}$ is verified thanks to \eqref{eq:inll},
    by Part 1 of Proposition~\ref{pro:charaveinfin} we conclude that
    \begin{equation}
        \label{eq:onestep}
        \big|\braket{(G_T^{z_1}(\ii\eta_1)AG_T^{z_2}(\ii\eta_2)-M_{12,T}^A)B}\big|\lesssim \frac{n^{3\xi}}{n\eta_*\sqrt{|\eta_1\eta_2|}},
    \end{equation}
    with very high probability uniformly in $\eta_*\gtrsim n^{-1/3}$ and $\eta^*\ge |z_1-z_2|^2$. Then, by Proposition~\ref{pro:GFTf1}, we readily conclude that
    \begin{equation}
        \label{eq:onestepgft}
        \big|\braket{(G^{z_1}(\ii\eta_1)AG^{z_2}(\ii\eta_2)-M_{12}^A)B}\big|\lesssim \frac{n^{3\xi}}{n\eta_*\sqrt{|\eta_1\eta_2|}},
    \end{equation}
    with very high probability uniformly in $\eta_*\gtrsim n^{-1/3}$ and $\eta^*\ge |z_1-z_2|^2$, i.e. the
    Gaussian component added in Proposition~\ref{pro:charaveinfin} can be completely removed using
    Proposition~\ref{pro:GFTf1}. This concludes the first step of the induction that reduced the lower threshold for $\eta_*$ from $|\eta_*|\gtrsim n^{-\xi}$ in~\eqref{eq:inll} to $|\eta_*|\gtrsim n^{-1/3}$ in \eqref{eq:onestepgft}.

    For the general induction  step define $\eta^{(k)}:=n^{-1+(2/3)^k}\vee n^{-1+\epsilon}$, we now show that if \eqref{eq:onestepgft} holds for $\eta_*\ge \eta^{(k)}$ then it holds for $\eta_*\ge \eta^{(k+1)}$ as well. Pick $z_1,z_2$ as above and $\eta_1,\eta_2$ such that $\eta_*\ge \eta^{(k+1)}$ and $\eta^*\ge |z_1-z_2|^2$.
    Choose $T=\eta^{(k)}$ and use the local law  \eqref{eq:onestepgft} for $\eta_*\ge \eta^{(k)}$
    as an initial input to apply Part 1 of Proposition~\ref{pro:charaveinfin} again. With the output
    of this step,  together with  an  application of Proposition~\ref{pro:GFTf1} for $\eta_*\ge \eta^{(k+1)}$, we conclude that \eqref{eq:onestepgft} holds for $\eta_*\ge \eta^{(k+1)}$. Iterating this procedure $k\sim |\log \epsilon|$ times we conclude \eqref{eq:goodll} for $\eta^*\ge |z_1-z_2|^2$ and $\eta_*\ge n^{-1+\epsilon}$.
    Note that the apparently accumulating factors of $n^\xi$ at each step are not a problem since the exponent $\xi$ can
    always be redefined before every iteration step. Since $\epsilon$ is given at the beginning, the number
    of iteration steps is finite, hence $\xi$ needs readjustment only finitely many times.

    We now prove Theorem~\ref{theo:impll} for the complementary
    case $\eta^*<|z_1-z_2|^2$; as before we proceed by induction. We start describing the first step of the induction. Fix
    $|z_1|,|z_2|\le 1-\tau$ such that $|z_1-z_2|\le n^{-\omega_d}$,
    $\eta_1,\eta_2$ with $|\eta_i|<|z_1-z_2|^2$, and choose $T=C|z_1-z_2|^2$ for a constant $C>0$.
    Applying Lemma~\ref{lem:charpropert},
    there exist $|\eta_{i,0}|\ge |z_1-z_2|^2$ and $z_{i,0}$ as initial conditions of the characteristics flow
    so that $\eta_{i,T}=\eta_i$ and $z_{i,T}=z_i$. Additionally, since $T\lesssim |z_{1,0}-z_{2,0}|^2$, by \eqref{comp1} it follows that $|\eta_{i,0}|\lesssim |z_{1,0}-z_{2,0}|^2$.   Then by an application of  Part 2 of  Proposition~\ref{pro:charaveinfin}, with $\mathcal{E}(n,\eta_{1,0},\eta_{2,0})=(n\eta_{1,0}\eta_{2,0})^{-1}$ for the  $\eta_{i,0},z_{i,0}$ from Lemma~\ref{lem:charpropert} and Proposition~\ref{pro:GFTf1} we conclude
    \begin{equation}
        \label{eq:onestepgftbet}
        \begin{split}
            \big|\braket{(G^{z_1}(\ii\eta_1)AG^{z_2}(\ii\eta_2)-M_{12}^A)B}\big|&\lesssim \frac{n^\xi}{n\eta_*\eta^*}\left(\eta_*^{1/6}+n^{-1/10}+\left(\frac{\eta^*}{\eta^*+|z_1-z_2|^2}\right)^{1/4}\right) \\
            &\quad+\frac{n^{2\xi}}{n\sqrt{|\eta_1\eta_2|\eta_{*,T}(\eta^*+|z_1-z_2|^2)}}+\frac{n^{3\xi}}{(n\eta_*)^{3/2}\sqrt{|\eta_1\eta_2|}} \\
            &\quad+n^{2\xi} \mathcal{E}(n,\eta_{1,0},\eta_{2,0}) \\
            &\lesssim  \frac{n^{3\xi}}{n\eta_*^{3/2}(\eta^*)^{1/2}}\left(\eta_*^{1/6}+n^{-1/10}+\frac{1}{\sqrt{n\eta_*}}+\left(\frac{\eta^*}{\eta^*+|z_1-z_2|^2}\right)^{1/4}\right)
        \end{split}
    \end{equation}
    holds with very high probability uniformly in $\eta_*\ge |z_1-z_2|^4\vee n^{-1/3}$.
    Note that in the last inequality in~\eqref{eq:onestepgftbet} we used that $\mathcal{E}(n,\eta_{1,0},\eta_{2,0})=(n\eta_{1,0}\eta_{2,0})^{-1}\lesssim [n\eta_*(\eta^*+|z_1-z_2|^2)]^{-1}$ since $|\eta_{i,0}|\ge |z_1-z_2|^2$.
    Along the flow $z_1$ and $z_2$ also move slightly, but it is easy to check that $|z_{1,t}-z_{i,t}|\sim |z_1-z_2|$ for any $0\le t\le T$ since $T\ll |z_1-z_2|$ and the speed of $z_{i,t}$ is bounded. Therefore we ignored the difference between $|z_{1,t}-z_{i,t}|$ and $|z_1-z_2|$ in the above estimates.

    The induction step is now exactly as above using $\widetilde{\eta}^{(k)}:=\big[ |z_1-z_2|^{2(k+1)}\vee n^{-1+(2/3)^k}\big]\vee n^{-1+\epsilon}$ instead of $\eta^{(k)}$ at each induction step and using
    \begin{equation}
        \label{eq:errortime0}
        \mathcal{E}(n,\eta_{1,0},\eta_{2,0})=  \frac{n^{3k\xi}}{n\eta_{*,0}^{3/2}(\eta_0^*)^{1/2}}\left(\eta_{*,0}^{1/6}+n^{-1/10}+\frac{1}{\sqrt{n\eta_{*,0}}}+\left(\frac{\eta_0^*}{\eta_0^*+|z_1-z_2|^2}\right)^{1/4}\right)
    \end{equation}
    for $T= \widetilde{\eta}^{(k)}$ and $|\eta_{i,0}|\ge \widetilde{\eta}^{(k)}$ as an input for  Part 2 of Proposition~\ref{pro:charaveinfin}. The application of Proposition~\ref{pro:GFTf1} is completely analogous to the
    case $\eta^*\ge |z_1-z_2|^2$ in the first part of the proof.
    In particular, at the $k+1$-th step we get that for $\eta_*\ge \widetilde{\eta}^{(k+1)}$ it holds
    \begin{equation}
        \begin{split}
            \label{eq:finproof}
            \big|\braket{(G^{z_1}(\ii\eta_1)AG^{z_2}(\ii\eta_2)-M_{12}^A)B}\big|&\lesssim  \frac{n^{3k \xi}}{n\eta_*\eta^*}\left(\eta_*^{1/6}+n^{-1/10}+\left(\frac{\eta^*}{\eta^*+|z_1-z_2|^2}\right)^{1/4}\right) \\
            &\quad+\frac{n^{2\xi}}{n\sqrt{|\eta_1\eta_2|\eta_{*,T}(\eta^*+|z_1-z_2|^2)}}+\frac{n^{3\xi}}{(n\eta_*)^{3/2}\sqrt{|\eta_1\eta_2|}} \\
            &\quad+n^{2\xi} \mathcal{E}(n,\eta_{1,0},\eta_{2,0}) \\
            &\lesssim  \frac{n^{3(k+1)\xi}}{n\eta_*^{3/2}(\eta^*)^{1/2}}\left(\eta_*^{1/6}+n^{-1/10}+\frac{1}{\sqrt{N\eta_*}}+\left(\frac{\eta^*}{\eta^*+|z_1-z_2|^2}\right)^{1/4}\right),
        \end{split}
    \end{equation}
    where we used the expression of $\mathcal{E}(n,\eta_{1,0},\eta_{2,0})$ from \eqref{eq:errortime0}.
    In the last inequality of \eqref{eq:finproof} we used the definition of $\widetilde{\eta}^{(k)}$ and
    the fact that $|\eta_{i,t}|$ is decreasing in time to show that $\mathcal{E}(n,\eta_{1,0},\eta_{2,0})$ is smaller than the last line of \eqref{eq:finproof} (recall that $\eta_i=\eta_{i,T}$).
    Iterating this procedure $\sim 1/(2\omega_d)\vee |\log \epsilon|$ times we conclude the proof of this theorem in the case $\eta^*<|z_1-z_2|^2$ as well. We remark that, also in this case, the accumulating factors of $n^\xi$ is not a problem,
    since the exponent $\xi$ can  always be redefined before every iteration step and the number of iteration steps is finite.

\end{proof}

We conclude this section with the proofs of Propositions~\ref{gft prop}--\ref{pro:GFTf1}.

\begin{proof}[Proof of~\cref{pro:GFTf1}]

    By \eqref{eq:GFTboundgood}, estimating $n^{-2/k-1/2}\le n^{-1/10}$ for any $k\ge 5$ in the first term of \eqref{eq:GFTboundgood}, we get
    \begin{equation}
        \label{eq:befyoung}
        \abs{\dif \E \abs{R_t}^p} \lesssim \left(\frac{n^\xi}{n\eta_1\eta_2}\right)^4|R_t|^{p-4}+\sum_{k=5}^p\left(\frac{n^{-1/10+\xi}}{n\eta_1\eta_2}\right)^k+\Bigl(1+\frac{1}{\sqrt{n}\eta_\ast}\Bigr) \E\biggl[\frac{\abs{R_t}^{p-1}}{n\eta_1\eta_2} + \abs{R_t}^{p-3}\Bigl(\frac{1}{n\eta_1\eta_2}\Bigr)^3\biggr].
    \end{equation}
    We now estimate
    \begin{equation}
        \begin{split}
            \label{eq:usedb}
            \left(\frac{n^\xi}{n\eta_1\eta_2}\right)^r|R_t|^{p-r}&\le \left(\frac{\eta^*+|z_1-z_2|^2}{\eta^*}\right)\left(\frac{n^\xi}{n\eta_*(\eta^*)^{1-1/r}(|z_1-z_2|^2+\eta^*)^{1/r}}\right)^r|R_t|^{p-r}, \\
            \frac{1}{\sqrt{n}\eta_*}\left(\frac{1}{n\eta_1\eta_2}\right)^r|R_t|^{p-r}&\le \frac{1}{\sqrt{n}\eta_*^{3/2}}\left(\frac{\eta_*^{1/2r}}{n\eta_1\eta_2}\right)^r|R_t|^{p-r}.
        \end{split}
    \end{equation}
    Then, combining \eqref{eq:befyoung} with \eqref{eq:usedb}, where we use the first bound for $r=1,3,4$ and the second bound for $r=1,3$, and using that $\eta^*\le 1$, $(\eta^*+|z_1-z_2|^2)/\eta^*\ge 1$ we conclude
    \begin{equation}
        \label{eq:bafty}
        \abs{\dif \E \abs{R_t}^p}\lesssim \left(\left(\frac{\eta^*+|z_1-z_2|^2}{\eta^*}\right)^4+\frac{1}{\sqrt{n}\eta_*^{3/2}}\right)\left(|R_t|^p+\left(\frac{n^\xi\eta_*^{1/6}+n^{-1/10+\xi}}{n\eta_*\eta^*}\right)^p+\left(\frac{n^\xi}{n\eta_*(|z_1-z_2|^2+\eta^*)}\right)^p\right).
    \end{equation}
    Finally, \eqref{eq:GFTfav} readily follows from \eqref{eq:bafty} by a simple Gronwall inequality.

\end{proof}

\begin{proof}[Proof of~\cref{gft prop}]
    For notational simplicity we drop the absolute value. In fact the whole proof verbatim applies to \(\prod_{k=1}^pR_t(\sigma_{k,1}\eta_1,\sigma_{k,2}\eta_2)\) for any \(\sigma_{k,1},\sigma_{k,2}=\pm 1\) and therefore the absolute value can be obtained by choosing half the \(\sigma\)'s to be \(+1\) and the other half \(-1\) with \(A,B\) replaced by \(B^\ast,A^\ast\).

    By It\^{o}'s formula we have
    \begin{equation}
        \label{eq:evexp}
        \dif \E R_t^p=\mathbf{E}\left[-\frac{1}{2}\sum_\alpha w_\alpha(t)\partial_\alpha R_t^p+\frac{1}{2}\sum_{\alpha,\beta} \kappa(\alpha,\beta) \partial_\alpha\partial_\beta R_t^p\right],
    \end{equation}
    where $w_\alpha(t)$ are the entries of \(W_t\) for double-indices \(\alpha=(a,b)\in[2n]^2\), and \(\kappa(\alpha,\beta)\) denotes the joint cumulant of \(w_\alpha,w_\beta\). Performing a cumulant expansion\footnote{Such an expansion
        was first used in the random matrix context in~\cite{MR1411619} and later
        revived in~\cite{MR3678478, MR3405746}. Technically we use a truncated version of the expansion above, see e.g.~\cite{MR3941370,MR3678478,MR3941370}.
        The truncation error of the cumulant expansion after \(K=\mathrm{const}\cdot p\) terms can be estimated trivially by the single-\(G\) local law for resolvent entries, and by norm for entries of \(GAG\cdots\) resolvent chains.} in \eqref{eq:evexp} we conclude that
    \begin{equation}
        \label{eq:cumexp}
        \dif \E R_t^p= - \frac{1}{2} \sum_{k=3}^K\sum_{ab}\sum_{\bm\alpha\in \set{ab,ba}^k} \frac{\kappa(\bm\alpha)}{(k-1)!} \E \partial_{\bm\alpha} R_t^p + \Omega_K
    \end{equation}
    with an error term \(\Omega_K=\landauOprec{(n\eta_1\eta_2)^{-p}}\) for \(K=\mathrm{const}\cdot p\), where we recall from~\cref{eq:eqderres} that \(\sum_{ab}\) is a shorthand notation for $\sum_{a\le n}\sum_{b>n}+\sum_{a>n}\sum_{b\le n}$. We first establish a priori bounds on derivatives of \(R_t\) which are sufficient for higher order cumulants.
    \begin{lemma}\label{lemma a priori}
        For any \(\bm\alpha\in\set{ab,ba}^k\) and \(k\ge 1\) we have
        \begin{equation}\label{lemma a priori eq}
            \abs{\partial_{\bm\alpha}R_t} \prec \frac{1}{n\eta_1\eta_2}.
        \end{equation}
    \end{lemma}
    \begin{proof}
        By the differentiation identity $\partial_{ab}G=-G\Delta^{ab}G$ for the resolvent (which follows from applying the resolvent identity to the difference quotient) we have that
        \begin{equation}
            \partial_{ab} R_t
            = -\frac{(G_1AG_2BG_1)_{ba}+(G_2BG_1AG_2)_{ba}}{n}
        \end{equation}
        and therefore any derivative is a sum of monomials of types
        $$
            (G_1AG_2BG_1)\prod G, \quad  (G_2BG_1AG_2)\prod G, \quad \mbox{or}\quad  (G_1AG_2)(G_2BG_1)\prod G,
        $$
        where \(\prod G\) stands for some product of entries of either \(G_1,G_2\). By estimating \(\abs{G_{ab}}\prec 1\) (see \eqref{eq:singlegllaw}) and Schwarz-inequlities of the form
        \begin{equation}\label{entrywise R der}
            \begin{split}
                \abs{(G_1AG_2)_{ab}}&\le \sqrt{(G_2^\ast A^\ast AG_2)_{bb}(G_1G_1^\ast)_{aa}} \le \frac{\norm{A}\sqrt{(\Im G_2)_{bb}(\Im G_1)_{aa}}}{\sqrt{\eta_1\eta_2}} \prec \frac{1}{\sqrt{\eta_1\eta_2}}\\
                \abs{(G_1AG_2BG_1)_{ab}} &\le \sqrt{(G_1^\ast B^\ast G_2^\ast A^\ast AG_2BG_1)_{bb}(G_1G_1^\ast)_{aa}} \le \frac{\norm{A}\norm{B}\sqrt{(\Im G_1)_{aa}(\Im G_1)_{bb}}}{\eta_1\eta_2} \prec \frac{1}{\eta_1\eta_2}
            \end{split}
        \end{equation}
        the claim follows. Here we also used the operator-bound $0\le A^\ast A\le \norm{A}^2I$ and the Ward identity $GG^\ast = \Im G/\eta$.
    \end{proof}
    By distributing the derivatives~\cref{eq:cumexp} according to the Leibniz rule and using~\cref{lemma a priori} it follows that we can estimate the \(k\)-th order terms using \(\abs{\kappa(\bm\alpha)}\lesssim N^{-k/2}\) by
    \begin{equation}
        n^{2-k/2}\sum_{l=1}^{k\wedge p} \Bigl(\frac{1}{n\eta_1\eta_2}\Bigr)^{l} \abs{R_t}^{p-l}.%
    \end{equation}

    We now consider the third order terms (\(k=3\)) which symbolically (with \(\partial=\partial_{ab}+\partial_{ba}\) and \(R=R_t\) and ignored summations and constants) are given by
    \begin{equation}\label{cumulant third order}
        (\partial^3 R) R^{p-1} + (\partial^2 R)(\partial R) R^{p-2} + (\partial R)^3 R^{p-3}.
    \end{equation}
    We begin with the last term in~\cref{cumulant third order} which is given by
    \begin{equation}\label{RRR est}
        \begin{split}
            n^{-3/2}\abs{(\partial R)^3} &\lesssim n^{-9/2} \sum_{ab} \abs{(G_1AG_2BG_1)_{ba} + (G_2BG_1 A G_2)_{ba} }^3 \\
            &\prec \frac{1}{n^{9/2}\eta_1\eta_2}\sum_{ab}\Bigl(\abs{(G_1AG_2BG_1)_{ba}}^2+\abs{(G_2BG_1 A G_2)_{ba}}^2\Bigr)\\
            &= \frac{1}{n^{7/2}} \Bigl(\frac{\braket{\Im G_1AG_2B\Im G_1 B^\ast G_2^\ast A^\ast }}{\eta_1^3\eta_2}+\frac{\braket{\Im G_2BG_1A\Im G_2 A^\ast G_1^\ast B^\ast }}{\eta_1\eta_2^3}\Bigr)\\
            &\prec \frac{1}{n^{7/2}}\frac{\eta_\ast}{\eta_\ast^2(\eta_1\eta_2)^3}=
            \Bigl(\frac{1}{n\eta_1\eta_2}\Bigr)^3 \frac{1}{\sqrt{n}\eta_\ast},
        \end{split}
    \end{equation}
    where in the second step we used the entrywise estimate from~\cref{entrywise R der} for one of the factors, and in the third step performed the \(a,b\) summation for the other two. The power counting behind~\cref{RRR est} is rather simple: the naive size of \((\partial R)^3\) would be given by \(n^{-3/2}\times n^2\times (n\eta_1\eta_2)^{-3}\) if all three factors were estimated by the a priori estimate in~\cref{lemma a priori eq}, where the \(n^{-3/2}\) comes from the third order cumulant, \(n^2\) from the summation, and \((n\eta_1\eta_2)^{-3}\) from the a priori estimate in~\cref{lemma a priori eq}. However, by summation off-diagonal resolvent chains, i.e.\ products of resolvents and deterministic matrices evaluated in the $(a,b)$ or $(b,a)$-entry, can be estimated more efficiently.
    By two \emph{Ward estimates} we can gain an additional factor of \([(n\eta_\ast)^{-1/2}]^2\) over the naive size, giving \((n\eta_1\eta_2)^{-3} n^{-1/2}\eta_\ast^{-1}\) for the final estimate. The very same power counting principle applies also to the first two terms in~\cref{cumulant third order} whenever at least two off-diagonal resolvent chains appear (in which case all three resolvent chains are off-diagonal by parity) that allows us to sum up \(a,b\) into a full trace. Now we record the remaining terms explicitly:
    \begin{equation}\label{R RR est}
        \begin{split}
            n^{-3/2}\abs*{(\partial^3 R)} &\prec \frac{(\sqrt{n}\eta_\ast)^{-1}}{n\eta_1\eta_2} +  \abs*{\sum_{ab} \frac{(GAGBG)_{ab} G_{aa} G_{bb} + (GAG)_{ab}(GBG)_{aa}G_{bb} + G_{ab} (GAG)_{aa} (GBG)_{bb}}{n^{5/2}} }\\
            n^{-3/2}\abs{(\partial^2 R)(\partial R)} &\prec \frac{(\sqrt{n}\eta_\ast)^{-1}}{(n\eta_1\eta_2)^2} +  \abs*{\sum_{ab} \frac{(GAGBG)_{ab} \Bigl((GAGBG)_{aa} G_{bb} + (GAG)_{aa}(GBG)_{bb}\Bigr)}{n^{7/2}} }
        \end{split}
    \end{equation}
    where we dropped the subscripts from \(G\) for notational brevity as they play no role in the sequel.
    For the terms in~\cref{R RR est} with some \(G_{aa}\) or \(G_{bb}\) we further split \(G_{aa}=(G-m)_{aa}+m\) so that using the entry--wise local law \(\abs{(G-m)_{aa}}\prec (n\eta_\ast)^{-1/2}\) from~\cref{eq:singlegllaw} gives the same estimate as the first term on the right hand side of~\cref{R RR est}\footnote{For instance, for the $(GAGBG)(GAGBG)(G-m)$ term we obtain a bound of \[n^{-7/2}\eta^{-2}(n\eta)^{-1/2}\sum_{ab}\abs{(GAGBG)_{ab}}\le n^{-5/2}\eta^{-5/2}\braket{\abs{GAGBG}^2}^{1/2}\prec n^{-5/2}\eta^{-5}=(n\eta^2)^{-2}(\sqrt{n}\eta)^{-1},\]ignoring the difference of $\eta_1,\eta_2$ for convenience.}. For the \(m\) contribution we use a so called \emph{isotropic resummation} trick: for example for the second line of~\cref{R RR est} after splitting \(G_{bb}=m+(G-m)_{bb}\) for the \(m\)-contribution the \(b\)-index only appears once in \((GAGBG)_{ab}\) and can be summed into the inner product \(\braket{\bm e_a,GAGBG \bm 1_a}\)
    with the \(a\)-th unit vector \(\bm e_a\) and the constant-\((0,1)\) vector $\bm 1_a$ defined as
    \(\bm 1_a:=(0,\ldots,0,1,\ldots,1)\) for  \(a>n\) and
    \(\bm 1_a: = (1,\ldots,1,0,\ldots,0)\) for \(a\le n\) (recall that the $\sum_{ab}$ summation is restricted to either $a\le n,b>n$ or $a>n,b\le n$). Thus
    \begin{equation}
        \begin{split}
            &\abs*{\sum_{ab} \frac{(GAGBG)_{ab} (GAGBG)_{aa}}{n^{7/2}}}\\
            &\quad=\abs*{\sum_{a} \frac{\braket{\bm e_a, GAGBG\bm 1_a} (GAGBG)_{aa}}{n^{7/2}}} \\
            &\quad\prec \frac{\sqrt{\braket{\bm1_n,G^\ast B^\ast G^\ast A^\ast G^\ast E_1 GAGBG\bm 1_n}}+\sqrt{\braket{\bm1_{2n},G^\ast B^\ast G^\ast A^\ast G^\ast E_2 GAGBG\bm 1_{2n}}}}{n^{3}\eta_1\eta_2} \\
            &\quad\le \frac{\norm{A}\norm{B}\sqrt{\braket{\bm1_n,\Im G\bm 1_n}+\braket{\bm1_{2n},\Im G\bm 1_{2n}}}}{n^{3}\eta_1^2\eta_2^2\sqrt{\eta_1}} \prec \Bigl(\frac{1}{n\eta_1\eta_2}\Bigr)^2\frac{1}{\sqrt{n\eta_\ast}}
        \end{split}
    \end{equation}
    using Cauchy-Schwarz, the operator bound $B^\ast G^\ast A^\ast G^\ast E_1 GAGB\lesssim (\norm{B}^2\norm{A}^2\eta^{-4})I$ and the \emph{isotropic} local law \(\braket{\bm 1_a,\Im G\bm 1_a}\prec \norm{\bm 1_a}^2= n\). The same argument applies to the diagonal \(G_{aa},G_{bb}\) terms in the first line of~\cref{R RR est}. Together with the following~\cref{lemma extra terms} for the remaining terms of~\cref{R RR est} we thus conclude the proof of~\cref{gft prop}.
\end{proof}
\begin{lemma}\label{lemma extra terms}
    \begin{align}\label{extra 1}
        \abs*{n^{-5/2}\sum_{ab} \E G_{ab} (GAG)_{aa} (GBG)_{bb}}       & \prec \frac{1}{n\eta_1\eta_2}\Bigl(1+\frac{1}{\sqrt{n}\eta_\ast}\Bigr),              \\ \label{extra 2}
        \abs*{n^{-7/2}\sum_{ab} \E (GAGBG)_{ab} (GAG)_{aa} (GBG)_{bb}} & \prec \Bigl(\frac{1}{n\eta_1\eta_2}\Bigr)^2\Bigl(1+\frac{1}{\sqrt{n}\eta_\ast}\Bigr)
    \end{align}
\end{lemma}
\begin{proof}[Proof of~\cref{lemma extra terms}]
    For~\cref{extra 1} we write
    \begin{equation}\label{G unWG}
        G=M-M\un{WG}+ G\landauOprec*{\frac{1}{n\eta_\ast}} :=M - M (WG+\SS[G]G) + G\landauOprec*{\frac{1}{n\eta_\ast}},
    \end{equation}
    c.f.~\cite[Eq.(5.2)]{1912.04100}, in terms of the ``underline renormalization'' from~\cref{underline def}. When using~\cref{G unWG} for \(G_{ab}\) in~\cref{extra 1} the \(M\)- and \(\landauOprec{\cdots}\)-contributions are trivially of the claimed size. For the \(\un{WG}\)-part we compute with another cumulant expansion
    \begin{equation}\label{1cum imrpov}
        \begin{split}
            \abs*{n^{-5/2}\sum_{ab} \E (\un{WG})_{ab} (GAG)_{aa} (GBG)_{bb}}
            &\lesssim n^{-7/2} \abs*{\sum_{abc} \E G_{cb} \partial_{ca}\Bigl((GAG)_{aa} (GBG)_{bb}\Bigr)} \\
            &\quad + n^{-5/2}\sum_{k\ge 2} \abs*{\sum_{abc} \sum_{\bm\beta \in\set{ac,ca}^k} \frac{\kappa(ac,\bm\beta)}{k!}\E \partial_{\bm\beta} \Bigl(G_{cb} (GAG)_{aa} (GBG)_{bb}\Bigr)}.
        \end{split}
    \end{equation}
    For the last term the trivial estimate is sufficient: the \(\partial_{\bm\beta}\)-derivative is bounded by \((\eta_1\eta_2)^{-1}\) c.f.~\cref{entrywise R der}, while the prefactor, i.e. the cumulant and the summation, contribute \(n^{-k/2}\). For the first term we exploit that the action of the \(\partial_{ca}\) derivative yields at least one additional off-diagonal resolvent chain, hence we gain an additional factor of \([(n\eta_\ast)^{-1/2}]^2\) over the naive size \(n^{-1/2}(\eta_1\eta_2)^{-1}\) to obtain a final bound of \(n^{-1/2}(\eta_1\eta_2)^{-1}(n\eta_\ast)^{-1}\), just as claimed. More explicitly, if for instance the derivative acts on the second $G$ factor, then we estimate
    \begin{equation*}
        \begin{split}
            n^{-7/2} \sum_{abc} \abs{G_{cb} (GAG)_{ac}G_{aa}(GBG)_{bb}}&\le n^{-7/2} \sum_{ab} \abs{G_{aa}}\abs{(GBG)_{bb}} \sqrt{\sum_c \abs{G_{cb}}^2}\sqrt{\sum_c\abs{(GAG)_{ac}}^2}\\
            &=  n^{-7/2} \sum_{ab} \abs{G_{aa}}\abs{(GBG)_{bb}} \sqrt{(G^\ast G)_{bb} (GAGG^\ast A^\ast G^\ast)_{aa}} \\
            &\prec n^{-3/2}\eta^{-3},
        \end{split}
    \end{equation*}
    just as claimed. We point out that in the last inequality we used \eqref{entrywise R der} to estimate all the factors in the second to last line.

    In order to prepare the more complicated estimate on~\cref{extra 2} we now explain the general power counting principle behind the improvement in~\cref{1cum imrpov}: the trivial estimate on the left hand side of~\cref{extra 1} using one off-diagonal gain for \(G_{ab}\) is by a factor of \(\sqrt{n\eta_\ast}\) larger than the right hand side. Performing the cumulant expansion for \(G=M-M\un{WG}+\ldots\) is neutral for the Gaussian (i.e.\ second order cumulant) term (one additional summation is compensated by \(\kappa(ac,ca)=1/n\)) but gains a factor of \(n^{-1/2}\) already for the third order cumulants. This and the additional gain in the Gaussian term due to the second off-diagonal resolvent chain is sufficient for the claimed bound.

    We now consider~\cref{extra 2} where the non-Gaussian terms, as well as the Gaussian terms with two off-diagonal resolvent chains, of the cumulant expansion can be estimated just as for~\cref{extra 1} using either the gain from the third cumulant or the additional Ward estimate. Thus we obtain
    \begin{equation}
        \begin{split}
            &%
            \abs*{n^{-7/2}\sum_{ab} \E (\un{WG}AGBG)_{ab} (GAG)_{aa} (GAG)_{bb}} \\
            & \quad \lesssim \Bigl(\frac{1}{n\eta_1\eta_2}\Bigr)^2\frac{1}{\sqrt{n}\eta_\ast} + \abs*{n^{-9/2}\sum_{abc} \E [(GAGBG)_{cc}G_{ab}+(GAG)_{cc}(GBG)_{ab}] (GAG)_{aa} (GBG)_{bb}  }.%
        \end{split}
    \end{equation}
    We now continue with another cumulant expansion~\cref{G unWG} for \(G_{ab}=M_{ab}-(M\un{WG})_{ab}+\ldots\) and \((GBG)_{ab}=(MBG)_{ab}-(M\un{WG}BG)_{ab}+\ldots\) and obtain
    \begin{equation}\label{iter 3}
        \begin{split}
            &\abs*{n^{-9/2}\sum_{abc} \E (GAGBG)_{cc}(M\un{WG})_{ab} (GAG)_{aa} (GAG)_{bb} } \prec \Bigl(\frac{1}{n\eta_1\eta_2}\Bigr)^2\Bigl(1+\frac{1}{\sqrt{n}\eta_\ast}\Bigr)\\
            &\abs*{n^{-9/2}\sum_{abc} \E (GAG)_{cc}(M\un{WG}BG)_{ab} (GAG)_{aa} (GBG)_{bb}  } \\
            &  \prec \Bigl(\frac{1}{n\eta_1\eta_2}\Bigr)^2\Bigl(1+\frac{1}{\sqrt{n}\eta_\ast}\Bigr) + \abs*{n^{-11/2}\sum_{abcd} \E (GAG)_{cc} (GBG)_{dd}(MG)_{ab} (GAG)_{aa} (GBG)_{bb}  } \prec \Bigl(\frac{1}{n\eta_1\eta_2}\Bigr)^2\Bigl(1+\frac{1}{\sqrt{n}\eta_\ast}\Bigr),
        \end{split}
    \end{equation}
    where for the first inequality we used that for the Gaussian term the \(\partial_{da}\)-derivative creates one additional off-diagonal resolvent chain. For the second inequality we kept the unique term in which the \(\partial_{da}\) derivative does not create an additional off-diagonal resolvent chain and estimated the remaining terms as before. Finally, for the last inequality we performed another cumulant expansion in \(G_{ab}=M_{ab}-(M\un{WG})_{ab}+\ldots=M_{ab}- n^{-1}\sum_e G_{eb} \partial_{ea}[\cdots]+\ldots\) and used that the \(\partial_{ea}\)-derivative creates a second off-diagonal resolvent chain which is sufficient to achieve the claimed bound. Finally the \(M\)- and \(\landauOprec*{\cdot}\)- contributions of~\cref{G unWG} towards~\cref{iter 3} can be trivially estimated.
\end{proof}

\appendix

\section{Additional technical results}
\label{app:techres}

Here we prove several technical inputs which are used in Section~\ref{sec:normllaw}.

\subsection{Proof of Proposition~\ref{prop:CLTresm}}
\label{app:CLT}
The proof is essentially identical to the proof of \cite[Proposition 3.3]{1912.04100} upon replacing \cite[Theorem 5.2]{1912.04100} by the improved~\cref{theo:impll}. For the sake of brevity we here only give a sketch highlighting the differences. The analogue of \cite[Eq.~(6.9)]{1912.04100} is
\begin{equation}
    \braket{G-M} = \braket*{\mathcal D^{-1}[I]M[-\un{WG}+\SS[G-M](G-M)]}, \quad \mathcal D[R]:=1-\SS[MR M]
\end{equation}
and by an explicit computation it follows that
\begin{equation}
    \mathcal D^{-1}[I] = \frac{I}{1-m^2-\abs{z}^2u^2}, \qquad \norm{\mathcal D^{-1}[I]}\lesssim \frac{1}{\tau}\lesssim 1
\end{equation}
and therefore
\begin{equation}
    \braket{G-M} = - \braket{A\un{WG}}+\landauOprec*{\frac{1}{(n\eta)^2}}, \qquad A:= \frac{M}{1-m^2-\abs{z}^2u^2}.
\end{equation}
Here
\begin{equation}\label{underline def}
    \underline{Wf(W)}:=Wf(W)-\widetilde{\mathbf{E}}\widetilde{W}(\partial_{\widetilde{W}}f)(W)
\end{equation}
for any given $f$, where $\widetilde{W}$ is an independent copy of $W$, and $\partial_{\widetilde{W}}$ denotes the directional derivative in the direction $\widetilde{W}$. In particular, we have
$$
    \underline{WG} =  WG +\SS[G]G.
$$

The accuracy of the expectation computation from~\cite[Lemma 6.2]{1912.04100} was already sufficient also on optimal mesoscopic scales and we recall from~\cref{prop clt exp} that
\begin{equation}
    \braket{G-\E G} = \braket{G-M-\mathcal E} + \landauOprec*{\frac{1}{n^{3/2}(1+\eta)}+\frac{1}{(n\eta)^2}},
\end{equation}
where
\begin{equation}
    \mathcal E:= -\frac{\ii\kappa_4}{4n}\partial_\eta (m^4) = \frac{\kappa_4}{n}m^3\braket{MA}=\frac{\kappa_4}{n}\frac{m^3\braket{M^2}}{1-m^2-\abs{z}^2u^2}
\end{equation}
using~\cite[Eq.~(6.10)]{1912.04100}.

Therefore, for the higher moments it suffices to compute
\begin{equation}
    \E \prod_i \braket{G_i-\E G_i} = \prod_i \braket{-A_i\un{WG_i}-\mathcal E_i} + \landauOprec*{\frac{\psi}{\sqrt{n\eta_\ast}}}, \qquad \psi:=\prod_i \frac{1}{n\abs{\eta_i}}
\end{equation}
due to
\begin{equation}
    \abs{\braket{A_i\un{WG_i}}} \prec \frac{1}{n\eta_i}
\end{equation}
by~\cite[Eq.~(6.14)]{1912.04100}. As in~\cite[Eq.~(6.17)]{1912.04100} we then perform a cumulant expansion to obtain
\begin{equation}\label{prod underline}
    \begin{split}
        &\E \prod_{i}\braket{-A_i\un{WG_i}-\mathcal E_i} \\
        &\quad= -\braket{\mathcal E_i} \E \prod_{i\ne 1} \braket{-A_i\un{WG_i}-\mathcal E_i}\\
        &\qquad + \sum_{i\ne 1}\E\wt\E\braket{-A_1\wt WG_1}\braket{-A_i\wt W G_i+A_i\un{WG_i\wt W G_i}}\prod_{j\ne 1,i}\braket{-A_j\un{WG_j}-\mathcal E_j}\\
        &\qquad + \sum_{k\ge 2}\sum_{ab}\sum_{\bm\alpha\in\set{ab,ba}^k}\frac{\kappa(ba,\bm\alpha)}{k!} \E \partial_{\bm\alpha}\Bigl[\braket{-A_1\Delta^{ba}G_1}\prod_{i\ne 1}\braket{-A_i\un{WG_i}-\mathcal E_i}\Bigr]\\
        &\quad= \sum_{i\ne 1}\E\Bigl[\wt\E\braket{-A_1\wt WG_1}\braket{-A_i\wt W G_i+A_i\un{WG_i\wt W G_i}}+\frac{\kappa_4 U_1U_i}{2n^2}\Bigr]\prod_{j\ne 1,i}\braket{-A_j\un{WG_j}-\mathcal E_j} + \landauOprec*{\frac{\psi}{\sqrt{n\eta_\ast}}}.
    \end{split}
\end{equation}
with \(U_i\) as in~\cref{eq:exder}, and where we used~\cite[Eqs.~(6.26),~(6.29)]{1912.04100} for the last equality. By combining~\cite[Eq.~(6.2)]{1912.04100} and the display below \cite[Eq.~(6.21)]{1912.04100} we have
\begin{equation}\label{underline 2G}
    \begin{split}
        &\E\wt\E\braket{-A_1\wt WG_1}\braket{-A_i\wt W G_i+A_i\un{WG_i\wt W G_i}}\prod_{j\ne 1,i}\braket{-A_j\un{WG_j}-\mathcal E_j}\\
        &\quad = \frac{\braket{G_1A_1 E G_i A_i E' + G_1 \SS[G_1 A_1 E G_i A_i]G_i E'}}{2n^2}\E\prod_{j\ne 1,i}\braket{-A_j\un{WG_j}-\mathcal E_j} + \landauO*{\frac{n^\epsilon\psi}{n\eta_\ast}},
    \end{split}
\end{equation}
where it is understood that \((E,E')\) is summed over \((E_1,E_2)\) and \((E_2,E_1)\). From~\cref{theo:impll} and the computations around~\cite[Eq.~(6.23)]{1912.04100} we obtain
\begin{equation}\label{2G V}
    \begin{split}
        \frac{\braket{G_1A_1 E G_i A_i E' + G_1 \SS[G_1 A_1 E G_i A_i]G_i E'}}{2n^2} &= \frac{\braket{M_{A_1 E}^{z_1,z_i} + M_{E'}^{z_i,z_1} \SS[M_{A_1 E}^{z_1,z_i}]}}{2n^2} + \landauOprec*{\frac{1}{n^2\eta_1\eta_i}\frac{1}{n\eta_\ast}}\\
        &= \frac{V_{1,i}}{2n^2} + \landauOprec*{\frac{1}{n^2\eta_1\eta_i}\frac{1}{n\eta_\ast}}
    \end{split}
\end{equation}
with \(V_{1,i}\) as in~\cref{eq:exder}. Inserting~\cref{underline 2G,2G V} into~\cref{prod underline} we conclude the proof of~\cref{prop:CLTresm} by induction.

\subsection{Asymptotic independence}
\label{app:ind}

In this section we present the proof of Theorem~\ref{lem:overb} and of Proposition~\ref{prop:indmr}.

\begin{proof}[Proof of Theorem~\ref{lem:overb}]
    Using the spectral symmetry of \(H^z\), for any \(z\in\C \) we write \(G^z\) in spectral decomposition as
    \[
        G^z(\ii\eta)=\sum_{j>0} \frac{2}{(\lambda_j^z)^2+\eta^2}\left( \begin{matrix}
                \ii \eta {\bm u}_j^z ({\bm u}_j^z)^*   & \lambda_j^z {\bm u}_j^z ({\bm v}_j^z)^* \\
                \lambda_j^z {\bm v}_j^z({\bm u}_j^z)^* & \ii \eta {\bm v}_j^z ({\bm v}_j^z)^*
            \end{matrix}\right).
    \]
    Let \(\eta = n^{-1+\epsilon}\), with $\epsilon\le\omega_p/10$, then by rigidity of the eigenvalues (see e.g. \cite[Eq. (7.4)]{1912.04100}), for any \(1\le i_0, j_0\le n^\omega_B\) such that \(\lambda_{i_0}^{z_l},\lambda_{j_0}^{z_l}\lesssim \eta\), with \(l=1,2\), and any \(z_1, z_2\) such that \(n^{-1/2+\omega_p} \le \abs{z_1-z_2}^2\le n^{-2\omega_d}\), it follows that
    \begin{gather}
        \begin{aligned}
             & \abs*{\braket{ {\bm u}_{i_0}^{z_1}, {\bm u}_{j_0}^{z_2}}}^2+\abs*{\braket{ {\bm v}_{i_0}^{z_1}, {\bm v}_{j_0}^{z_2}}}^2                                                                                                      \\
             & \qquad\lesssim \sum_{i,j=1}^n \frac{4\eta^4}{((\lambda_i^{z_1})^2+\eta^2)((\lambda_j^{z_2})^2+\eta^2)} \left(\abs*{\braket{ {\bm u}_i^{z_1}, {\bm u}_j^{z_2}}}^2+\abs*{\braket{ {\bm v}_i^{z_1}, {\bm v}_j^{z_2}}}^2 \right) \\
             & \qquad=\eta^2\Tr  (\Im G^{z_1})(\Im G^{z_2}) \lesssim \frac{n\eta^2}{|z_1-z_2|^2} +n^\xi\left(\eta^{1/6}+n^{-1/10}+\frac{1}{\sqrt{n\eta}}+\left(\frac{\eta}{|z_1-z_2|^2}\right)^{1/4}\right)                                 \\
             & \qquad\lesssim n^{-2\delta}+n^{(\delta+2\xi-\omega_p)/2}+n^{-(\delta+\omega_p-4\xi)/4}.
        \end{aligned}\label{eq:ooo}\raisetag{-6em}
    \end{gather}
    The first inequality in the third line of~\eqref{eq:ooo} is from Theorem~\ref{theo:impll}. In the last inequality we chose $\eta=n^{-1/2-\delta}|z_1-z_2|$, for any \(\delta\le\omega_p/10\). This concludes the proof by choosing \(\omega_B\) small enough and \(\omega_E:=(2\delta)\wedge [(\omega_p-\delta-2\xi)/2]\wedge [(\delta+\omega_p-4\xi)/4]\) which is still positive by our choice of $\delta$ and $\xi$.
\end{proof}

\begin{proof}[Proof of Proposition~\ref{prop:indmr}]
    The proof of this proposition is completely analogous to \cite[Proof of Proposition 3.5 in Section 7]{1912.04100},
    the only difference is that \cite[Lemma 7.9]{1912.04100} has to be replaced by Theorem~\ref{lem:overb}.
    The key change is that in \cite[Proposition 3.5]{1912.04100} we required $|z_1-z_2|$ to basically
    be order one, whilst now we consider the entire mesoscopic regime $|z_1-z_2|\gg n^{-1/2}$. In particular, Theorem~\ref{lem:overb}
    ensures that \cite[Assumption 7.1]{1912.04100} is satisfied even in the mesoscopic regime
    $|z_1-z_2|\ge n^{-1/2+\omega_p}$ and so that \cite[Proposition 7.14]{1912.04100} holds in this case as well.
    This proposition  and a simple standard GFT were the only input to prove \cite[Proposition 3.5]{1912.04100}.
\end{proof}

\subsection{Proof of Lemma~\ref{pro:imppro}}

Without loss of generality we assume that $\|(\mathcal{B}_{12}^{-1})^*[B^*]\|\lesssim 1$. The case $\|(\mathcal{B}_{21}^{-1})^*[A^*]\|\lesssim 1$ is completely analogous using cyclicity of the trace.

We start writing down the equation for $G_1AG_2$. First of all notice that
\begin{equation}
    \label{eq:G}
    G=M-M\underline{WG}+M\mathcal{S}[G-M]G,
\end{equation}
where $W:=H+Z$, with
\begin{equation}
    Z:=\left(\begin{matrix}
            0            & z \\
            \overline{z} & 0
        \end{matrix}\right)=zF^*+\overline{z}F.
\end{equation}
Here we also recall that $\underline{WG}=WG+\braket{G}G$.

Next, using \eqref{eq:G} for $G_1$ and writing $G_2=M_2+(G_2-M_2)$, we find
\begin{equation}
    \label{eq:GFGFsaa}
    \begin{split}
        G_1AG_2&=M_1AM_2+M_1A(G_2-M_2)-M_1\underline{WG_1AG_2}+M_1\mathcal{S}[G_1-M_1]G_1AG_2\\
        &\quad +M_1\mathcal{S}[G_1AG_2]G_2 \\
        &=M_1AM_2+M_1A(G_2-M_2)-M_1\underline{WG_1AG_2}+M_1\mathcal{S}[G_1-M_1]G_1AG_2\\
        &\quad +M_1\mathcal{S}[G_1AG_2]M_2+M_1\mathcal{S}[G_1AG_2](G_2-M_2),
    \end{split}
\end{equation}
where
\[
    \underline{WG_1AG_2}:=\underline{WG_1}AG_2+\mathcal{S}[G_1AG_2]G_2.
\]
Then, by \eqref{eq:GFGFsaa} it follows that
\begin{equation}
    \begin{split}
        \braket{G_1AG_2B-M_{12}^AB}&=\braket{M_1A(G_2-M_2)((\mathcal{B}_{12}^{-1})^*[B^*])^*}-\braket{M_1\underline{WG_1AG_2}((\mathcal{B}_{12}^{-1})^*[B^*])^*} \\
        &\quad+\braket{M_1\mathcal{S}[G_1-M_1]G_1AG_2((\mathcal{B}_{12}^{-1})^*[B^*])^*} +\braket{M_1\mathcal{S}[G_1AG_2](G_2-M_2)((\mathcal{B}_{12}^{-1})^*[B^*])^*}.
    \end{split}
\end{equation}

Finally, the single resolvent local law $|\braket{G_i-M_i}|\prec (n\eta_i)^{-1}$ from \eqref{eq:singlegllaw} and the bound
\[
    |\braket{M_1\underline{WG_1AG_2}((\mathcal{B}_{12}^{-1})^*[B^*])^*}|\prec \frac{1}{n\eta_*\sqrt{\eta_1\eta_2}}
\]
from \cite[Eq. (5.10c)]{1912.04100} we conclude the proof of this lemma.

\subsection{Proof of Lemma~\ref{lem:mder}}
We present the proof in the most general setting for convenience. Consider the matrix Dyson equation (MDE) \(M=M(\Lambda)\) solving
\begin{equation}\label{MDE}
    -M^{-1} =\SS[M] + \Lambda, \qquad \sgn \Im M=\sgn \Im \Lambda
\end{equation}
for some generalised spectral parameter \(\Lambda\) with \(\Im \Lambda\) either positive or negative definite in order for \(\sgn \Im \Lambda\) to be well defined.
\begin{lemma}\label{lemma single Mdot}
    If \(\Lambda=\Lambda(t)\in \C_\pm^{N\times N}\) (with \(\pm\) independent of \(t\)) solves the ODE
    \begin{equation}
        \label{eq:newmatchar}
        \dot \Lambda := \frac{\dif \Lambda}{\dif t} = - \frac{\Lambda}{2} - \SS[M],
    \end{equation}
    then
    \begin{equation}
        \dot M = \frac{M}{2}.
    \end{equation}
\end{lemma}
\begin{proof}
    By inverting and differentiating~\cref{MDE} we have
    \begin{equation}
        M^{-1}\dot M M^{-1} = \SS[\dot M] + \dot \Lambda = \SS[\dot M] - \SS[M] - \frac{\Lambda}{2} = \SS[\dot M] - \frac{\SS[M]}{2} + \frac{M^{-1}}{2}
    \end{equation}
    and therefore
    \begin{equation}
        \cB[\dot M] = \cB\Bigl[\frac{M}{2}\Bigr], \qquad \cB:=1-M\SS[\cdot ]M
    \end{equation}
    and the claim follows upon inverting \(\cB\).
\end{proof}
\begin{lemma}\label{lemma Mdot}
    Let \(B\) be arbitrary, let \(M_i=M(\Lambda_i)\) be the solution to~\cref{MDE} for \(i=1,2\), and let
    \begin{equation}\label{M12 def}
        M_{12}^{B} = M^{B}(\Lambda_1,\Lambda_2) = \cB_{12}^{-1}[M_1 B M_2], \qquad \cB_{12} := 1-M_1\SS[\cdot]M_2.
    \end{equation}
    If \(\Lambda_i=\Lambda_i(t)\) satisfy the ODE
    \begin{equation}
        \dot\Lambda_i = -\frac{\Lambda_i}{2} - \SS[M_i],
    \end{equation}
    then
    \begin{equation}
        \dot M_{12}^B = \cB_{12}^{-1}[M_{12}^B ].
    \end{equation}
\end{lemma}
\begin{proof}
    We invert \(\cB_{12}\) in~\cref{M12 def} and differentiate to obtain
    \begin{equation}
        \begin{split}
            \dot M_{12} - M_1 \SS[\dot M_{12}]M_2 &= \dot M_1 (B +\SS[M_{12}])M_2 +  M_1 (B +\SS[M_{12}])\dot M_2 \\
            &=\dot M_1 M_1^{-1} M_{12} + M_{12} M_2^{-1}\dot M_2 = M_{12}%
        \end{split}
    \end{equation}
    where we used~\cref{M12 def} in the second step and~\Cref{lemma single Mdot} in the final step.
\end{proof}
\Cref{lemma Mdot} implies for arbitrary \(B_1,B_2\) that
\begin{equation}
    \begin{split}
        \label{eq:12derm}
        \braket{\dot M_{12}^{B_1}B_2} &= \braket*{ B_2\cB_{12}^{-1}[M_{12}^{B_1}]} = \braket*{(1-\SS[M_2\cdot M_1])^{-1}[B_2] M_{12}^{B_1}   }\\ %
        &= \braket*{M_2^{-1}M_{21}^{B_2} M_1^{-1}M_{12}^{B_1}} = \braket*{B_2M_{12}^{B_1}} + \braket*{\SS[M_{21}^{B_2}]M_{12}^{B_1}}
    \end{split}
\end{equation}
using \(M_{2}^{-1}M_{21}^BM_{1}^{-1}=\SS[M_{21}^B]+B\) in the last step. This concludes the proof of Lemma~\ref{lem:mder}.

\subsection{Proof of Lemma~\ref{lem:newlem1}}

Adding and subtracting the deterministic approximations of $G_1L_-'G_2$ and $G_2L_-G_1$, and using the definition of $\mathcal{S}$ in \eqref{defS}, we obtain
\begin{equation}
    \begin{split}
        \braket{\mathcal{S}[G_1L_-'G_2]G_2L_-G_1}&=\braket{\mathcal{S}[M_{12}^{L_-'}]M_{21}^{L_-}}+  \braket{G_1L_-'G_2-M_{12}^{L_-'}}\braket{M_{21}^{L_-}}+\braket{G_1L_-'G_2}\braket{G_2L_-G_1-M_{21}^{L_-}}\\
        &\quad-\braket{(G_1L_-'G_2-M_{12}^{L_-'})E_-}\braket{M_{21}^{L_-}E_-}-\braket{G_1L_-'G_2E_-}\braket{(G_2L_-G_1-M_{21}^{L_-})E_-}.
    \end{split}
\end{equation}

In Appendix~\ref{app:evect} we will derive various elementary facts about the eigenvectors of the stability operator. In particular, in \eqref{eq:usetaylexp} we will prove that $\eta_{1,t}\eta_{2,t}<0$ we have
\begin{equation}
    \label{eq:mayneedon}
    I=(1+\mathcal{O}(|z_{1,t}-z_{2,t}|))L_-+\mathcal{O}(|z_{1,t}-z_{2,t}|)L_+, \qquad E_-=(1+\mathcal{O}(|z_{1,t}-z_{2,t}|))L_++\mathcal{O}(|z_{1,t}-z_{2,t}|)L_-,
\end{equation}
and similar relations hold with $L_\pm$ replaced with $L_\pm'$. Next, using \eqref{eq:mayneedon}, we write
\[
    \braket{(G_1L_-'G_2-M_{12}^{L_-'})E_-}=(1+\mathcal{O}(|z_{1,t}-z_{2,t}|))\braket{(G_1L_-'G_2-M_{12}^{L_-'})L_+}+\mathcal{O}(|z_{1,t}-z_{2,t}|)\braket{(G_1L_-'G_2-M_{12}^{L_-'})L_-}.
\]
We can thus estimate
\begin{equation}
    \label{eq:e-b}
    \braket{(G_1L_-'G_2-M_{12}^{L_-'})E_-}=\mathcal{O}\left(\frac{n^\xi}{n\eta_{*,t}\sqrt{|\eta_{1,t}\eta_{2,t}|}}+|z_1-z_2||Y_t|\right),
\end{equation}
where we used Lemma~\ref{pro:imppro} for $A=L_+$ and $B=L_-'$ to estimate the terms containing $L_+$. Notice that $\| (\mathcal{B}_{12}^*)^{-1} [L_+^*]\| = \| L_+\|/|\beta_+| \lesssim \| L_+\| \sim 1$, where we used that $|\beta_+|\sim 1$ by \eqref{eq:lowbeigva}; this ensures the applicability Lemma~\ref{pro:imppro} for $A=L_+$ and $B=L_-'$.

Using that
\[
    \big|\braket{M_{21}^{L_-}E_-}\big|\lesssim \big|\braket{M_{21}^{L_-}L_+}\big|+|z_{1,t}-z_{2,t}|\big|\braket{M_{21}^{L_-}L_-}\big| \lesssim 1+\frac{|z_{1,t}-z_{2,t}|}{|z_{1,t}-z_{2,t}|^2+\eta_{*,t}},
\]
together with \eqref{eq:e-b} and similar bounds for $\braket{G_1L_-'G_2E_-}=\braket{M_{12}^{L_-'}E_-}+\braket{(G_1L_-'G_2-M_{12}^{L_-'})E_-}$, we obtain
\begin{equation}
    \begin{split}
        \label{eq:addinfneedapp}
        \braket{\mathcal{S}[G_1L_-'G_2]G_2L_-G_1}&=\braket{\mathcal{S}[M_{12}^{L_-}]M_{21}^{L_-}}+  \braket{G_1L_-'G_2-M_{12}^{L_-'}}\braket{M_{21}^{L_-}}+\braket{G_1L_-'G_2}\braket{G_2L_-G_1-M_{21}^{L_-}} \\
        &\quad+\mathcal{O}\left(\frac{n^\xi}{n\eta_{*,t}^2\sqrt{|\eta_{1,t}\eta_{2,t}|}}+|Y_t|+|z_{1,t}-z_{2,t}|^2|Y_t|^2\right),
    \end{split}
\end{equation}
where we used that $\epsilon>\xi/10$ (recall $\eta_{*,t}\ge n^{-1+\epsilon}$).

Finally, writing $\braket{G_1L_-'G_2}=\braket{M_{12}^{L_-'}}+\braket{(G_1L_-'G_2-M_{12}^{L_-'})}$, using
\begin{equation}
    \begin{split}
        \braket{M_{21}^{L_-}}&=\braket{M_{12}^I}(1+\mathcal{O}(|z_{1,t}-z_{2,t}|)), \\
        \braket{G_1L_-'G_2-M_{12}^{L_-'}}&=(1+\mathcal{O}(|z_{1,t}-z_{2,t}|))Y_t+\mathcal{O}(|z_{1,t}-z_{2,t}|)\braket{(G_1L_-'G_2-M_{12}^{L_-'})L_+},
    \end{split}
\end{equation}
by \eqref{eq:mayneedon}, and a similar approximation for the last term in the first line of \eqref{eq:addinfneedapp}, we conclude \eqref{eq:addinfneed}.

\subsection{Proof of Lemma~\ref{lem:m12bound}} Define
\[
    \Xi_t=\Xi(\eta_{1,t},\eta_{2,t},z_{1,t},z_{2,t}):=|\eta_{1,t}|+|\eta_{2,t}|+|z_{1,t}-z_{2,t}|^2.
\]
Then, we estimate
\begin{equation}
    \label{eq:explM12a}
    \begin{split}
        &\braket{M_{12,t}^I}=\frac{m_{1,t}m_{2,t}+1-\Re[z_{1,t}\overline{z_{2,t}}]u_{1,t}u_{2,t}}{1+|z_{1,t}z_{2,t}|^2u_{1,t}^2u_{2,t}^2-m_{1,t}^2m_{2,t}^2-2u_{1,t}u_{2,t}\Re [z_{1,t}\overline{z_{2,t}}]}-1 \\
        &=\frac{\sqrt{(1-|z_{1,t}|^2)(1-|z_{2,t}|^2)}+1-\Re[z_{1,t}\overline{z_{2,t}}]}{|z_{1,t}-z_{2,t}|^2+|\eta_{1,t}|\sqrt{1-|z_{1,t}|^2}+|\eta_{2,t}|\sqrt{1-|z_{2,t}|^2}}\big(1+\mathcal{O}\left(\Xi_t\right)\big) \\
        &=\frac{\sqrt{(1-|z_{1,0}|^2)(1-|z_{2,0}|^2)}+1-\Re[z_{1,0}\overline{z_{2,0}}]}{|z_{1,0}-z_{2,0}|^2+|\eta_{1,0}|\sqrt{1-|z_{1,0}|^2}+|\eta_{2,0}|\sqrt{1-|z_{2,0}|^2}-t[2-|z_{1,0}|^2-|z_{2,0}|^2]}\big(1+\mathcal{O}\left(\Xi_0+t\right)\big) \\
        &\le \frac{2-|z_{1,0}|^2-|z_{2,0}|^2}{|z_{1,0}-z_{2,0}|^2+|\eta_{1,0}|\sqrt{1-|z_{1,0}|^2}+|\eta_{2,0}|\sqrt{1-|z_{2,0}|^2}-t[2-|z_{1,0}|^2-|z_{2,0}|^2]}\big(1+\mathcal{O}\left(\Xi_0+t\right)\big),
    \end{split}
\end{equation}

We point out that here we also used that
\[
    |\eta_{i,t}|=|\eta_{i,0}|-t\sqrt{1-|z_{i,0}|^2}+\mathcal{O}(t|\eta_{i,t}|), \qquad z_{i,t}=e^{-t/2}z_{i,0}=z_{i,0}\big(1+\mathcal{O}(t)\big),
\]
and that
\[
    1-\Re[z_{1,0}\overline{z_{2,0}}]=1-\frac{|z_{1,0}|^2+|z_{2,0}|^2}{2}+|z_{1,0}-z_{2,0}|^2.
\]
Additionally, to go from the first to the second line of \eqref{eq:explM12} we also used \eqref{eq:expsmalleta}.

\section{Eigendecomposition of the stability operator}
\label{app:evect}

The stability operator \(\cB_{12}:=1-M_1\SS[\cdot]M_2\), with $M_i:=M^{z_i}(w_i)$,
acts on the $4n^2$ dimensional space of $(2n)\times (2n)$ block matrices.
From the action of $\SS$ and the fact that $M_i$ are block constant, it
immediately follows that both $\cB_{12}$ and $\cB^*_{12}$
leave the $4n^2-4$ dimensional subspace of block traceless matrices invariant and they
act trivially as the
identity on it. Now we describe its spectral data on the 4 dimensional space of  block constant matrices,
which are a constant multiple of the \(n\times n\) identity in all four blocks.
We use the notation $u_i:=u^{z_i}(\ii\eta_i)$, $m_i:=m^{z_i}(\ii\eta_i)$, with $u^z$, $m^z$ being defined in \eqref{Mz} and \eqref{eq:MDEscal}, respectively. There are eigenvalues \((1,1,\beta_+,\beta_-)\) with \footnote{The complex square root $\sqrt{\cdot}$ in \eqref{eq:evalues} and \eqref{eq:Rpm} is defined using the standard branch cut ${\mathbf{C} \setminus (-\infty, 0)}$.}
\begin{equation}
    \label{eq:evalues}
    \beta_\pm := 1\pm \sqrt{s} - u_1 u_2 \Re z_1\ov{z_2} , \quad s:= m_1^2 m_2^2 - u_1^2 u_2^2 (\Im z_1\ov{z_2})^2,
\end{equation}
and right eigenvectors \(F,F^\ast,R_+,R_-\) in the sense that \(\cB_{12}[R_\pm]=\beta_{\pm}R_\pm\) and \(\cB_{12}[F^{(\ast)}]=F^{(\ast)}\), where \(F\) has been defined in \eqref{eq:defe1e2f} and
\begin{equation}\label{eq:Rpm}
    R_\pm =   \begin{pmatrix}
        -u_1u_2 \Re z_1 \bar{z}_2 \pm \sqrt{s}                                                          & z_1 u_1 m_2  +
        \frac{z_2u_2m_1^2m_2}{ \ii u_1 u_2 \Im z_1 \bar{z}_2 \mp \sqrt{s}  }                                             \\[2mm]
        \bar{z}_2 u_2 m_1  + \frac{\bar{z}_1u_1m_2^2m_1}{ \ii u_1 u_2 \Im z_1 \bar{z}_2 \mp \sqrt{s}  } &
        \frac{m_1 m_2}{ \ii u_1 u_2 \Im z_1 \bar{z}_2 \mp \sqrt{s}  } \big(-u_1u_2 \Re z_1 \bar{z}_2 \pm \sqrt{s}\big)
    \end{pmatrix}.
\end{equation}

We point out that the quantities $s,\beta_\pm$, $R_\pm$, and $L_\pm$ below naturally depend on $z_1,z_2,\eta_1,\eta_2$; we omitted this from the notation for simplicity. Furthermore, we remark that the explicit formulas for the eigenvalues in \eqref{eq:evalues}, for the right eigenvectors in \eqref{eq:Rpm}, and for the left eigenvectors in \eqref{eq:Lpm} below hold for any $z_1,z_2\in\C$ and for any spectral parameters $w_1,w_2\in \C\setminus \R$; however we present the following estimates only on the imaginary axis, i.e. for $w_i=\ii\eta_i$, and for $1-|z_i|^2\sim 1$, $|z_1-z_2|\ll 1$, since this is the only regime needed in the proof of Theorem~\ref{theo:mesoclt}. Note that by an explicit computation for $z_1=z_2$ and a simple Taylor expansion it follows that $\norm{R_\pm}\sim 1$. On the eigenvalues $\beta_\pm$ we also have the following asymptotics:
\begin{equation}
    \label{eq:lowbeigva}
    \beta_-\sim |z_1-z_2|^2+\eta_1+\eta_2, \qquad |\beta_+|\sim  1.
\end{equation}
The fact that $|\beta_+|\sim 1$ follows trivially by its explicit expression in \eqref{eq:evalues}, the lower bound for $|\beta_-|$ follows by \cite[Lemma 6.1]{MR4235475}, whilst the upper bound follows from the fact that for $\eta_1=\eta_2=0$ we have
\[
    \begin{split}
        \beta_-&=1-\Re[z_1\overline{z_2}]-\sqrt{1-|z_1|^2-|z_2|^2+\Re[z_1\overline{z_2}]^2} \\
        &= 1-\frac{|z_1|^2+|z_2|^2}{2}+\frac{|z_1-z_2|^2}{2}-\sqrt{\left(1-\frac{|z_1|^2+|z_2|^2}{2}\right)^2-|z_1-z_2|^2\left(|z_1|^2+|z_2|^2+\frac{|z_1-z_2|^2}{4}\right)} \\
        &\lesssim |z_1-z_2|^2
    \end{split}
\]
and a simple Taylor expansion in the $\eta_i$ variables. We point out that here we used $2\Re[z_1\overline{z_2}]=|z_1|^2+|z_2|^2-|z_1-z_2|^2$ and that for $\eta_i=0$ it holds $m_i^2=|z_i|^2-1$, $u_i=1$.

Since \(\cB_{12}\) is not self-adjoint, it has a separate set of left eigenvectors defined by
$$
    \mathcal{B}_{12}^*[L_\pm^*]=\overline{\beta_\pm}L_\pm^*, \qquad \mathcal{B}_{12}^*[L_{(*)}^*]=L_{(*)}^*.
$$
The left eigenvectors corresponding to \(\overline{\beta_\pm}\) are given by
\begin{equation}
    \label{eq:Lpm}
    L_\pm = \frac{1}{m_1 m_2}\begin{pmatrix}
        \ii u_1 u_2 \Im z_1 \bar{z}_2 \mp \sqrt{s} & 0 \\ 0 &  m_1 m_2
    \end{pmatrix}.
\end{equation}
We do not give the explicit form of the eigenvectors $L_{(*)}^*$ since they are not used for our analysis. With the normalizations above we have $\norm{L_\pm}\sim 1$, and $\braket{L_-R_-}\sim 1$; however $\braket{L_+R_+}$
can be small when $|z_1|,|z_2|\approx 2^{-1/2}$, i.e. when $\beta_+$ resonates with
the eigenvalue $1$.  As before, for these conclusions  we used that for $z_1=z_2=z$ it holds
\begin{equation}\label{R+}
    \braket{L_\pm R_\pm}=|m|^2\mp |z|^2|u|^2,
\end{equation}
and Taylor expansion to access the $|z_1-z_2|\ll 1$ regime.

We will also need to express the standard basis vectors $I, E_-$ in terms of $L_\pm$.
For $z_1=z_2$ it holds $I=L_-$, $E_-=L_+$ if $\eta_1\eta_2<0$, and $I=L_+$, $E_-=L_-$ if $\eta_1\eta_2>0$.
Then again using a simple Taylor expansion it follows that
\begin{equation}
    \label{eq:usetaylexp}
    I=(1+\mathcal{O}(|z_1-z_2|))L_\sigma+\mathcal{O}(|z_1-z_2|)L_{-\sigma}, \qquad E_-=(1+\mathcal{O}(|z_1-z_2|))L_{-\sigma}+\mathcal{O}(|z_1-z_2|)L_\sigma,
\end{equation}
with $\sigma:=\mathrm{sign}(\eta_1\eta_2)$.

\printbibliography

\end{document}

%% file: packages2.tex
\usepackage[english]{babel}
\usepackage[p,osf]{cochineal}
\usepackage[scale=.95,type1]{cabin}
\usepackage[zerostyle=c,scaled=.94]{newtxtt}
\usepackage[T1]{fontenc} 
\usepackage[utf8]{inputenc} 
\usepackage{amsthm}

\usepackage{amsmath}
\usepackage{amssymb}
\usepackage{pgfplots} 
\pgfplotsset{compat=newest} 
\usepgfplotslibrary{fillbetween,colorbrewer}
\usepackage{amsfonts}
\usepackage{amsaddr}
\usepackage{xspace,subcaption}
\usepackage{enumitem} 
\usepackage{csquotes}
\usepackage{xcolor}
\usepackage[colorlinks]{hyperref}
\hypersetup{colorlinks, breaklinks, citecolor=teal,filecolor=black}
\usepackage[nosort,capitalise]{cleveref}
\crefname{enumi}{}{}
\crefname{equation}{}{}
\usepackage{graphicx}
\usepackage{mathtools}
\mathtoolsset{centercolon}
\usepackage{bm}
\usepackage{pgfkeys}
\usepackage{pgf,interval}
\intervalconfig{soft open fences}

%% file: notations2.tex
\DeclareMathOperator{\Tr}{Tr}

\DeclareMathOperator{\sgn}{sgn}

\DeclareMathOperator{\E}{\mathbf{E}}
\DeclareMathOperator{\Prob}{\mathbf{P}}

\DeclareMathOperator{\Spec}{Spec}

\newcommand{\ov}{\overline}
\newcommand{\ii}{\mathrm{i}}

\newcommand{\F}{\mathbf{F}}

\ifdefined\C
\renewcommand{\C}{\mathbf{C}}
\else
\newcommand{\C}{\mathbf{C}}
\fi

\renewcommand{\SS}{\mathcal{S}}

\newcommand{\un}{\underline}
\newcommand{\vx}{\bm{x}}

\newcommand{\vy}{\bm{y}}

\newcommand{\wt}{\widetilde}

\newcommand{\R}{\mathbf{R}}
\newcommand{\N}{\mathbf{N}}
\newcommand{\Z}{\mathbf{Z}}

\newcommand{\DD}{\mathbf{D}}

\newcommand{\cO}{\mathcal{O}}
\newcommand{\co}{{\scriptstyle\mathcal{O}}}
\newcommand{\cB}{\mathcal{B}}

\newcommand{\dif}{\operatorname{d}\!{}}
\DeclarePairedDelimiter{\braket}{\langle}{\rangle}%
\DeclarePairedDelimiter{\abs}{\lvert}{\rvert}%
\DeclarePairedDelimiter{\norm}{\lVert}{\rVert}%
\providecommand\given{}
\newcommand\SetSymbol[1][]{\nonscript\:#1\vert\allowbreak\nonscript\:\mathopen{}}
\DeclarePairedDelimiterX{\tuple}[1](){\renewcommand\given{\SetSymbol[\delimsize]}#1}
\DeclarePairedDelimiterX{\set}[1]\{\}{\renewcommand\given{\SetSymbol[\delimsize]}#1}
\DeclarePairedDelimiterXPP{\landauO}[1]{\cO}(){}{#1}
\DeclarePairedDelimiterXPP{\landauo}[1]{\co}(){}{#1}
\DeclarePairedDelimiterXPP{\landauOprec}[1]{\cO_\prec}(){}{#1}

%% file: bibliography_setup.tex
\usepackage[giveninits=true,url=false,doi=false,isbn=false,eprint=true,datamodel=mrnumber,sorting=nty,maxcitenames=4,maxbibnames=99,backref=false,block=space,backend=biber,bibstyle=phys]{biblatex} 
\AtEveryBibitem{\clearfield{month}}
\AtEveryCitekey{\clearfield{month}} 
\renewbibmacro{in:}{}
\DeclareFieldFormat[article]{title}{\emph{#1}} 
\DeclareFieldFormat{mrnumber}{\ifhyperref{\href{http://www.ams.org/mathscinet-getitem?mr=#1}{\nolinkurl{MR#1}}}{\nolinkurl{#1}}}
\DeclareFieldFormat{pmid}{\ifhyperref{\href{https://www.ncbi.nlm.nih.gov/pubmed/#1}{\nolinkurl{PMID#1}}}{\nolinkurl{#1}}}
\DeclareFieldFormat{eprint}{\ifhyperref{\href{https://arxiv.org/abs/#1}{\nolinkurl{arXiv:#1}}}{\nolinkurl{#1}}}
\renewbibmacro*{doi+eprint+url}{%
  \iftoggle{bbx:doi}{\printfield{doi}}{}
  \newunit\newblock%
  \printfield{mrnumber}%
  \newunit\newblock%
  \printfield{pmid}%
  \newunit\newblock%
  \printfield{eprint}%
  \iftoggle{bbx:url}{\usebibmacro{url+urldate}}{}}